\documentclass[11pt,reqno]{amsart}
\title[FQuad: Abel-Jacobi map and deformation]
{Moduli spaces of quadratic differentials: Abel-Jacobi map and deformation}

\author{Yu Qiu}
\address{Qy: Yau Mathematical Sciences Center and Department of Mathematical Sciences, Tsinghua University, 100084 Beijing, China.
    \&
Beijing Institute of Mathematical Sciences and Applications, Yanqi Lake, Beijing, China}
\email{yu.qiu@bath.edu}

\usepackage[a4paper]{geometry}
\usepackage{float}
\usepackage[usenames,dvipsnames]{xcolor}
\usepackage{subfigure}

\usepackage[colorlinks, linkcolor=blue,anchorcolor=Periwinkle,
   citecolor=red,urlcolor=Emerald]{hyperref}
\usepackage{amsmath}           % AMS 的主包
\usepackage{amssymb}           % AMS 符号包，会自动加载 amsfonts
\usepackage{latexsym}          % 几个额外的特殊符号，包括 \Box
\usepackage{bookmark}
\usepackage{enumitem}
\usepackage{array}
\usepackage{graphicx}
\usepackage[bottom]{footmisc}
\usepackage{cleveref}
\usepackage{transparent}

\usepackage{tikz}
\usetikzlibrary{matrix}
\usepackage{url}

\usetikzlibrary{decorations.markings,cd}
\usetikzlibrary{arrows}
\usetikzlibrary{decorations.pathreplacing,decorations.pathmorphing,shadings,fadings,calc}
\usepackage{extarrows,stmaryrd}

\tikzset{->-/.style={decoration={  markings,  mark=at position #1 with
    {\arrow{>}}},postaction={decorate}}}
\tikzset{-<-/.style={decoration={  markings,  mark=at position #1 with
    {\arrow{<}}},postaction={decorate}}}
\tikzcdset{arrow style=tikz, diagrams={>=stealth}}

\newcolumntype{L}{>{$}l<{$}} % math-mode version of "l" column type

\usepackage{tikz}
\usetikzlibrary{matrix, calc, arrows,cd,
decorations.pathreplacing, decorations.markings, decorations.pathmorphing,
backgrounds,fit,positioning,shapes.symbols,chains,shadings,fadings}

\usepackage{mathabx}
\let\Sun\relax

\let\Aries\relax

\usepackage{marvosym}

\theoremstyle{plain}
\newtheorem{theorem}{Theorem}[section]
\newtheorem{lemma}[theorem]{Lemma}
\newtheorem{corollary}[theorem]{Corollary}
\newtheorem{proposition}[theorem]{Proposition}

\newtheorem{convention}[theorem]{Convention}

\theoremstyle{definition}
\newtheorem{definition}[theorem]{Definition}
\newtheorem{example}[theorem]{Example}
\newtheorem{remark}[theorem]{Remark}

\numberwithin{equation}{section}
\newcommand{\Note}[1]{\textcolor{red}{#1}}

\newcommand\hua{\mathcal}
\newcommand\NN{\mathbb{N}}
\newcommand\ZZ{\mathbb{Z}}

\newcommand\RR{\mathbb{R}}
\newcommand\CC{\mathbb{C}}
\newcommand{\mai}{\mathbf{i}} % imaginary number i

\newcommand\<{\langle}
\renewcommand\>{\rangle}

\newcommand{\textfrac}[2]{\textstyle{\frac{#1}{#2}}}
\renewcommand{\setminus}{\smallsetminus}
\renewcommand{\emptyset}{\varnothing}
\newcommand{\isom}{\cong}

\newcommand{\rank}{\operatorname{rank}}

\newcommand\Sim{\operatorname{Sim}}
\newcommand\Hom{\operatorname{Hom}}

\newcommand\Ext{\operatorname{Ext}}

\newcommand{\Int}{\operatorname{Int}}

\newcommand\Br{\operatorname{Br}} % braid group
\newcommand\CT{\operatorname{CT}} % cluster braid group
\newcommand\Grot{\operatorname{K}} % Grothendieck group

\newcommand{\h}{\operatorname{\hua{H}}} %heart
\newcommand{\D}{\operatorname{\hua{D}}} % triangulated category
\def\Dsop{\D_3(\sop)} % triangulated category
\newcommand{\per}{\operatorname{per}} % perfect derived category

\newcommand\Aut{\operatorname{Aut}} % automorphism group
\newcommand\Autp{\Aut^\circ} % effectively acting automorphism group in principal comp
\newcommand\Stab{\operatorname{Stab}} % space of stability conditions
\newcommand\Stap{\Stab^\circ} % principal component of Stab

\newcommand{\EG}{\operatorname{EG}} % oriented exchange graph of hearts
\newcommand{\FG}{\operatorname{FG}} % oriented exchange graph of hearts
\newcommand{\FGT}{\FG^\T} % oriented exchange graph of hearts
\newcommand{\EGp}{\EG^\circ}       % principal component of EG
\newcommand{\EGT}{\EG^{\T}}
\newcommand{\uEG}{\underline{\EG}} % unoriented exchange graph of hearts
\newcommand{\CEG}{\operatorname{CEG}} % oriented cluster exchange graph
\newcommand{\uCEG}{\underline{\CEG}} % unoriented cluster exchange graph

\newcommand{\fg}{\operatorname{\hua{FG}}} % oriented exchange graph of hearts
\newcommand{\fgt}{\fg^{\T}} % oriented exchange graph of hearts

\newcommand\eg{\operatorname{\hua{EG}}}

\newcommand\ueg{\underline{\eg}}

\renewcommand{\k}{\mathbf{k}}
\newcommand{\Ho}[1]{\operatorname{\bf H}_{#1}}
\newcommand{\tilt}[3]{{#1}^{#2}_{#3}}
\newcommand{\Cone}{\operatorname{Cone}}
\newcommand\Sph{\operatorname{Sph}}

\newcommand{\numarc}{n}
\newcommand{\numtri}{w}

\newcommand{\Tri}{\Delta}

\newcommand\Diff{\operatorname{Diff}} % diffeomorphism group
\newcommand{\MCG}{\operatorname{MCG}} % mapping class group
\newcommand{\BT}{\operatorname{BT}}  %braid twist group
\newcommand{\MT}{\operatorname{MT}}  %mixed twist group
\newcommand{\FT}{\operatorname{FT}}  %flip twist group

\newcommand\Bt[1]{\operatorname{B}_{#1}}

\newcommand\M{\mathbf{M}} % marked points
\renewcommand\P{\mathbf{P}} % punctures

\newcommand{\Quad}{\operatorname{Quad}}
\newcommand{\FQuad}{\operatorname{FQuad}}

\newcommand{\conf}{\operatorname{conf}}

\newcommand{\CA}{\operatorname{CA}}

\newcommand{\SBr}{\operatorname{SBr}}

\newcommand{\UHP}{\mathbf{H}_\circ} % upper half plane
\newcommand{\uhp}{\mathbf{H}} % upper half plane
\newcommand{\skel}{\wp} % skeleton embedding
\newcommand{\cub}{\operatorname{U}} % cell
\newcommand{\cubc}{\cub^\circ} % cell

\newcommand\rs{\mathbf{X}} %compact version of S
\newcommand\surp{\rs^\circ}
\newcommand\dd{d} % exterior derivative

\newcommand\conj{\operatorname{ad}}

\newcommand{\sli}{\hua{P}}

\newcommand{\twi}{\Phi} % spherical twist

\newcommand{\on}[1]{\operatorname{#1}}
\newcommand{\wh}[1]{\widetilde{#1}}
\newcommand{\double}[2]{\overline{#1}^{#2}}

\usepackage{pifont}%\Ding
\def\punctures{\Sun}
\def\sun{\text{\Sun}}
\def\vortex{\text{\Yingyang}}
\def\dpole
    {node[white]{\tiny{$\blacksquare$}} node[]{\scriptsize{\punctures}}}

\def\Vpole
    {node[white]{\tiny{$\blacksquare$}} node[]{\scriptsize{\Yingyang}}}

\def\nn{node{$\bullet$}}
\def\ww{node[white]{$\bullet$}node[red]{$\circ$}}

\def\uk{\mathbf{k}}

\def\ZZv{ \ZZ_2^{{\oplus \vortex}} }

\newcommand\surf{\mathbf{S}}  % marked surface
\newcommand\surfp{\surf^{\sun}}  % marked surface
\newcommand\sop{\surfo^{\sun}}  % decorated marked surface
\newcommand\surfo{\surf_\Tri}  % decorated marked surface

\newcommand\RT{\mathrm{T}} % ordinary triangulation
\newcommand\T{\mathbb{T}} % decorated triangulation

\def\be{\begin{equation}}
\def\ee{\end{equation}}

\def\eveng{2\ZZ^{+}_{\le g}}
\newcommand\iv[1]{\underline{#1}}
\newcommand\Co{\operatorname{Co}}

\def\qlan{cyan!50}  %qian lan
\def\hhd{orange}  % closed loop color-hh dual
\def\LA{\on{LA}}

\def\AJ{\on{AJ}}

\def\REG{\EG_{\circledR} }
\def\FGC{\FG_{\text{\tiny{$\copyright$}}} }
\def\REGp{\REG^\circ}
\def\FGCp{\FGC^{\circ} }
\def\hh1{\REGp[\h,\h[1]]}
\def\hhT1{\REGp[\h_\T,\h_\T[1]]}

\newcommand\szb{\mathbf{S}^{\sun}_\Tri}  % decorated marked surface
\def\yue{\leftmoon}
\def\kui{ {\rightmoon} } %\sun\setminus\yue
\newcommand\MTsoy{\MTrel{\yue}}  % decorated marked surface
\newcommand\MTsot{\MTrel{\yue(\T)}}  % decorated marked surface

\newcommand{\MTrel}[1]{\MT(\sop,{#1})}

\newcommand\surfy{\surf^{\yue}}  % decorated marked surface

\def\wt{\mathbf{w}}
\def\wty{\wt_{\on{fin}}}
\def\wtw{\wt_{\on{inf}}}
\def\rk{\on{rk}}
\def\lp{\zeta}
\def\Xto{\on{X}^\sun}

\newcommand{\dbloop}[1]{\overline{#1}^{\circlearrowleft}}

\newcommand{\MTT}{\MT^\sun}  %mixed twist group T
\def\man{i}

\newcommand\sox{\surfo^\mix}  % marked surface new
\newcommand\sov{\surfo^{\vortex}}  % marked surface new
\newcommand\surfx{\surf^\mix}  % marked surface new
\newcommand\surfk{\surf^{\yue,\kui}}  % marked surface new
\newcommand\surfv{\surf^{\vortex}}  % marked surface new

\def\Dsox{\D_3(\surfx)} % triangulated category
\def\Dsov{\D_3(\surfv)} % triangulated category
\newcommand{\FQuads}{\FQuad^\pm}

\def\Qp{Q_\RT^{\sun}}
\def\Wp{W_\RT^{\sun}}
\def\Qx{Q_{\xT}^{\mix}}
\def\Wx{W_{\xT}^{\mix}}
\def\hx{\h_{\xT}^{\mix}}
\def\Gx{\Gamma_{\xT}^{\mix}}
\def\mix{{\yue,\vortex}}

\def\Qv{Q_\RT^{\vortex}}
\def\Wv{W_\RT^{\vortex}}

\def\LLk{\on{L}^2(\kui)}
\def\cok{\on{CO}(\kui)}

\def\sok{\surfo^{\yue,\kui}}
\def\MTsoyk{\MT(\soyk)}
\def\soyk{\surfo^{\yue/\kui^2}}

\newcommand\xT{\RT^\times} % ordinary triangulation
\usetikzlibrary{knots}

\def\DorE{\text{\Aries}}
\def\RootS{\Omega}
\def\ELam{\wh{\RootS}}

\begin{document}
%=========================================================
\begin{abstract} %Notation free version:
We study the moduli space of framed quadratic differentials with prescribed singularities parameterized by a decorated marked surface with punctures (DMSp), where simple zeros, double poles and higher order poles respectively correspond to decorations, punctures and boundary components. We show that the fundamental group of this space equals the kernel of the Abel-Jacobi (AJ) map from the surface braid group of DMSp to the first homology group of the marked surface (without decorations/punctures). Moreover, a universal cover of this space is given by the space of stability conditions on the associated 3-Calabi-Yau category.

Furthermore, when we partially compactify and orbifold this moduli space by allowing the collision of simple zeros and some of the double poles, the resulting moduli space is isomorphic to a quotient of the space of stability conditions on the deformed (with respect to those collidable double poles) 3-Calabi-Yau category.

Finally, we show that the fundamental group of this partially compactified orbifold equals the quotient group of the kernel of the AJ map by the square of any point-pushing diffeomorphism around any collidable double pole.
This construction can produce any non-exceptional spherical/Euclidean Artin braid groups.


\end{abstract}
\maketitle

\tableofcontents \addtocontents{toc}{\setcounter{tocdepth}{1}}
\setlength\parindent{0pt}
\setlength{\parskip}{5pt}

%=========================================================
%=========================================================
\section{Introduction}
%=========================================================
\subsection{Topology of moduli spaces of quadratic differentials}\
%=========================================================

The moduli spaces of abelian/quadratic differentials are classical objects of study
in dynamical systems, which also naturally interact with topology, algebraic geometry
and, recently, representation theory of algebras.
For instance, the topology of moduli spaces has attracted a lot of attention.

%=========================================================
\paragraph{\textbf{Connected components}}\
Kontsevich-Zorich \cite{KZ} classified the connected components of the space of holomorphic abelian differentials;
later Lanneau \cite{L} and Boissy \cite{Boi} treated the case of meromorphic quadratic differentials;
more recently, Chen-Gendron \cite{CG} extended this the case of $k$-differentials.
In most of these cases, the components are classified by spin invariants, hyper-ellipticity and torsion conditions.

%=========================================================
\paragraph{\textbf{Fundamental groups}}\
The next interesting and challenging question is the calculation of (orbifold) fundamental groups of components of these spaces.
Until now, few cases have been known: see \cite{W,Ham,CS2} for the studies on certain quotient groups of the fundamental groups of such types, \cite{FM} and \cite{LM} for the genus 2 and 3 cases respectively.

In the previous work with King \cite{KQ2},
we calculated the connected components and fundamental groups of moduli spaces of quadratic differentials
with only simple zeros and higher order poles,
in which case, the prescribed singularities are encoded by a decorated marked surface (DMS) $\surfo$.
One of the key ingredients is the introduction of the braid twist group $\BT(\surfo)\lhd\MCG(\surfo)$ of $\surfo$ in \cite{QQ}, generated by the braid twists along any closed arcs).
Such a group is a (usually proper) normal subgroup of the surface braid group $\SBr(\surfo)$.
The motivation came from the study of the spherical twist group of the Calabi-Yau-3 category associated to $\surfo$, where these two (braid and spherical) twist groups can be naturally identified (\cite[Thm.~1]{QQ}).

In this work, we investigate the case with simple zeros, higher order poles and also double poles.
The first key observation is that the braid twist group is in fact,
the kernel of the Abel-Jacobi map.
In the finite area case (no double or higher order poles), this map was considered in \cite{W}.

%=========================================================
\paragraph{\textbf{Topological Abel-Jacobi maps}}\

We fix a Riemann surface $\rs$ and consider the quadratic differential $\phi$ of
prescribed singularity type
\[
    (\wty=\wt_{\ge-1},\wtw=\wt_{\le-2}),
\]
where $\wty$ and $\wtw$ are the collections of the finite/infinite singularities, respectively,
represented by their orders.
The numerical datum is encoded by a weighted DMS with punctures,
$\sop$\footnote{\footnotesize{The symbol $\sun$ is pronounced \emph{sun}.}}
which is the real blow-up of $\rs$ with respect to $\phi$ (as a topological surface).
The weighted decorations in $\Tri$ are finite singularities (=zeros and simple poles) in $\wty$;
the punctures in $\sun$ are order two poles in $\wt_{-2}$;
the marked points $\M\subset\partial\surf$ on the boundary components correspond to
(the distinguish directions at) higher order poles in $\wt_{\le-3}$.
We will mainly focus on the moduli space $\FQuad(\surfp)$ of $\surfp$-framed quadratic differentials.
It is a manifold and a $\MCG(\surfp)$-covering of the (unframed) moduli space $\Quad(\surfp)$.

The monodromy map naturally maps each loops in $\FQuad(\surfp)$ to
a loop in the configuration space of $\Tri$ of weighted points on $\rs$ with special marking $\wtw$.
The classical Abel-Jacobi map $\AJ^{\wt_{\le-3}}$ maps
special marking $\wt_{\le-3}(\subset\wtw)$, regarded as a divisor of $\rs$,
to a line bundle in the Picard variety $\on{Pic}(\rs,\wt_{\le-3})$.
Note that we forget about the punctures in $\sun=\wt_{-2}$ here.
Then we obtain the following map:
\begin{gather}\label{eq:ses0}
%    \FQuad(\surfp)
%        \xrightarrow{}
    \AJ^{\wt_{\le-3}} \colon  \on{conf}_{\Tri}(\rs,\wtw) \to \on{Pic}(\rs,\wt_{\le-3}).
\end{gather}
Taking $\pi_1$ of these spaces and combining with the monodromy map, we obtain a sequence of maps:
\begin{gather}\label{eq:ses0+}
    \pi_1\FQuad(\surfp)
        \xrightarrow{\Xi} \SBr(\surfp_{\Tri})
        \xrightarrow{ \AJ } \Ho{1}(\surf),
\end{gather}
where the second term is the surface braid group and the last term is the first homology.
We call the map $\AJ$ induced from $\AJ^{\wtw}$ the \emph{(topological) Abel-Jacobi map}.

We will focus on the simple zero case, i.e. $\wty=\wt_1$.
We will also consider the Teichm\"{u}ller framed (i.e. $\sop$-framed) version $\FQuad(\sop)$,
which is the $\MCG(\sop)$-covering of $\Quad(\surfp)$.
The following is the main result of the paper (\Cref{thm:main}).

\begin{theorem}\label{thm:1}
Let $\sop$ be a DMS with punctures. There is a short exact sequence
\begin{gather}\label{eq:ses1}
    1\to \pi_1\FQuad(\surfp)\xrightarrow{\Xi} \SBr(\sop)\xrightarrow{ \AJ } \Ho{1}(\surf)\to1
\end{gather}
Moreover, we have
\begin{equation}\label{eq:main}
  \begin{cases}
     \pi_1\FQuad^{\T}(\sop)=1,\\
     \pi_0\FQuad(\sop)\cong\Ho{1}(\surf),
  \end{cases}
\end{equation}
where $\FQuad^{\T}(\sop)$ is any connected component of $\FQuad(\sop)$.
\end{theorem}

When $\sun=\emptyset$, the theorem is a slight improvement of the main result in \cite{KQ2}.

In the works \cite{W,Ham,CS2}, the authors studied maps similar to \eqref{eq:ses0+}
and described the image of the corresponding $\Xi$.
In our case, the most challenging part is the injectivity of $\Xi$.
Notice that in general, injectivity does not holds,
cf. \cite[Rem.~12.4]{Ham} and \cite[Thm.~A]{G}.

%=========================================================
\paragraph{\textbf{$K(\pi,1)$-conjecture and categorification}}\

In \cite{KZ0},
Kontsevich-Zorich conjectured that `each connected component of (a stratum of the) moduli space of quadratic differentials has homotopy type $K(\pi,1)$ for $\pi$ which is a group commensurable with some mapping class group'.
At least in our case (i.e. all finite singularities are simple zeros), we believe that is true.
In fact, it is closely related to the contractibility conjecture for the spaces of stability conditions.

The notion of stability condition on triangulated categories was introduced by Bridgeland \cite{B1},
with motivations coming from $\Pi$-stability in string theory.
By the work of \cite{HKK,CHQ},
any connected component of a moduli space of quadratic differentials with prescribed singularity type $\surfp_{\wty}$
can be realized by the space of stability conditions on some triangulated category $\D$:
\[  \FQuad^{\T}(\surfp_{\wty}) \cong \Stap(\D)    \]
Such a category $\D$ is not unique and can be chosen to be Calabi-Yau.
One should regard that $\Stab\D$ provides a categorification of the moduli space;
other topological operations (e.g. collision and collapsing subsurface)
may also be understood by some categorical operations (e.g. taking Verdier quotient in \cite{BMQS}).
In particular, we will use the 3-Calabi-Yau categories $\Dsop$ constructed in \cite{CHQ},
to prove the main result above.
More precisely, we show the following (\Cref{thm:stab}).

\begin{theorem}\label{thm:2}
$\Stap(\Dsop)$ is simply connected and hence is a universal cover of $\FQuad(\surfp)$. %$\FQuad(\surfp)$.
\end{theorem}

%=========================================================
\paragraph{\textbf{Donaldson-Thomas (DT) invariants and quantum dilogarithm identities}}\

Motivated by classical curve-counting problems,
Thomas \cite{T} proposed the so-called DT-invariants to count stable sheaves on Calabi–Yau 3-folds.
Kontsevich–Soibelman \cite{KS1,KS2} generalized their work and formalized a way of counting semistable objects
in 3-Calabi–Yau triangulated ($A_\infty$-)categories (or counting BPS states in physics terminology),
using the so-called motivic DT-invariants.
The quantum version of DT-invariants can be calculated via a product of quantum dilogarithms cf. \cite{K2,R,KS1,KS2,KW}.
Recently Haiden \cite{Hai} and Kidwai-Williams \cite{KW} gave formulae for calculating the quantum DT-invariants in the surface case by counting saddle trajectories and ring domains

The point is that one can think of the DT-invariants of abelian categories
(or hearts of some triangulated category) as a function on the spaces of stability conditions.
It is constant in a fixed cell and changes values when crossing a wall (of the first kind).
The fact that the product is an invariant is due to quantum dilogarithm identities (QDI).
In many cases, the QDI reduce to the pentagon identity (and the trivial commutation relation or square identity),
(e.g. \cite{Q1} for the Dynkin case, cf. \cite{Na1,Na2}).
In the marked surface case, such a reducibility (to square/pentagon identity) is equivalent to
the fact that the fundamental group of the unoriented flip exchange graph $\uEG(\surfp)$ of triangulations
is generated by squares and pentagons,
cf. \Cref{thm:H} for the classical result, \cite[Thm.~10.2]{FST} for the cluster case and
our slight generalization \Cref{pp:FST}.
Such results interpreted as local simply connectedness.

%=========================================================
\subsection{Techniques}\
%=========================================================

%=========================================================
\paragraph{\textbf{Mixed twist group}}(\Cref{sec:twist})\

The surface braid group $\SBr(\sop)$
was first introduced by Zariski (naturally generalizing Artin’s geometric definition)
and was rediscovered by Fox (cf. the survey \cite[\S~2]{GJP} and the references).
There are two types of standard generators:
the braid twist and the L-twists (point-pushing diffeomorphisms)
see \Cref{fig:T} or \Cref{fig:2loop} for the monodromy version in the moduli spaces.

\begin{figure}[bth]\centering
\makebox[\textwidth][c]{
 \includegraphics[width=7cm]{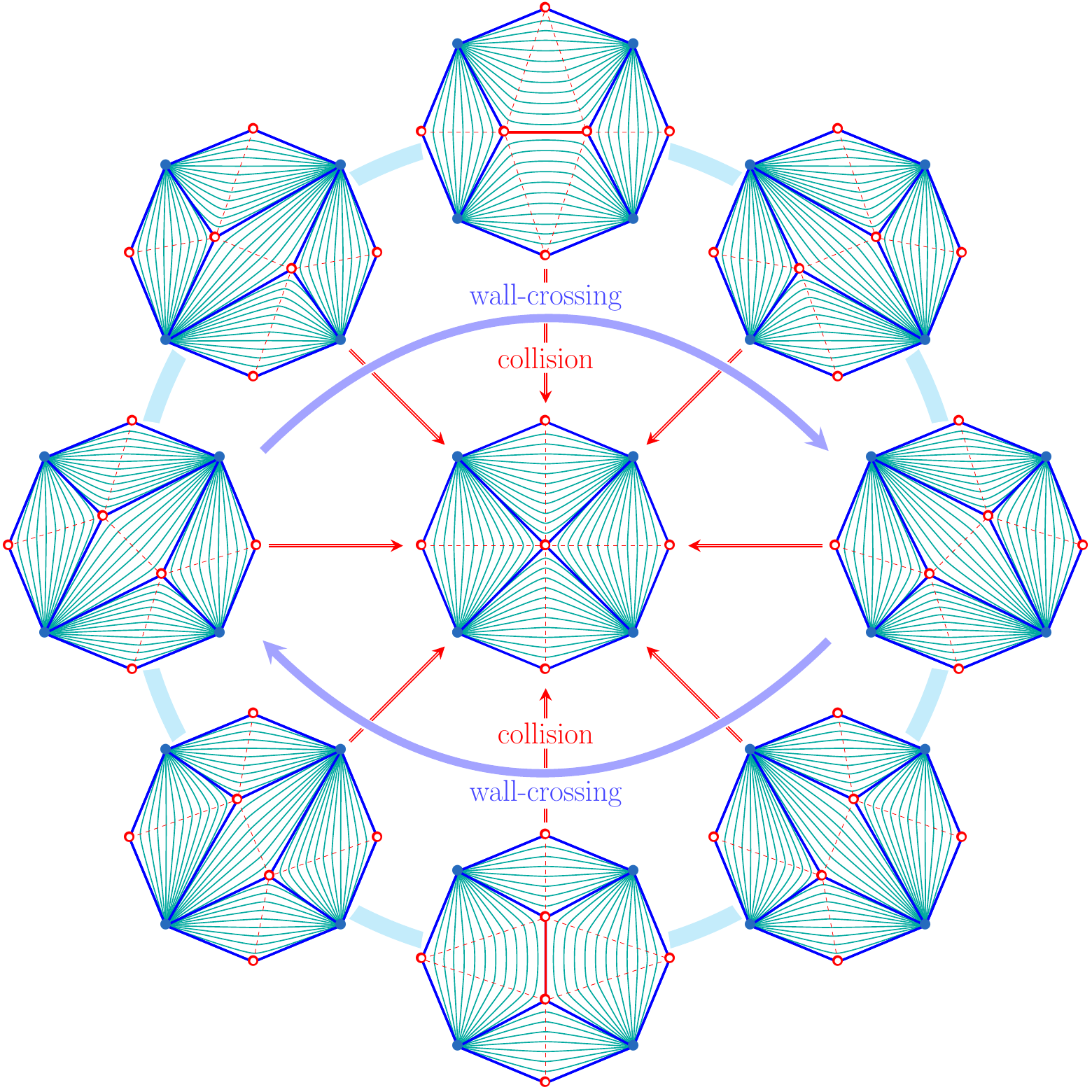}\;
  \includegraphics[width=7cm]{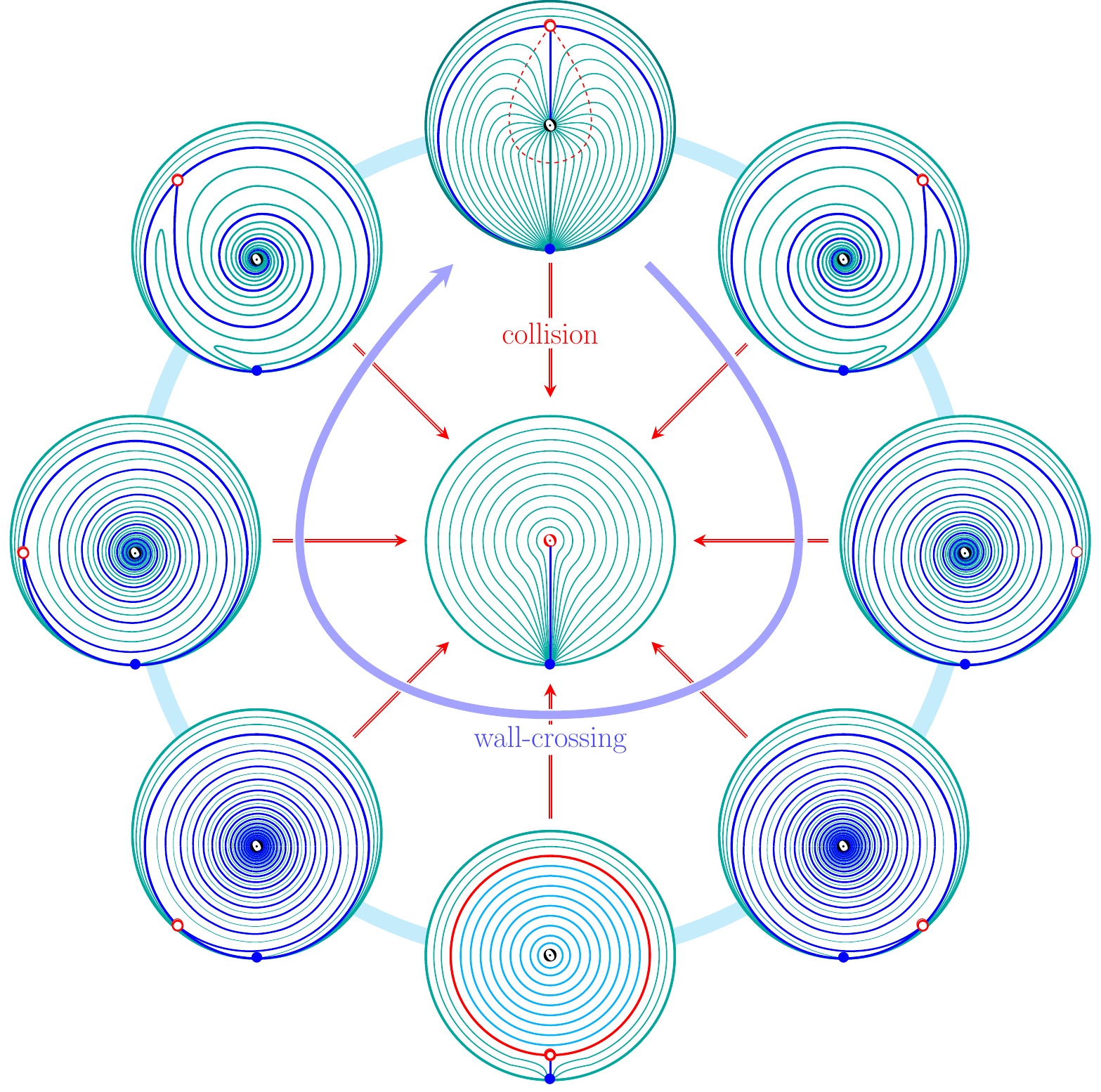}\;
}\caption{Loops in moduli spaces as braid twist/L-twist}\label{fig:2loop}
\end{figure}

The first step of proving \Cref{thm:1} is to understand the kernel of the Abel-Jacobi map in \eqref{eq:ses1}.
For this purpose, we introduce the \emph{mixed twist groups} $\MTsoy$,
which is the subgroup of $\SBr(\sop)$ generated by all braid twists and
the L-twists that wind a puncture in a chosen subset\footnote{\footnotesize{The symbol $\yue/\kui$ are pronounced \emph{left/right moon} respectively.}}
$\yue$ of $\sun$. Denote by $\kui=\sun\setminus\yue$.

We prove that (\Cref{thm:QZ+})
%\begin{theorem}
the mixed twist group $\MTsoy$ is the kernel of the (relative) Abel-Jacobi map $\AJ^{\sun}_{\yue}$:
\begin{equation}\label{eq:MT}
  \MTsoy=\ker \big(\SBr(\sop)\xrightarrow{\AJ^{\sun}_{\yue}}\Ho{1}(\surf^{\kui})\big),
\end{equation}
which admits a finite presentation in \Cref{thm:MT's}.
%\end{theorem}

Here, $\Ho{1}(\surf^{\kui})=\Ho{1}(\surf)\oplus\Ho{1}(\kui)$ for $\Ho{1}(\kui)=\ZZ^{\oplus\kui}$.
In particular, the biggest one is
\begin{equation}\label{eq:MT=pi1}
    \MT(\sop)=\MTrel{\sun}=\ker\AJ=\pi_1\FQuad(\surfp)
\end{equation}
for $\AJ=\AJ^{\sun}_{\sun}$ that appears in \eqref{eq:main},
and the smallest mixed twist group is the braid twist group
$\BT(\sop)=\MTrel{\emptyset}=\ker \AJ^{\sun}_\emptyset$.
The inbetweeners will appear when we calculating the fundamental groups of partial compactifications of $\FQuad(\surfp)$ later.

%=========================================================
\paragraph{\textbf{Flip groupoids and flip twist groups}}(\Cref{sec:cat,sec:fgoupoids})\

The other ingredient of proving \Cref{thm:1} is the study of fundamental groups of oriented flip graphs.
As mentioned in the DT-invariant paragraph, QDI is equivalent to the fact that
$\pi_1\uEG(\surfp)$ is generated by unoriented squares and pentagons.
We need to analyze the fundamental group $\pi_1\FG(\sop)$ of the flip graph of decorated triangulations of $\sop$.
There are four types of basic loops:
\begin{itemize}
  \item The squares/pentagons lifted from $\pi_1\uEG(\surfp)$, see \Cref{fig:456}.
  \item The short dumbbell relations introduced in \cite{KQ2}, see \eqref{eq:fat DB} and cf. \Cref{fig:456}.
  \item The long dumbbell relations, see \Cref{fig:Hex} and \Cref{fig:C-1}.
  \item The symmetric hexagons, see \Cref{fig:Hex} and \Cref{fig:C-2}.
\end{itemize}
When imposing these relations/faces in the oriented version $\FG(\surfp)$ of the flip graph for $\surfp$, e.g. the purple graph in the right picture of \Cref{fig:op-tri},
we obtain the \emph{flip groupoid} $\fg(\surfp)$, as a generalization of cluster exchange groupoid introduced in \cite{KQ2}.
The subgroup of the point group at a triangulation $\RT$ of $\pi_1\fg(\surfp)$,
generated by local flip twists, is the \emph{flip twist group} $\FT(\RT)$, cf. \Cref{def:LT}.

We show that there is a natural isomorphism between
a flip twist group and the corresponding mixed twist group (\Cref{thm:FT}):
\begin{equation}
    \FT(\RT)\cong\MTsot,
\end{equation}
where $\RT$ is the triangulation of $\surfp$ induced from a triangulation $\T$ of $\sop$, forgetting about the decorations, and $\yue(\T)$ is the set of $\T$-isolated punctures (which are in self-folded $\T$-triangles).

A highlight of the proof is that the higher braid relation
$\eta\delta\eta\delta=\delta\eta\delta\eta$ decomposes into 2 long dumbbell hexagons and 2 symmetric hexagons
as shown in \Cref{fig:4hex}, which is very different from
the other decomposition (via folding in cluster theory, cf. \cite{HHQ})
into non-symmetric hexagons shown in \Cref{fig:decomposition}.

%=========================================================
\subsection{Partial compactification with orbifolding via deformation}\
%=========================================================

After we study the mixed/flip twist groups, there is a strong hint
of studying the various compactifications of $\FQuad(\surfp)$,
in particular their categorifications using spaces of stability conditions.
Recently, there are many works on the topic of compactifying spaces of stability conditions,
including: \cite{BDL,KKO} of Thurston type compactification;
\cite{Bol,BPPW} of partial compactification
and \cite{BMQS,BMS} towards multi-scale compactification in the sense of \cite{BCGGM}.
In \Cref{sec:MSx}, we give a result of a similar nature.

%=========================================================
\paragraph{\textbf{The other natural moduli space}}\

The moduli space $\FQuad(\surfp)$ and 3-Calabi-Yau category $\Dsop$ in \Cref{thm:2} are
the natural ones associated to $\surfp$ from the point of view of dynamical systems and of higher categories (i.e. perverse schobers, cf. \cite{CHQ}).
Also, recent work \cite{KW} suggests the same from the point of view of DT-theory (and mathematical physics).

However, they are \emph{unnatural} from the point of view of cluster theory and representation theory of algebras (cf. \cite{FST,BS}), at least for the first glance.
Indeed, associated to a triangulation $\RT$ of $\surfp$,
the quiver with potential $\Qp$ with 3-cycles potential $\Wp$ for $\Dsop$ is \emph{degenerate},
in the sense of \cite{DWZ}.
On the other hand, the unique \emph{non-degenerate} quiver with potential
$(\Qv,\Wv)$\footnote{\footnotesize{The symbol $\vortex$ is pronounced \emph{Yin-Yang},
meaning $\ZZ_2$-symmetry in some sense.}}
is the cluster favorite (cf. \cite{LF}).
Indeed, the mutations of which match the original quiver mutations in the sense of \cite{FZ} and
the Jacobian algebras are of skew-gentle type, generalizing type $D_n$ and $\wh{D_n}$.
Fomin-Shapiro-Thurston \cite{FST} introduced the extra $\ZZ_2$ symmetry according to cluster theory,
for each puncture (which we will call it a vortex).
The corresponding interpretation in the moduli spaces of quadratic differentials
is the inclusion of signs of residue at double poles and
allowing the collision of simple zeros and double poles, cf. \cite{GMN,BS}.

%=========================================================
\paragraph{\textbf{The comparison}}(App.~\ref{sec:MSx})\

In the main theorem
\begin{equation}\label{eq:BS}
    \Quad(\surfv)=\FQuad(\surfv)/\MCG(\surfv)\cong\Stap(\Dsov)/\Autp
\end{equation}
of \cite[Thm.~1.2]{BS} (comparing to the quotient version \eqref{eq:CHQ2+} of \Cref{thm:CHQ2}),
\begin{itemize}
  \item The moduli space $\Quad(\surfv)$ is a partial compactification (with orbifolding) of $\Quad(\sop)$.
  \item The 3-Calabi-Yau category $\Dsov$ is a deformation of $\Dsop$,
  constructed from $(\Qv,\Wv)$ that provides categorification of the corresponding cluster algebra).
\end{itemize}
We show that the deformation process of the 3-Calabi-Yau categories in \Cref{thm:2}
perfectly matches with the partial compactification (with orbifolding) process.

More precisely, we can choose a partition $\sun=\yue\cup\kui$,
call punctures in $\kui$ by vortices and use $\vortex$
for $\kui$ to indicate the extra $\ZZ_2$-symmetry associated.
Roughly speaking,
by subtracting the cycles around puncture in $\vortex$ from the potential $\Wp$
we obtain a deformation $\Dsox$ of $\Dsop$.

Correspondingly, we first allow the collision of simple zeros with double poles in $\kui$
and get a partial compactification $\FQuad(\surfk)$.
Then by orbifolding the extra strata we get an orbifold $\FQuad(\surfx)$.
Note that a tricky thing is that we are actually
partially compactify and orifold $\FQuad(\surfp)$ instead of the unframed $\Quad(\surfp)$
since we need to distinguish usual punctures and the collidable one.

We show the following (\Cref{thm:app}).% for the unframed version of $\FQuad(\surfx)$.

\begin{theorem}
There is an isomorphism $$\Quad(\surfx)\colon=\FQuad(\surfx)/\MCG(\surfk)\cong\Stap(\Dsox)/\Autp$$ between complex orbifolds,
which specializes to \eqref{eq:CHQ2+} in \cite{CHQ} for $(\yue,\vortex)=(\sun,\emptyset)$
and \eqref{eq:BS} in \cite{BS} for $(\yue,\vortex)=(\emptyset,\sun)$.
\end{theorem}

In \cite{BMQS}, collapsing surfaces include the special case of collision of singularities.
However, one uses quotient categories for the extra (lower) strata there
while we use deformed categories for the unions with the extra strata.

%=========================================================
\paragraph{\textbf{Fundamental groups of compactified spaces with/without orbifolding}}(\S~\ref{sec:xxx})\

We end the paper with the calculation of the fundamental groups of
the partial compactifications of $\FQuad(\surfp)$,
namely, $\FQuad(\surfk)$ (without orbifolding) and $\FQuad(\surfx)$ (with orbifolding),
cf. \Cref{thm:pi1-kui} and \Cref{thm:pi1-v}.

\begin{theorem}
The fundamental groups of compactified moduli spaces are
\begin{equation}\label{eq:pix}
\begin{cases}
    \pi_1\FQuad(\surfk)=\pi_1\FQuad(\surfp)/\mu(\kui)\\
    \pi_1\FQuad(\surfx)=\pi_1\FQuad(\surfp)/\mu^2(\kui).
\end{cases}
\end{equation}
where $\mu(\kui)$ consisting of loops of the form in the right picture of \Cref{fig:2loop}, where the puncture is in $\kui$
and $\mu^2(\kui)$ consisting of the square of the loops above.
\end{theorem}

Denote by $\sox$ the DMS with mixed punctures \& vortices (=DMSx).
The difference from $\sop$ to $\sox$ is that when considering associated symmetric groups,
one quotients by the (normal) subgroup $\LLk$,
generated by squares of L-twist along any L-arc enclosing precisely one puncture in $\kui$.
Then we have DMSx version of \eqref{eq:ses1},% and \eqref{eq:ses1+},
which is equivalent to the second equation in \eqref{eq:pix}:
\begin{equation}\label{eq:pix2}
    \pi_1\FQuad(\surfx)=\MT(\sox)=
    \ker \big(\SBr(\sox)\xrightarrow{\AJ}\Ho{1}(\surf)\big).
\end{equation}

A key ingredient of the seminal works \cite{FST,BS} is the $\ZZv$-symmetry mentioned above.
Our topological interpretation of which is
\[
    \Ho{1}(\vortex)\colon=\Ho{1}(\kui)\big/ \AJ^{\sun}_{\yue}(\LLk) \cong \ZZv.
\]
%after quotienting $\LLk$.
Hence such a symmetry is hidden naturally in \eqref{eq:pix2}.

%=========================================================\mathcal{}
\paragraph{\textbf{Euclidean Artin braid groups and hyperplane arrangements}}(App.~\ref{sec:ex})\

\def\fkh{\mathfrak{h}}
\def\fkhd{\mathfrak{h}_{\ELam}}
\def\reg{{\on{reg}}}
\def\hreg{\fkh^\reg}

We have a byproduct by examples:
for any non-exceptional Artin braid group $\Br(\DorE)$ of (spherical or) Euclidean type $\DorE$,
there is a DMSx $\surfx$ such that
\[
    \Br(\DorE)=\MTsoyk=
        \ker \big(\SBr(\sox)\xrightarrow{\AJ^\vortex}\Ho{1}(\surf^\vortex)\big).
\]
The mixed twist group $\MTsoyk$ here is the quotient of $\MT(\sox)$ in \eqref{eq:pix2} by $\ZZv$.

Moreover, the corresponding hyperplane arrangements $\fkh^\reg_\DorE/W_\DorE$, quotient by the Coxeter/Weyl group $W_\DorE$,
can be realized by
\[
    \FQuad^\pm(\surfx)=\Stap(\D_3(\surfx))/\Br(\DorE),
\]
the moduli space of signed quadratic differentials on $\surfx$ or quotient of stability space on the associated 3-Calabi-Yau category (cf. \Cref{thm:PS+}).
This provides another justification for orbifolding in partial compactification.
Such a result is also the 3-Calabi-Yau analogue of \cite{B2} for Kleinian singularities.

%=========================================================
\subsection*{Convention}
%=========================================================
\begin{itemize}
\item \textbf{Composition:} we will write $a\circ b$ as $ab$.
\item \textbf{Inverse} we will write $s^{-1}$ as $\iv{s}$ and $\iv{st}=\iv{s}\cdot\iv{t}$.
\item \textbf{Conjugation:} we will write $\iv{b}ab$ as $a^b$ and then $\iv{(a^b)}=\iv{a}^b$.
\item \textbf{Relations:} we will use the higher braid relation of rank $m$:
\[
    \Br^m(a,b)\colon \underbrace{aba\cdots}_{m}=\underbrace{bab\cdots}_{m}.
\]
In particular,
we will often use commutation relation $\Co=\Br^2$ and usual braid relation $\Br=\Br^3$.
And the only other higher rank braid relation we will consider in this paper is $\Br^4$.
\item There is a partition $\sun=\yue\cup\kui$ and
$\vortex$ is the same set of $\kui$ but equipped with $\ZZ_2$-symmetry.
\end{itemize}

%=========================================================
\subsection*{Acknowledgments}
%=========================================================
I would like to thank Huang Jingyin, Han Zhe, Alastair King, Bernhard Keller, Martin M\"oller, Ivan Smith and Zhou Yu for inspiring conversations.
Also thanks to Anna Barbieri, Dawei Chen, Yitwah Cheung, He Ping and Nicholas Williams for useful comments.

This work is supported by National Natural Science Foundation of China (Grant No. 12425104).

%Ethical Statement: The authors declare that this work is original, has not been published elsewhere, is not under consideration by another journal, and complies with all ethical responsibilities of authors as required by this journal.

%No new data were created or analyzed in this study. Data sharing is not applicable to this article.

%=========================================================
%=========================================================
\section{Topological Abel-Jacobi maps and mixed twist groups}\label{sec:twist}
%=========================================================
%=========================================================

%=========================================================
\subsection{Decorated marked surfaces with punctures (DMSp)}\label{sec:DMSp}\
%=========================================================

Let $\sop=(\surf,\M,\sun,\Tri)$ be a \emph{decorated marked surface with punctures},
which is determined by the following data:
\begin{itemize}
  \item $\surf$ is a genus $g$ surface;
  \item A set $\M\ne\emptyset$ of \emph{marked points} in the boundary $\partial\surf$ of $\surf$;
  \item A set $\sun$ of \emph{punctures} in the interior $\surf^\circ$ of $\surf$;
  \item A set $\Tri$ of \emph{decorations} in $\surf^\circ$.
\end{itemize}
We require that each boundary component of $\partial\surf$ contains at least one marked point.
Let $m=|\M|$ and
\[
    \sun=\{P_1,\cdots,P_p\}, \quad \Tri=\{Z_1,\cdots,Z_\numtri\}
\]
for $p=|\sun|$, $\numtri=|\Tri|$.
In the figures, we draw marked points, punctures and decorations as black bullets, suns and red circles
(i.e. $\bullet$/$\sun$/\Note{$\circ$}), respectively.

We will denote by $\surf=(\surf,\M)$ the marked surface (without punctures and decorations),
$\surfo$ the decorated marked surface (DMS, forgetting about the punctures) and
$\surfp$ the marked surface with punctures (forgetting about the decorations).

Denote by $\MCG(\sop)$ the \emph{(mapping class) group} of isotopy classes of homeomorphisms of $\sop$
that preserves $\M$, $\sun$ and $\Tri$ setwise, respectively.
Similarly for other variations of $\MCG$ when forgetting about certain sets of points.

The DMSp $\sop$ can be regarded as a point in the (unordered) configuration space $\conf_\Tri(\surfp)$
of $\numtri=|\Tri|$ points in $\surfp$.
The \emph{surface braid group} is defined to be
\[
    \SBr(\sop)\colon=\pi_1\conf_\Tri(\surfp).
\]
Alternatively, consider the forgetful map $F^\Tri\colon \sop\to\surfp$ with induced map
\begin{equation}\label{eq:forget1}
    F^\Tri_*\colon\MCG(\sop)\to\MCG(\surfp).
\end{equation}
Then it is well-known that $\SBr(\sop)=\ker F_*^\Tri$
or there is a short exact sequence
\begin{gather}\label{eq:Sbr}
  1\to \SBr(\sop)\to \MCG(\sop)\xrightarrow{F^\Tri_*}\MCG(\surfp)\to1.
\end{gather}

For the later use, we introduce the following notions.
\begin{itemize}
  \item An \emph{open arc} is a curve connecting marked points or punctures,
  which will be usual drawn in blue/green colorway.
  \item A \emph{closed arc} is a curve connecting decorations,
  which will be usual drawn in red/orange colorway.
\end{itemize}
We say a closed arc is \emph{simple} if it has no self-intersection and its endpoints are different
and an L-arc is a closed arc with endpoints coincide and without self-intersection in $\surf^\circ$.
Denote by $\CA(\sop)$ and $\LA(\sop)$ the sets of simple closed arcs and of L-arcs in $\sop$, respectively.

Along any arc $\eta$ in $\CA(\sop)$, there is a \emph{braid twist} $B_\eta$ in $\SBr(\sop)$,
which only moves the two endpoints/decorations of $\eta$ anticlockwise within a neighbourhood of $\eta$,
as shown in the upper pictures of \Cref{fig:T}.
Along any arc $\lp$ in $\LA(\sop)$, there is an \emph{L-twist (or point-pushing diffeomorphisms)} $L_\lp$ in $\SBr(\sop)$,
which corresponds to drag the endpoint/decoration of $\lp$ along $\lp$ anticlockwise in $\pi_1\conf_\Tri(\surfp)$, as shown in the lower pictures of \Cref{fig:T}.

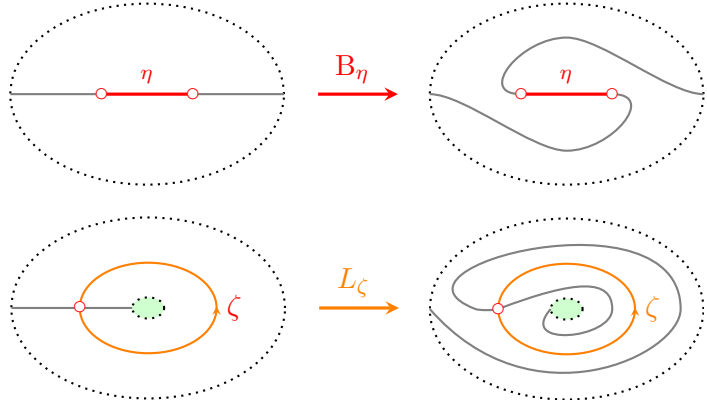
\begin{figure}[ht]\centering
\begin{tikzpicture}[scale=.3]
  \draw[dotted,thick](0,0)ellipse(6 and 4)node[above,red]{$_\eta$};
  \draw[red, very thick](-2,0)to(2,0);
  \draw[gray,thick](2,0)to(6,0)(-2,0)to(-6,0);
  \draw(-2,0)\ww(2,0)\ww;
  \draw[red,thick](0:7.5)edge[very thick,-stealth](0:11)
    (0:9)node[above]{$\Bt{\eta}$};
\end{tikzpicture}\quad
%=======================================================
\begin{tikzpicture}[scale=.3]
  \draw[dotted,thick](0,0)ellipse(6 and 4)node[above,red]{$_\eta$};
  \draw[red, very thick](-2,0)to(2,0);
  \draw[gray,thick](2,0).. controls +(0:2) and +(0:2) ..(0,-2.5)
    .. controls +(180:1.5) and +(0:1.5) ..(-6,0)
    (-2,0).. controls +(180:2) and +(180:2) ..(0,2.5)
    .. controls +(0:1.5) and +(180:1.5) ..(6,0);
  \draw(-2,0)\ww(2,0)\ww;
\end{tikzpicture}

	\begin{tikzpicture}[xscale=.3,yscale=.2]
	\draw[dotted,thick](0,0)circle(6);
    \draw[dotted,thick,fill=green!19](0,0)circle (.7);
	\draw[gray,thick](-6,0)to(-.7,0);
	\draw[\hhd,thick] (0,0) circle (3);
	\draw[\hhd,<-,>=stealth](3,0.1)to(3,-0.01);
	\draw(-3,0)\ww(3,0)node[right,red]{$\lp$};
	\draw[\hhd](0:7.5)edge[very thick,-stealth](0:11);
    \draw[\hhd,thick](0:9)node[above]{$L_{\lp}$};
	\end{tikzpicture}\quad
	%=======================================================
	\begin{tikzpicture}[xscale=.3,yscale=.2]\clip(-6,8)rectangle(6,-6);
	\draw[dotted,thick](0,0)circle(6);
    \draw[dotted,thick,fill=green!19](0,0)circle (.7);
	\draw[\hhd,thick] (0,0) circle (3);
	\draw[\hhd,<-,>=stealth](3,0.1)to(3,-0.01);
	\draw[gray,thick](-6,0)
    .. controls +(-60:9) and +(-90:3) ..(5,0)
	.. controls +(90:8) and +(150:5.5) ..(-4,0)
	.. controls +(-30:.2) and +(-150:.2) ..(-3,0)
	.. controls +(30:4) and +(90:2) ..(2,0)
	.. controls +(-90:2) and +(-120:3) ..(-.7,0);
	\draw(-3,0)\ww(3,0)node[right,\hhd]{$\lp$};
	\end{tikzpicture}
	\caption{The braid twist and L-twist}
	\label{fig:T}
\end{figure}

These are the two standard generators of $\SBr(\sop)$ (cf. \Cref{thm:SBr}),
which satisfies the following conjugation formula:
\begin{equation}\label{eq:Psi}
\begin{cases}
    B_{\Psi(\eta)}=\Psi\circ B_{\eta}\circ \Psi^{-1}=B_{\eta}^{\Psi^{-1}},\\
    L_{\Psi(\lp)}=\Psi\circ L_{\lp}\circ \Psi^{-1}=L_{\lp}^{\Psi^{-1}},\\
\end{cases}
\end{equation}
for any $\Psi\in\MCG(\sop)$, $\eta$ in $\CA(\sop)$ and $\lp\in\LA(\sop)$.
We will brief $B_\eta$ and $L_\lp$ as $\eta$ and $\lp$ unless there are confusions or
we want to emphasis they are twists instead of arcs.
Then equations in \eqref{eq:Psi} becomes $\Psi(a)=a^{\iv{\Psi}}$ for $a=\eta$ or $a=\lp$.

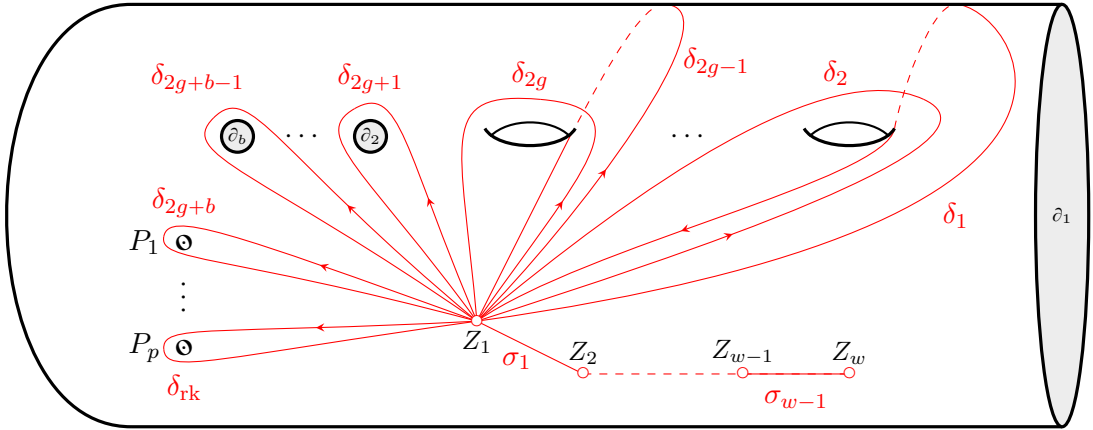
\begin{figure}[hbt]\centering
	\begin{tikzpicture}[scale=.3]
	\draw[red](10,-7)to(14,-9)(24,-9)to(20,-9);\draw[dashed,red](14,-9)to(24,-9);

	\draw[red](10,-7)to(8+2+4-.5,0);
	\draw[red,dashed](8+2+4-.5,0).. controls +(60:1) and +(180:1) ..(17,5);
	\draw[red,->-=.5,>=stealth](10,-7).. controls +(45:5) and +(0:3) ..(17,5);
	\draw[red,-<-=.5,>=stealth](10,-7).. controls +(35:5) and +(-100:2) ..(20+2+4-.3,0.3);
	\draw[red,->-=.5,>=stealth]
              (10,-7).. controls +(18:8) and +(-25:4) ..($(24,0)!-.2!(10,-7)$);
	\draw[red](10,-7).. controls +(40:8) and +(155:4) ..($(24,0)!-.2!(10,-7)$);
	\draw[red,->-=.5,>=stealth]
              (10,-7).. controls +(55:7) and +(-5:4) ..($(12,0)!-.2!(10,-7)$);
	\draw[red](10,-7).. controls +(95:7) and +(175:4) ..($(12,0)!-.2!(10,-7)$);
	\draw[red,->-=.5,>=stealth]
              (10,-7).. controls +(113:7) and +(30:2) ..($(6,0)!-.15!(10,-7)$);
	\draw[red](10,-7).. controls +(126:7) and +(-150:2) ..($(6,0)!-.15!(10,-7)$);
	\draw[red,->-=.5,>=stealth](10,-7).. controls +(140:10) and +(45:2) ..($(1,0)!-.1!(10,-7)$);
	\draw[red](10,-7).. controls +(145:10) and +(-135:2) ..($(1,0)!-.1!(10,-7)$);
	\draw[red,->-=.5,>=stealth](10,-7).. controls +(160:10) and +(79:1) ..($(-1,-4)!-.07!(10,-7)$);
	\draw[red](10,-7).. controls +(165:10) and +(-101:1) ..($(-1,-4)!-.07!(10,-7)$);
	\draw[red,->-=.5,>=stealth](10,-7).. controls +(183:10) and +(95:1) ..($(-1,-8)!-.07!(10,-7)$);
	\draw[red](10,-7).. controls +(188:10) and +(-85:1) ..($(-1,-8)!-.07!(10,-7)$);
	\draw[red,dashed](20+2+4-.3,0.3).. controls +(80:1) and +(180:1) ..(28,5);
	\draw[red](10,-7).. controls +(8:27) and +(0:2) ..(28,5);
	
	\foreach \j in {1,2.5} {
		\draw[thick](8*\j-2+4+.5,0).. controls +(45:1) and +(135:1) ..(8*\j+2+4-.5,0);
		\draw[very thick](8*\j-2+4+.3,0.3).. controls +(-60:1) and +(-120:1) ..(8*\j+2+4-.3,0.3);
	}
	\draw[very thick]
	   (-3,5)to(32,5)(32,-11)  (-3,-11)to(32,-11)
	   (-3,5)to[bend right=90](-3,-11);
	\draw[very thick,fill=gray!14] (32,-3) ellipse (1 and 8) 	node{\tiny{$\partial_{1}$}};
	\draw[very thick,fill=gray!14](1,0)circle(.6) node{\tiny{$\partial_{b}$}};
	\draw[very thick,fill=gray!14](6,0)circle(.6) node{\tiny{$\partial_{2}$}};
	\node at(3.5,0) {$\cdots$};
	\node at(18,0) {$\cdots$};
	\foreach \x/\y in {14/-9,10/-7,24/-9,20/-9}
	   {\draw(\x,\y)\ww;}
    \draw (10,-7)node[below]{\small{$Z_{1}$}}
    	(14,-9)node[above]{\small{$Z_2$}}
    	(20,-9)node[above]{\small{$Z_{\numtri-1}$}}
    	(24,-9)node[above]{\small{$Z_{\numtri}$}};
	\draw[red](-.5,2.2)node{$\delta_{2g+b-1}$} (6,2.2)node{$\delta_{2g+1}$}
    	(12,2.3)node{$\delta_{2g}$} (19,2.8)node{$\delta_{2g-1}$}
    	(23.5,2.4)node{$\delta_2$} (28,-3)node{$\delta_{1}$}
    	(11.5,-8.5)node{$\sigma_1$} (22,-10)node{$\sigma_{\numtri-1}$};
	%\draw[very thick,fill=gray!14] (38,-3) ellipse (1 and 8)
	%	node{\tiny{$\partial_{1}$}};
    \draw[]
        (-1,-4)\dpole (-2.5,-4)node{$P_1$} (-1,-2.5)node[red]{$\delta_{2g+b}$}
        (-1,-6)node[rotate=90]{$\cdots$}
        (-1,-8)\dpole (-2.5,-8)node{$P_p$} (-1,-9.5)node[red]{$\delta_{\rk}$};
	\end{tikzpicture}
	\caption{Generators for $\SBr(\szb)$}
	\label{fig:B's}
\end{figure}

In particular, a set of generators of $\SBr(\sop)$ is shown in \Cref{fig:B's},
which consists of braid twists along $\sigma_i$ and L-twist along $\delta_r$.
Set
\begin{equation}\label{eq:rk}
  \rk\colon=2g+b+p-1.
\end{equation}
Let
\begin{equation}\label{eq:lp}
  \lp_s\colon=\delta_{2g+b-1+s},\quad 1\le s\le p
\end{equation}
be the L-arc encloses the puncture $P_s$.

%=========================================================
\subsection{Mixed twist groups}\label{sec:middle}\
%=========================================================

In \cite{QQ}, we introduce a subgroup $\BT(\sop)$ of $\SBr(\sop)$,
called \emph{braid twist group} of $\sop$, which
is generated by the braid twist $B_\eta$ for any $\eta\in\CA(\sop)$.
In this paper, we shall introduce yet another class of twist groups,
that sits between the two groups above.

Let $\yue\subset\sun$ be a subset of punctures
and $\LA(\sop,\yue)$ be the set of L-arcs consisting of the one that encloses exactly one puncture $P_i$ in $\yue$.

For instance, $\delta_{2g+b+s-1}$ is in $\LA(\sop,\yue)$ if and only if $P_s\in\yue$,
for any $1\le s\le p $ in \Cref{fig:B's}.
In fact, we have
\[
    \LA(\sop,\yue)=\{ \Psi( \lp_s ) \mid P_s\in\yue, \Psi\in\SBr(\sop) \}.
\]
On the one hand, $\Psi\in\SBr(\sop)$ preserves punctures pointwise.
Hence $\Psi(\lp_s)$ encloses the exactly the same puncture $P_{s}$ as $\lp_s$ does.
On the other hand, for any $\delta$ in $\LA(\sop,\yue)$ that encloses the puncture $P_s$,
one can find a mapping class that maps $\delta$ to $\lp_s$, which confirms the claim.
Let
\begin{gather}\label{eq:L-yue}
  \on{L}(\yue)\colon=\< L_{\lp} \mid \lp\in \LA(\sop,\yue)  \>
\end{gather}
be the normal subgroup of $\SBr(\surfo)$ generated L-twist along the L-arcs in $\LA(\sop,\yue)$.

\begin{definition}\label{def:MT}
The \emph{mixed twist group} $\MTsoy=\<\BT(\sop),\on{L}(\yue)\>$ is defined to be
the (normal) subgroup of $\SBr(\sop)$ generated by L-twists in $\on{L}(\yue)$ and all braid twists.
\end{definition}

Take a sequence of subsets
\[
    \emptyset=\yue_0\subset\yue_1\subset\cdots\subset\yue_j\subset\cdots\subset\yue_p=\sun
\]
with $\yue_j=\{ P_1,\cdots, P_j \}$.
It induces a normal series/subnormal series
\begin{equation*}
    \BT(\sop)=\MTrel{\yue_0}
        \lhd \cdots \lhd\MTrel{\yue_j}\lhd\cdots
    \lhd \MTrel{\yue_p}=\MT(\sop) \lhd \SBr(\sop).
        %\textcolor{gray}{\;\;\big(\lhd\MCG(\sop)\big)}.
\end{equation*}
In fact, by the conjugation formula,
we know that any former group in the series is a normal subgroup of a latter one.
These are essentially all the different mixed twist groups we will study.

\begin{example}
When $\surf$ is a disk with $\sun=\emptyset$,
the only mixed twist groups equals $\MCG(\surfo)$,
which is the classical braid group $\Br_{\numtri}$ of type $A_{\numtri-1}$.

When $\surf$ is disk and $p=1$, then $\BT(\sop)$ is isomorphic to braid group of affine type $A_{\numtri-1}$
and is strictly smaller than $\MT(\sop)=\SBr(\sop)$, which is isomorphic to braid group of type $B_{\numtri}$.
\end{example}
%=========================================================
\subsection{Topological Abel-Jacobi maps}\label{sec:AJ}\
%=========================================================

Let $\Ho{1}(\surfy)=\Ho{1}(\surfy,\ZZ)$ be the first homology of $\surfy$.
Forgetting about punctures in $\yue$ gives a map $F^\yue\colon\surfp\to\surf^{\kui}$ and induces a short exact sequence
\[
    1\to\ZZ^{\yue}\to \Ho{1}(\surfp)\xrightarrow{F^\yue_*} \Ho{1}(\surf^{\kui})\to 1.
\]
We will write $\Ho{1}(?)$ for the lattice $\ZZ^?$ (indexed by $?$).

For any element $\lp$ in $\SBr(\sop)$,
there are $\numtri$ paths (or strings in $\pi_1\on{conf}_{\numtri}(\surfp)$)
\[
    \{ p_i\colon Z_i\to\lp(Z_i) \mid Z_i\in\Tri \}
\]
on $\sop$, which forms some numbers of cycles in $\surfp$.
Taking the union/product of these cycles gives an element in $\Ho{1}(\surfp)$,
which define a map
\begin{gather}\label{eq:AJp}
  \AJ^\sun\colon\SBr(\sop)\to\Ho{1}(\surfp)
\end{gather}
sending $\lp$ to $[\coprod p_i]$.
\begin{definition}
We define the \emph{(topological) Abel-Jacobi=AJ map relative to $\yue$} as
\begin{gather}\label{eq:AJ}
    \AJ^{\sun}_{\yue}=F^\yue_*\circ\AJ^\sun\colon\SBr(\sop)\to\Ho{1}(\surf^{\kui}).
\end{gather}
In particular, we have two special ones: $\AJ=\AJ^{\sun}_\sun$ and $\AJ^\sun=\AJ^\sun_\emptyset$.
%\emph{small AJ map}.
\end{definition}

In the remaining of the section, we prove that the mixed twist groups are
exactly the kernels of the corresponding relative AJ maps and give finite presentations of which.

%=========================================================
\subsection{The kernels of the relative AJ maps}\label{sec:SES}\
%=========================================================

We first recall a finite presentation of $\SBr(\sop)$.

\begin{theorem}[Bellingeri-Godelle]\label{thm:SBr}
There is a finite presentation of $\SBr(\sop)$ with
\begin{itemize}
  \item generators $\sigma_1,\cdots,\sigma_{\numtri-1}$ and $\delta_1,\cdots\delta_{\rk }$;
  \item relations
  \begin{itemize}
    \item $\Co(\sigma_i,\sigma_j)$, i.e. $\sigma_i,\sigma_j=\sigma_j\sigma_i$,\; if $|i-j|=1$;
    \item $\Br(\sigma_i,\sigma_j)$, i.e. $\sigma_i,\sigma_j\sigma_i=\sigma_j\sigma_i\sigma_j$,\; if $|i-j|>1$;;
    \item $\Co(\delta_r,\sigma_1\delta_r\sigma_1)$;
    \item $\Co(\sigma_i,\delta_r)$,\; if $i\neq 1$;
    \item $\Co(\delta_s^{\sigma_1},\delta_r)$,\;
        if $s<r$ and $s+1\notin \eveng$;
    \item $\sigma_1 \delta_{r} \sigma_1 \delta_{r-1} \sigma_1=\delta_{r-1} \sigma_1 \delta_{r}$,\;
        if $r+1\in \eveng$.
  \end{itemize}
\end{itemize}
Here $\eveng$ is the set of positive even (resp. odd) integer not bigger than $2g$.
\end{theorem}
For consentience, we introduce the $\varepsilon_r$ inductively as alternative generators
(set $\varepsilon_0=1$):
\begin{equation}\label{eq:vareps}
    \varepsilon_r=
    \begin{cases}
    \delta_r\varepsilon_{r-1}&\text{if $r\notin \eveng$,}\\
    \delta_r\varepsilon_{r-2}&\text{if $r\in \eveng$,}
    \end{cases}
  \quad\Longleftrightarrow\quad
    \delta_r=\begin{cases}
    \varepsilon_r\iv{\varepsilon_{r-1}}&\text{if $r\notin \eveng$,}\\
    \varepsilon_r\iv{\varepsilon_{r-2}}&\text{if $r\in \eveng$,}
    \end{cases}
\end{equation}
for $1\le r\le \rk $.

For $1\le s<r\le \rk $, the relative position of $\varepsilon_s$ and $\varepsilon_r$
are shown in \Cref{fig:abrs} depending on $s+1\notin \eveng$ (left picture)
or $s+1\in\eveng$ (right picture).
Let
\begin{align}\label{tau}
\tau_r=\sigma_1^{\iv{\varepsilon_r}},\quad 1\leq r\leq \rk .
\end{align}
Take any decoration other than $Z_1/Z_2$ and two closed arcs $a,b$ as shown there.
For instance, one can take $a=\sigma_1^{\iv{\sigma_2}}, b=\sigma_2$.

\begin{lemma}\label{lem:tech}
For $1\le s<r\le \rk $, the commutator
\begin{gather}\label{eq:in-BT}
    [\varepsilon_s,\varepsilon_r]=
    \begin{cases}
        \iv{\tau_{s}b}a\iv{\tau_{r}a\tau_{s}}ab\tau_{r}b  & \text{if $s+1\notin \eveng$},\\
        \iv{bb\tau_{s}b}a\tau_{r}\iv{a}\tau_{s}ab\tau_{r}b  & \text{if $s+1\in \eveng$},
    \end{cases}
\end{gather}
is in fact in $\BT(\szb)$. Denote the inverse of the right hand side of \eqref{eq:in-BT} by $n_{s,r}$.
\end{lemma}

\begin{figure}[hbt]
\begin{tikzpicture}[scale=.9]
\def\cc{c}\clip(-3,-1.3)rectangle(10.5,3.5);
\begin{scope}[shift={(0,0)}]
    \draw[thick,blue!80,->-=.72,>=stealth]
        (-2,0) .. controls +(-16:4) and +(30:4) .. node[above]{$\quad\varepsilon_s$} (-2,0);
    \draw[thick,blue!80,->-=.5,>=stealth]
        (-2,0) .. controls +(-20:7) and +(69:7.5) .. node[right]{$\varepsilon_r$} (-2,0);
    \draw[red] (0,3) node[above]{$\tau_r$}
        plot[smooth,tension=3] coordinates {(2,0) (0,3) (-2,0)}
            (0,1) node[above]{$\tau_s$}
        plot[smooth,tension=1.5] coordinates {(2,0) (0,1) (-2,0)};
    \draw[gray!23,fill=gray!23](0,.25)circle(.4)(0,2)circle(.35);
    \draw[red](2,0) edge[Green] node[left]{$\cc_2$}(2,-2) edge node[below]{$b$}(0,-1) \ww
        (-2,0) edge[Green] node[right]{$\cc_1$}(-2,-2) edge node[below]{$a$}(0,-1)\ww
        (0,-1)\ww(-2,0) node[left]{$Z_1$} (2,0) node[right]{$Z_2$};
\end{scope}
\begin{scope}[shift={(7,0)}]
    \draw[thick,blue!80,->-=.5,>=stealth]
        (-2,0) .. controls +(-5:8) and +(52:7) .. node[right]{$\varepsilon_s$} (-2,0);
    \draw[red] (1.5,2.5) node[above]{$\tau_s$}
        plot[smooth,tension=3] coordinates {(2,0) (1.5,2.5) (-2,0)};
    \draw[white,line width=1mm]
        (-2,0) .. controls +(5:7.5) and +(120:5.5) .. (-2,0);
    \draw[thick,blue!80,->-=.42,>=stealth]
        (-2,0) .. controls +(5:7.5) and +(120:5.5) ..  (-2,0) (1.1,1)node{$\varepsilon_r$};
    \draw[white,line width=1mm]
        plot[smooth,tension=3] coordinates {(2,0) (-1.5,2.5) (-2,0)};
    \draw[red, thick] (-1.5,2.5) node[above]{$\tau_r$}
        plot[smooth,tension=3] coordinates {(2,0) (-1.5,2.5) (-2,0)};
    \draw[gray!23,fill=gray!23](0,1)circle(.4) (-2,1.5)circle(.25)(2,1.5)circle(.25);
    \draw[red](2,0) edge[Green] node[left]{$\cc_2$}(2,-2) edge node[below]{$b$}(0,-1) \ww
        (-2,0) edge[Green] node[right]{$\cc_1$}(-2,-2) edge node[below]{$a$}(0,-1)\ww
        (0,-1)\ww(-2,0) node[left]{$Z_1$} (2,0) node[right]{$Z_2$};
\end{scope}
\end{tikzpicture}
\caption{Commutators of $H$}\label{fig:abrs}
\end{figure}
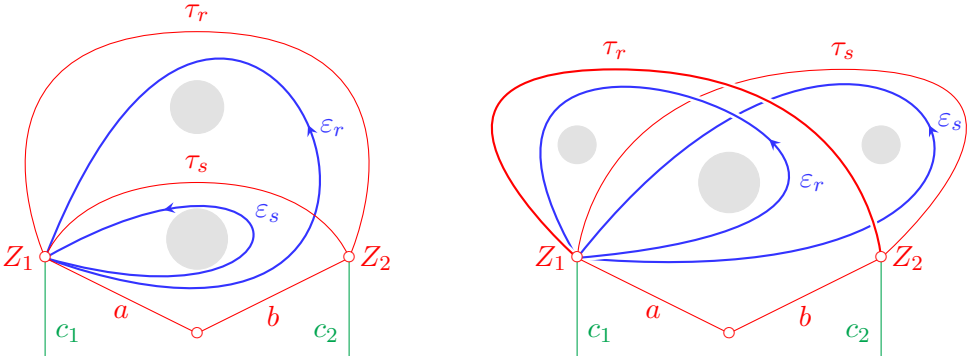

\begin{proof}
To examine the mapping classes, we can restrict to the neighbourhood of the union of arcs $a\cup b\cup\tau_s\cup\tau_r$.
A straightforward checking that $n_{s,r}$ preserves both the arc $b$ and $c_2$.
(Obviously $[\varepsilon_s,\varepsilon_r]$ preserves them too.)
Hence, we only need to looking at the neighbourhood of $(a\cup b\cup\tau_s\cup\tau_r)\setminus b$,
which is equivalent to the neighbourhood of the union of loops $\varepsilon_s\cup\varepsilon_r$.
This reduces to the case of the surface braid group of a $S_*$ with only one decoration $Z_1$.
As this local surface braid group can be identified with $\pi_1(S_*)$,
we only need to check if both sides of \eqref{eq:in-BT} move $Z_1$ in the same way.
%(or equivalently if they map $c_1$ to the same arc).
This is again an easy picture-chasing:
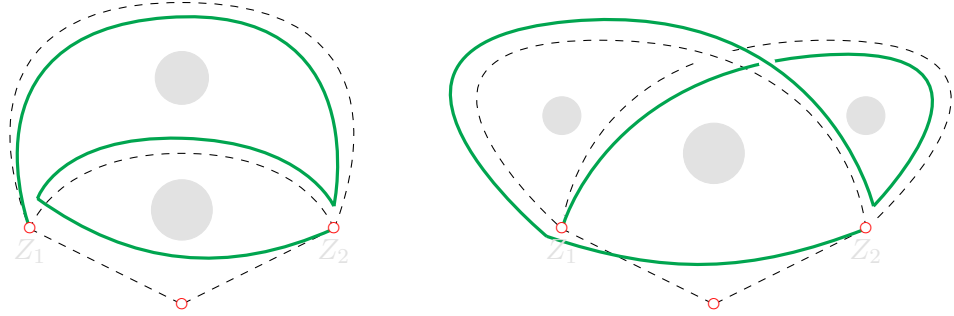
\begin{figure}[hbt]
\begin{tikzpicture}[scale=.9]
\def\cc{c}\clip(-3,-1.1)rectangle(10.5,3.5);
\begin{scope}[shift={(0,0)}]
    \draw[dashed,thin]
        plot[smooth,tension=3] coordinates {(2,0) (0,3) (-2,0)}
        plot[smooth,tension=1.5] coordinates {(2,0) (0,1) (-2,0)}
        (2,0) to (0,-1) (-2,0) to (0,-1);

    \draw[Green, very thick]
        plot[smooth,tension=2.5] coordinates {(2,0.3) (-.2,2.8) (-2,0)}
        plot[smooth,tension=1.5] coordinates {(2,0.3) (0,1.2) (-1.9,0.4)}
        (-1.9,0.4)to[bend right=30](2,0);

    \draw[gray!23,fill=gray!23](0,.25)circle(.4)(0,2)circle(.35)
        (0,-1) \ww (2,0)\ww node[below]{$Z_2$} (-2,0)\ww node[below]{$Z_1$} ;
\end{scope}
\begin{scope}[shift={(7,0)}]
    \draw[dashed,thin]
        (2,0) to (0,-1) (-2,0) to (0,-1)
        plot[smooth,tension=3] coordinates {(2,0) (1.5,2.5) (-2,0)};
    \draw[white, line width=3mm]
        plot[smooth,tension=3] coordinates {(2,0) (-1.5,2.5) (-2,0)};
    \draw[dashed,thin]
        plot[smooth,tension=3] coordinates {(2,0) (-1.5,2.5) (-2,0)};

    \draw[Green, very thick]
        plot[smooth,tension=2.5] coordinates {(2.1,0.3) (1.4,2.3) (-2,0)};
    \draw[fill=white,white] (.7,2.2) circle (.1);
    \draw[Green, very thick]
        plot[smooth,tension=2.8] coordinates {(2.1,0.3) (-1.9,2.75) (-2.2,-.1)}
        (-2.2,-.1)to[bend right=20](2,0);

    \draw[gray!23,fill=gray!23](0,1)circle(.4) (-2,1.5)circle(.25)(2,1.5)circle(.25)
        (0,-1) \ww (2,0)\ww node[below]{$Z_2$} (-2,0)\ww node[below]{$Z_1$} ;
\end{scope}
\end{tikzpicture}
\caption{The arcs $\tau_s^{a\tau_rb}$ and $\tau_r^{\iv{a}\tau_sb}$ in two cases, respectively}
\label{fig:check}
\end{figure}
\begin{itemize}
  \item In the former case, one can rewrite the right hand side as (using $\Br(b,\tau_r)$)
    \[
        \iv{\tau_{s}b}a\iv{\tau_{r}a} \tau_{s} a\tau_{r}b\tau_{r}=
            \iv{\tau_s} \cdot a^{b} \cdot \tau_s^{a\tau_rb} \cdot \tau_r,
    \]
    where the arc $\tau_s^{a\tau_rb}$ is the arc in the left pictures of \Cref{fig:check}
    that makes the checking straightforward.
  \item In the latter case, rewrite the right hand side as (using $\Br(b,\tau_r), \Br(b,\tau_s)$)
    \[
        \iv{b\tau_sb\tau_s}a\tau_r\iv{a}\tau_{s}ab\tau_{r}b=
        \iv{\tau_s}\cdot \iv{b}^{\iv{\tau_s} } \cdot \tau_r^{\iv{a}\tau_sb} \cdot a^b \cdot \tau_r^b \cdot \tau_r.
    \]
    The movement of $Z_1$ by the twist above is as follows:
    \begin{itemize}
      \item it moves along $\tau_r$ to $Z_2$, wouldn't effect by the twist $\tau_r^b$ and then
        moves along $a^b$ back to $Z_1$ (which is equivalent to moving along $\iv{\varepsilon_r}$ at this stage);
      \item it continues moving along $\tau_r^{\iv{a}\tau_sb}$ to $Z_2$, as shown in the right pictures of \Cref{fig:check}; then wouldn't effect by the twist $\iv{b}^{\iv{\tau_s} }$;
      \item it finally moves along $\iv{\tau_s}$ and one sees that the claim holds.\qedhere
    \end{itemize}
\end{itemize}
\end{proof}

Note that the calculation above (and the theorem below) is essentially hidden in \cite{QZ3}.
%More precisely, the expression in \eqref{eq:in-BT} is there.

\def\Hyue{H^{\yue}}
\def\Hkui{H^{\kui}}
\def\Hyui{H^{\kui}}
Let
\begin{gather}
    \Hyue\colon=\pi_1(\surfy,Z_1) = \SBr(\surfy_{Z_1}) \le \SBr(\sop),
\end{gather}
which can be identified with a subgroup of $\SBr(\szb)$ generated by $\lp_s$ for $P_s\in\yue$.
Moreover, one has $\Ho{1}(\surfy)=\Hyue\big{/}[\Hyue,\Hyue]$.

\begin{theorem}\label{thm:QZ+}
%Let $\kui=\kui.$
There is a short exact sequence
\begin{gather}\label{eq:SES-b}
    1 \to \MTsoy \to \SBr(\sop) \xrightarrow{\AJ^{\sun}_{\yue}} \Ho{1}(\surf^{\kui}) \to 1,
\end{gather}
or equivalently $\MTsoy=\ker\AJ^{\sun}_{\yue}$.
Moreover, $\Hyui\cap\MTsoy = [\Hyui,\Hyui]$.
\end{theorem}

\begin{proof}
\def\NN{N}
We write $\NN=\MTsoy=\<\BT(\surfp),\lp_s\mid s\in\yue\>, H=\Hyui$ in this proof.
As $\{\sigma_i,\delta_j\}$ is a generating set,
one deduces that $\SBr(\sop)=\NN\cdot H$.
Hence we have a short exact sequence
\begin{equation}\label{eq:ses.b}
  1\to \NN \to \SBr(\sop) \to H\big{/}  H\cap \NN  \to 1
\end{equation}
or $H \big{/} H\cap \NN=\SBr(\sop)/\NN$.

Clearly, $\BT(\surfp)\subset\ker\AJ^\sun$.
This is because any braid twist $B_\eta$ corresponds to two paths $p_i$
(i.e. going along $\eta$ one way or the other),
which are inverse to each other and form a trivial/contractible cycle.
As $\AJ^{\sun}_{\yue}=F_*^{\yue}\circ\AJ^\sun$, we have
$N\subset\ker\AJ^\sun\lhd \ker\AJ^{\sun}_{\yue}$.
Also, $\lp_s\in\AJ^{\sun}_{\yue}$ for $P_s\in\yue$ as such a puncture is forgotten in $F^{\yue}$.
Therefore the map $\AJ^{\sun}_{\yue}$ factors through the quotient by (the intersection with) $N$:
\begin{equation}\label{eq:Xi_*}
  \begin{tikzcd}[column sep=0]
    \SBr(\sop) \ar[rr,"\AJ^{\sun}_{\yue}"] \ar[dr,->>,] && \Ho{1}(\surf^{\kui})&(=H\big{/}[H,H]).\\
    & H\big{/} H\cap N \ar[ur,"\Xi"']
\end{tikzcd}
\end{equation}
Restricted to $H$, we see that $H\cap N \subset [H,H]$.

One the other hand, by \Cref{lem:tech}, $[\varepsilon_r,\varepsilon_s]$ is in $\BT(\sop)\lhd N$ for any $\varepsilon_r,\varepsilon_s\in H$.
As they (normally) generated $[H,H]$, it implies that $H\cap N \supset [H,H]$ and thus $H\cap N = [H,H]$.
Looking back at \eqref{eq:Xi_*}, we see that $\Xi$ is an isomorphism and
\eqref{eq:ses.b} becomes the required \eqref{eq:SES-b}.
\end{proof}

In particular, we have
\begin{equation}\label{eq:kerker}
\begin{cases}
 \BT(\sop)=\MTrel{\emptyset}=\ker\AJ^{\sun}_\emptyset,\\
 \MT(\sop)=\MTrel{\sun}=\ker\AJ^\sun_\sun=\ker\AJ
\end{cases}\end{equation}
and a commutative diagram of short exact sequences of groups (which will be called an octahedron)
show as the left diagram of \Cref{fig:3x3}.
The right graph of \Cref{fig:3x3} is the Auslander-Reiten quiver style view (of the octahedron among more short exact sequences).

\begin{figure}\centering
\begin{tikzpicture}[yscale=1.4,scale=1.1,font=\small]
\clip(-6,.8)rectangle(5.5,3.3);
\draw (-3,1.7) node{
    \begin{tikzcd}[column sep=17, row sep=24,font=\small]
      \MTsoy\ar[r,hookrightarrow] \ar[d,equal] &
        \MT(\sop) \ar[r,twoheadrightarrow] \ar[d,hookrightarrow] & \Ho{1}(\kui) \ar[d,hookrightarrow]\\
      \MTsoy \ar[r,hookrightarrow] &
        \SBr(\sop)\ar[r,twoheadrightarrow] \ar[d,twoheadrightarrow] & \Ho{1}(\surf^{\kui}) \ar[d,twoheadrightarrow]\\
      & \Ho{1}(\surf) \ar[r,equal] & \Ho{1}(\surf)\\
  \end{tikzcd}
};
\draw (3,3) node(x0){$\SBr(\sop)$}
    (2,2) node(x1){$\MT(\sop)$}
    (1,1) node(x2){$\MTsoy$}
    (4,2) node(z3){$\Ho{1}(\surf^{\kui})$}
    (5,1) node(z2){$\Ho{1}(\surf)$}
    (3,1) node(y1){$\Ho{1}(\kui)$};
\draw[-stealth] (x2)edge(x1) (x1)edge(x0)
    (x0)edge(z3) (z3)edge(z2) (x1)edge(y1) (y1)edge(z3);
\draw(4,1)node{$\oplus$};
%\draw (3,3) node(x0){$\SBr(\sop)$}(2,2) node(x1){$\MT(\sop)$}
%    (1,1) node(x2){$\MTsoy$}(0,0) node[gray](x3){$\BT(\sop)$}
%    (4,2) node[gray](z3){$\Ho{1}(\surfp)$}
%    (5,1) node(z2){$\Ho{1}(\surf^{\kui})$}(6,0) node(z1){$\Ho{1}(\surf)$}
%    (2,0) node[gray](y3){$\Ho{1}(\yue)$}
%    (4,0) node(y2){$\Ho{1}({\kui})$}(3,1) node[gray](y1){$\Ho{1}(\sun)$};
%\draw[-stealth] (x3)edge[gray](x2)(x2)edge(x1)(x1)edge(x0)
%    (x0)edge(z3)(z3)edge(z2)(z2)edge(z1)
%    (y3)edge[gray](y1)(y1)edge(y2) (x1)edge(y1)(x2)edge[gray](y3)(y1)edge[gray](z3) (y2)edge(z2);
%\draw[gray] (5,0)node{$\oplus$}(3,0)node{$\oplus$}(4,1)node{$\oplus$};
\end{tikzpicture}
\caption{An octahedron (left) and an AR-quiver style view (right)}\label{fig:3x3}
\end{figure}
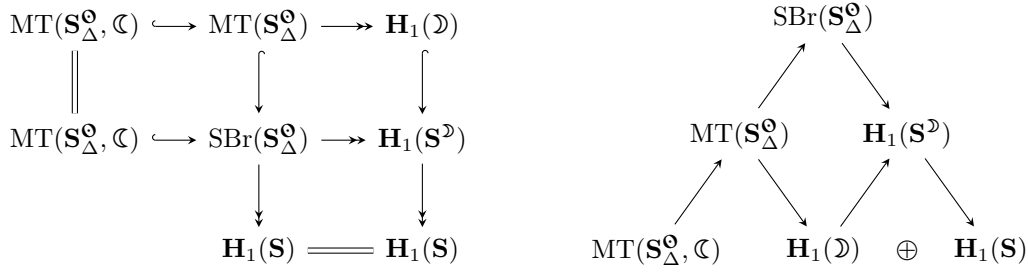

%=========================================================
\subsection{Generators from dual graph of a decorated triangulation}\label{sec:gen}\
%=========================================================

Recall that an open arc in $\surfp$ is a curve connected points in $\M\cup\sun$.
A \emph{triangulation} $\RT=\{\gamma_i\}$ of $\surfp$ is a maximal collection of compatible (i.e. no intersection in the interior of $\surfp$) open arcs $\gamma_i$, which will divide $\surfp$ into triangles
(called \emph{$\RT$-triangle} or just a triangle for short sometimes).
The following is a well-known fact.

\begin{lemma}\label{lem:wk}
Any triangulation $\RT$ of $\surfp$ consists of
\begin{gather}\label{eq:n}
    n=\rank(\surfp)=6g+3p+3b+m-6
\end{gather}
(simple essential) open arcs that divides $\surfp$ in to $({2n+m})/{3}$ triangles.
\end{lemma}

From now on, we assume that the DMSp $\sop$ satisfies
\begin{equation}\label{eq:aleph}
  \numtri=\frac{2n+m}{3}=4g+2p+2b+m-4\ge1.
\end{equation}
Since $m,b\ge1$ and $g\ge0$, then $\numtri\ge 2p-1$.

We associated a quiver with potential to any (decoration) triangulations.

\begin{definition}\label{def:QP}
Let $\RT$ be a triangulation of $\surfp$.
Define the (degenerate when $\sun\ne\emptyset$) quiver with potential $(\Qp,\Wp)$ as follows:
\begin{itemize}
\item The vertices of $\Qp$ are indexed by open arcs in $\RT$.
\item The arrow of $\Qp$ correspond to clockwise angles between edges of $\RT$-triangle $T$.
Hence, there is a 3-cycle $Q_T$ sub-quiver in $(\Qp,\Wp)$ for each $T$.
\item The potential $\Wp$ are the sum over all 3-cycle over all $\RT$-triangle $T$,
which will be called the 3-cycles potential.
\end{itemize}
\end{definition}

See \Cref{fig:QP} for example.
By definition, the (forward) flip of triangulations induces
the (forward) mutation of the associated quivers with potential.
\begin{figure}[hb]\centering
    \includegraphics[width=\textwidth]{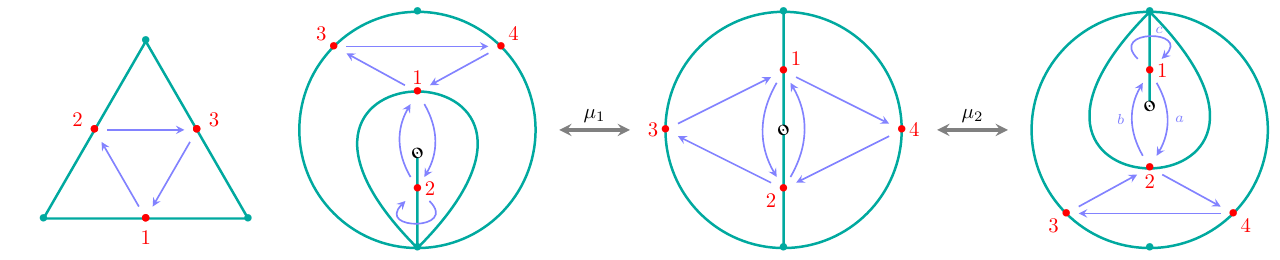}
\caption{Quivers with 3-cycles potential associated to triangulations}
\label{fig:QP}
\end{figure}

A $\RT$-triangle $T$ is \emph{self-folded} if two of its edges are coincide as one $\gamma_i$.
We call such a $\gamma_i$ a \emph{self-folded edge}.
The other edge of $T$ is call an \emph{enclosing edge}.

\begin{definition}
A puncture is called \emph{$\RT$-isolated}, if it is enclosed in a self-folded $\RT$-triangle.
Denote by $\yue(\RT)$ the set of $\RT$-isolated triangles.
We say a triangulation $\RT$ is \emph{admissible} if all punctures are $\RT$-isolated, i.e. $\yue(\RT)=\sun$.
\end{definition}

Note that the quiver $\Qp$ has exactly $|\yue(\RT)|$ loops and $|\yue(\RT)|$ 2-cycles,

For any decorated triangulation $\T$ of $\sop$,
the associated quiver with potential $(Q_\T,W_\T)$
is defined to the one associated to  $\RT=F^\Delta(\T)$ of $\surfp$.
Then we have the same notions of $\T$-isolated, $\yue(\RT)$ and admissible in the decorated case.

Let $\T^*$ be the dual graph of $\T$ consisting of closed arcs.
There is a partition $\T^*=\T^*_{\CA}\sqcup\T^*_{\LA}$ for
\[
    \T^*_{\CA}=\T^*\cap\CA(\sop) \quad\text{and}\quad \T^*_{\LA}=\T^*\cap\LA(\sop).
\]
Note that $\T^*_{\LA}$ are precisely the dual of self-folded edges in $\T$.

\begin{convention}
For $Q_\T$, we will identify its vertices with (simple or L-) closed arcs.
%(but for $\Qp$, they are identified with open arcs instead).
\end{convention}

\begin{definition}\label{def:MT2}
The \emph{mixed twist group} $\MT(\T)$ with respect to a triangulation $\T$ is defined to be
the subgroup of $\SBr(\sop)$ generated by the braid twists $\eta$ for $\eta\in\T^*_{\CA}$ and
the L-twists $\lp$ for $\lp\in\T^*_{\LA}$.
\end{definition}

By \cite[Lem.~A.1]{QZ1}, $\surfp$ (and hence $\sop)$ admits admissible triangulations.
Assume $\T$ is admissible.
Then by cutting out all the self-folded edges in self-folded $\T$-triangles,
we obtain a triangulation $\T_0$ for the result DMS $\surfo^0$ without punctures.
Moreover, one can naturally identify arcs in $\T_{\CA}$ and arcs in $\T_0$
as well as their dual graphs $\T^*_{\CA}=\T_0^*$.

By \cite[Thm.~5.9]{QZ3}, we have a presentation of
$\BT(\surfo^0)=\BT(\T_0)$ with $\T_0^*$ as the set of generators and
a set of relations, denote by $R(\T_0)$,
described in terms of the associated quiver with potential $(Q_{\T_0},W_{\T_0})$.
Here, $(Q_{\T_0},W_{\T_0})$ is the full sub-quiver of $(Q_{\T},W_{\T})$, restricted to $Q_{\T_{\CA}}=Q_{\T_0}$
and $\BT(\T_0)$ can be naturally embedded into $\MT(\T)$ (induced from the embedding $\surfo^0\hookrightarrow\sop$.

\begin{theorem}\label{thm:MT's}
Let $\T$ be an adimissible triangulation of $\sop$ and
$\T_0$ the corresponding triangulation of the surface $\surfo^0$ cutting at the punctures in $\yue(\T)$.
Then the mixed twist group $\MT(\sop)=\MT(\T)$ has the following finite presentation:
\begin{itemize}
	\item Generators:
    \begin{itemize}
      \item generators of $\BT(\T_0)$ (i.e. braid twists $B_\eta$ for $\eta\in\T^*_{\CA}$) together with
      \item L-twists $L_\lp$ for $\lp\in\T^*_{\LA}$.
    \end{itemize}
	\item Relations:
    \begin{itemize}
      \item $R(\T_0)$ for $\BT(\T_0)$ together with
      \item $\Br^{2\Int(\beta,\lp)+2}(\beta,\lp)$ for any $\beta\in \T^*$ and $\lp\in\T^*_{\LA}$.
    \end{itemize}
\end{itemize}
\end{theorem}

A special case of the theorem above is the following (when $Q_\T$ has no double arrows),
where the relations are much simpler.

\begin{corollary}\label{cor:MT's2}
For an admissible triangulation $\T$ of $\sop$ such that there is no double arrow in $Q_\T$,
then the mixed twist group $\MT(\sop)$ has the following finite presentation:
\begin{itemize}
	\item Generators: braid twists $\eta$ for $\eta\in\T^*_{\CA}$ and
        L-twists $\lp$ for $\lp\in\T^*_{\LA}$.
	\item Relations:
    \begin{itemize}
      \item higher braid relation $\Br^{2\Int(a,b)+2}(a,b)$ for any $a,b\in \T^*$.
      \item triangle relation $abca=bcab=cabc$ if $a,b,c\in\T^*_{\CA}$ share a common endpoint/decoration $Z$ and they are in clockwise order at $Z$.
    \end{itemize}
\end{itemize}
\end{corollary}

Note that for the (higher) braid relation in \Cref{cor:MT's2}, there are three cases:
\begin{itemize}
    \item $\Co(a,b)$ for any $a,b\in\T^*$ with $\Int(a,b)=0$.
    \item $\Br(a,b)$ for any $a,b\in\T^*_{\LA}$ with $\Int(a,b)=1/2$.
    \item $\Br^4(a,b)$ for $a\in\T^*_{\CA}$ and $b\in\T^*_{\LA}$ with $\Int(a,b)=1$.
\end{itemize}
Here $\Int$ is the intersection number, cf. \cite[Def.~3.1]{QQ}.

%=========================================================
\subsection{Proof of \Cref{thm:MT's}}\label{sec:pf}\
%=========================================================

%=========================================================
\paragraph{\textbf{Generators}}\

We first show that $\MTrel{\yue_{j+1}}=\<\MTrel{\yue_j},\lp_{j}\>$,
which follows from the following lemma.

\begin{lemma}\label{lem:conju-lp}
If two L-arcs $\lp_1$ and $\lp_2$ in $\LA(\yue)$ enclosing the same puncture $P$,
then there is $b\in\BT(\sop)$ such that $\lp_1=b(\lp_2)$.
In particular, $L_{\lp_2}=L_{\lp_1}^b$.
\end{lemma}
\begin{proof}
An L-arc $\lp$ in $\LA(\yue)$ cut out a once-puncture disk,
which contains a unique simple arc connecting the puncture inside and the decoration (its endpoint).
Call such an arc the bone of $\lp$. Denote by $g_i$ the bone of $\lp_i$ for $i=1,2$.
Then we only need to show that there exist $b\in\BT(\sop)$ such that $b(g_2)=b(g_1)$.

We choose a triangulation $\T$ of $\sop$ such that $g_i$ is contained in a triangle.
This is can be done easily due to change of coordinate principal.
Then we have
$$\Int(\T,g_1)=\sum_{\gamma\in\T}\Int_{\surf^\circ}(\gamma,g_1)=0.$$
By (the proof of) \cite[Lem.~3.14]{QQ},
there exist $b'\in\BT(\sop)$ such that $\Int(\T,b'(g_2))<\Int(\T,g_2)$ unless $\Int(\T,g_2)=0$
By induction on $\Int(\T,g_i)$, there exists $b\in\BT(\sop)$ such that $\Int(\T,b(g_2))=0$,
which implies that $b(g_2)=g_1$.
\end{proof}

%=========================================================
\paragraph{\textbf{An inductive algorithm}}\

\begin{lemma}\label{lem:adding}
Suppose that there is a finite presentation
\[
    \MTrel{\yue_j}=\frac{ \<a,\eta,x,\cdots\> } { \Big(
        \Br(\eta,x), \Br(a,x), \Br(a,x^\eta), \Co(?,a), \Co(?,\eta), ?\ne x \Big)
        \bigcup R_j },
\]
where $\eta,x,a$ are shown in \Cref{fig:br4},
the rest of the generators are disjoint with $\eta\cup a$ and
any relation in $R_j$ does not involve $\eta$ and $a$.

Then there is a finite presentation
\begin{equation}\label{eq:MT+1}
    \MTrel{\yue_{j+1}}=\frac{ \<\lp,\eta,x,\cdots\> } { \Big(
        \Br^4(\lp,\eta), \Br(\eta,x), \Co(x,\lp),\Co(?,\lp), \Co(?,\eta), ?\ne x \Big)
        \bigcup R_j}
\end{equation}
for some conjugation $\lp$ of $\lp_{j+1}$.
\end{lemma}
\begin{proof}
We can make the (inner) semidirect product $\MTrel{\yue_j}\rtimes_{\rho}\ZZ\<\lp\>$ that fits into the following commutative diagram of short exact sequences of groups:
\begin{equation}\label{eq:2x3}
  \begin{tikzcd}
    & 1 \ar[r] & \MTrel{\yue_j} \ar[r] \ar[equal]{d} &
        \MTrel{\yue_j}\rtimes_{\rho}\ZZ\<\lp\> \ar[d] \ar[r] & \ZZ\<\lp\> \ar[r]\ar[equal]{d} & 1\\
    & 1 \ar[r] & \MTrel{\yue_j} \ar[r]  & \MTrel{\yue_{j+1}}  \ar[r] & \ZZ\<\lp\> \ar[r] & 1.\\
  \end{tikzcd}
\end{equation}
By calculating the conjugation action of generators by $\lp$  (denoted by $\rho_{\lp}$),
e.g.
\[
    \rho_\lp(\eta)=a'=a^{\iv{\eta}} \quad\text{gives relation}\quad \eta^{\lp}=a^{\iv{\eta}},
\]
we deduce that $\MTrel{\yue_j}$ admits a presentation with generators
$\<\lp,a,\eta,x,\cdots\>$ and relations
\begin{equation}\label{eq:MT1}
\Big(
        \begin{cases}
            \Br(\eta,x), \Br(a,x), \Br(a,x^\eta), \Co(?,a), \Co(?,\eta), ?\ne x \\
            a^{\lp}=\eta, \eta^{\lp}=a^{\iv{\eta}}, x^{\lp}=x, \Co(*,\lp), *\ne \eta,a,x
        \end{cases}
        \Big) \bigcup R_j.
\end{equation}
Here, for most conjugations by $\rho_\lp$, the element $*$ (e.g. $x$) does not change
so that we get $\Co(*,\lp)$.
%For instance, $x^{\iv{\lp}}=x\Longleftrightarrow\Co(\lp,x)$.
By $a^{\lp}=\eta \Leftrightarrow a=\eta^{\iv{\lp}}$, we can eliminate $a$ from the generating set and direct calculation shows the following
\begin{gather}\label{eq:rels}
    \eta^{\lp}=a^{\iv{\eta}} \;\Longleftrightarrow\; \iv{\lp} \eta \lp=\eta \eta^{\iv{\lp}} \iv{\eta}
         \;\Longleftrightarrow\; \lp\eta\lp\eta=\eta\lp\eta\lp
\end{gather}
or $\Br^4(\lp,\eta)$ as expected.
Furthermore, using these relation, we have
\[\begin{cases}
    \Br(a,x) \;\Longleftrightarrow\; \Br(\eta^{\iv{\lp}},x) \;\xLongleftrightarrow{\Co(\lp,x)}\; \Br(\eta,x)\\
    \Br(a,x^\eta) \;\Longleftrightarrow\; \Br(\eta^{\iv{\lp}},x^\eta)
        \;\Longleftrightarrow\; \Br(\eta^{\iv{\lp\eta}},x)
        \;\xLongleftrightarrow{\Br^4(\lp,\eta)}\; \Br(\eta^{\lp},x)
        \;\xLongleftrightarrow{\Co(\lp,x)}\; \Br(\eta,x).
\end{cases}\]
Thus, the presentation \eqref{eq:MT1} simplifies to \eqref{eq:MT+1}.
\end{proof}

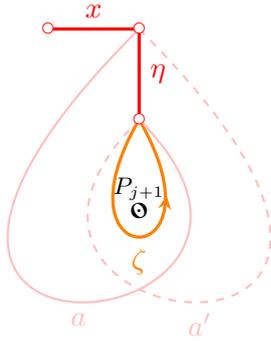
\begin{figure}[ht]\centering\vskip -.5cm
\begin{tikzpicture}[scale=.6]\clip (-6,5.1)rectangle(6,-2.8);
  \draw[\hhd, very thick,>=stealth,-<-=.35]
    (0,2).. controls +(-60:4) and +(-120:4) ..node[below]{$\lp$}(0,2);
  \draw[pink, thick]
    (0,4).. controls +(-135:12) and +(-45:7) ..node[below]{$a$}(0,2);
  \draw[pink, dashed, thick]
    (0,4).. controls +(-45:12) and +(-135:7) ..node[below]{$a'$}(0,2);
  \draw[very thick,red]
    (0,2)\ww to node[right]{$\eta$}(0,4)    (0,4)\ww to node[above]{$x$}(-2,4)\ww;
  \draw(0,0)\dpole node[above,font=\footnotesize]{$P_{j+1}$};
\end{tikzpicture}
\caption{Presentation by extension}\label{fig:br4}
\end{figure}

%=========================================================
\paragraph{\textbf{A first presentation for $\MT(\sop)$}}\

We proceed to calculate presentations of the mixed twist groups $\MTrel{\yue_j}$.
We follow the notations of \Cref{fig:B's} and \eqref{eq:lp}.
%and recall that  $\tau_r=\sigma_1^{ \iv{\varepsilon_r} }$.
Set
\[
    \tau_{-s}=\sigma_1^{ \zeta_p\cdots \zeta_{s} } ,\quad 1\le s\le p
\]
to replace $\tau_{2g+b-1+s}$ for $1\le s\le p$.
Then the relative position of these $\tau_r$'s is shown in the top picture of \Cref{fig:gen}.
We recall another finite presentation in \cite{QZ3} (with slight modification).

\begin{theorem}\cite[Thm.~4.1]{QZ3}\label{thm:pre}
Suppose that either $\numtri\geq 5$, or $\numtri=4$ and $\rk\le 2$.
The braid twist group $\BT(\sop)$ has the following finite presentation.
	\begin{itemize}
		\item Generators: $\sigma_i$ and $\tau_r$ for $1\leq i\leq \numtri-1, -p\leq r(\ne0)\leq 2g+b-1$.
		\item Relations: the standard ones
        \begin{equation}\label{eq:standards}\begin{cases}
        	\Co(\sigma_i,\sigma_j)&    \text{if $|i-j|>1$}\\
            \Br(\sigma_i,\sigma_{j})&   \text{if $|i-j|=1$}\\
            \Co(\tau_r,\sigma_i)&       \text{if $i>2$}\\
            \Br(\tau_r,x)& \forall r\\
            \Br(\tau_r,y)& \forall r\\
        \end{cases}\end{equation}
        together with two extra ones
        \begin{equation}\label{eq:whtau sr}\begin{cases}
            \Co({\tau_r}^y,{\tau_s}^x) &\text{if $r<s<0$ or $s<0<r$,}\\
            \Co({\tau_r}^y,{\tau_s}^x) &\text{if $0<s<r$ and $s\in\eveng$,}\\
            \Co({\tau_r}^{\iv{y}},{\tau_s}^x) &\text{if $0<s<r$ and $s+1\in\eveng$,}
        \end{cases}\end{equation}
        for (cf. \cite[Fig.~11]{QZ3}
        $$x:=\tau_{-1}{}^{\iv{\sigma_2}}=\sigma_2^{\tau_{-1}},\quad y:=\sigma_3^{\iv{\sigma_2}}=\sigma_2^{\sigma_3}$$
        and any $-p\le s, r< 2g+b$ with $sr\ne0$.
	\end{itemize}
\end{theorem}

Denote by $\Sigma_j$ the surface obtained form $\sop$ by forgetting about the punctures
$\{P_s\mid 1\le s\le j\}$.
Note that $\Sigma_0=\sop$ and $\Sigma_p\cong\surfo^0$ (as decorated surfaces).
The braid twist group $\BT(\Sigma_j)$ is a subgroup of $\BT(\sop)$ with generators
\begin{gather}\label{eq:Gj}
    G_j=\{\sigma_i, \tau_r \mid 1\leq i\leq \numtri-1, \ -p+j\leq r(\ne0) \leq 2g+b-1\}
\end{gather}
and same relations involving these generators.
Denote the set of relations by $\Lambda_j$.

The relative positions of the generators in $G_0$ are shown in the top picture of \Cref{fig:gen}
and there are four kinds of them:
\begin{itemize}
  \item $\tau_s$ for ${1\le s\le 2g}$ correspond to arcs going around the genus.
  \item $\tau_s$ for ${2g+1\le s\le 2g+b-1}$ correspond to arcs going around the boundary components.
  \item $\tau_{-s}$ for $1\le s\le p$ correspond to arcs going around punctures.
  \item $\sigma_i$ for ${1\le i\le \numtri-1}$ correspond to the ones connecting other decorations.
\end{itemize}

\begin{figure}[htb]\centering
    \includegraphics[width=13cm]{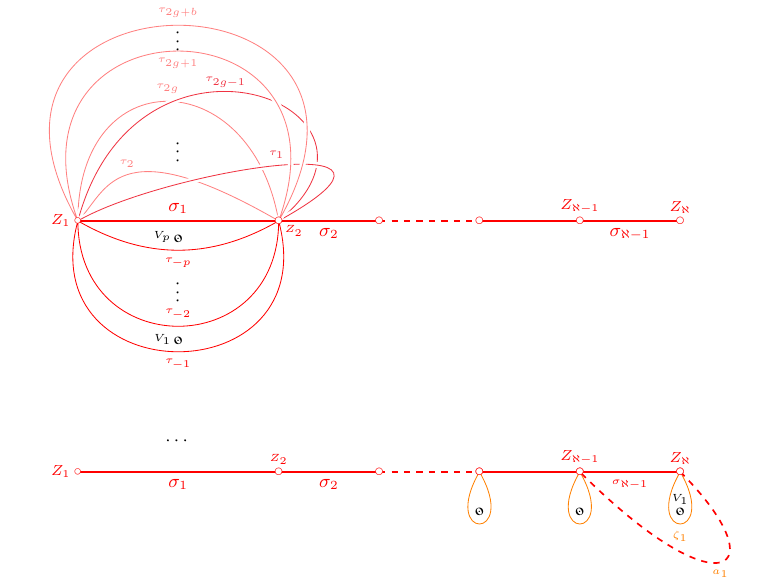}\vskip -.5cm
\caption{Moving the generators}
\label{fig:gen}
\end{figure}

Our strategy is to moving the arc $\tau_{-s}$, for $1\le s\le p$, so that their endpoints are separated
the upper part of the top picture of \Cref{fig:gen} as shown in the bottom picture there.
The method of moving is by conjugation.
For instance, we begin with moving $\tau_{-1}$ to $a_1$,
by conjugating $\sigma_2, \ldots , \sigma_{\numtri-1}$ first, then by $\tau_{-2}$ and
then by $\sigma_2, \ldots , \sigma_{\numtri-2}$.
Explicitly, we have
\begin{gather}\label{eq:move}
  a_1=\tau_{-1}^{\quad\iv{ \left( \sigma_2 \cdots \sigma_{\numtri-1} \right) \cdot \tau_{-2}  \left( \sigma_2 \cdots \sigma_{\numtri-2} \right)   }}
\end{gather}
The change of the relations are easy to write down accordingly as
at each step we only conjugate by one of the generators.
Now we apply \Cref{lem:adding},
setting $(x,\eta,a,\lp)=(\sigma_{\numtri-2}.\sigma_{\numtri-1},a_1,\lp_1)$,
to get a presentation of the mixed twist group
\[
    \MTrel{\yue_1}=\<\MTrel{\yue_0},\lp_{1}\>=\<\BT(\surfo),\lp_1\>
\]
from \Cref{thm:pre},
with the set of generators $(G_0\setminus\{\tau_{-1}\})\cup \{\lp_1\}=G_1\cup \{\lp_1\}$.
Directly calculation shows that the relations are precisely $\Lambda_1$ together with
\[
    \Br^4(\lp_1,\sigma_{\numtri-1}) \quad\text{and}\quad
    \Co(\lp_1,\alpha), \alpha\in G_1\setminus\sigma_{\numtri-1}.
\]

Now we can proceed to move next $\tau_{-s}$ and then include another $\lp_s$ in $\LA(\sop)$ to
get a presentation inductively for
\begin{equation}\label{eq:MTj}
    \MTrel{\yue_j}=\<\MTrel{\yue_{j-1}},\lp_{j}\>.
\end{equation}
Note that in order to apply \Cref{lem:adding} properly, we need to move $\sigma_{\numtri-s}$ a bit,
e.g. by conjugating $\iv{\sigma_{\numtri-s-1}\sigma_{\numtri-s-2}}$ (and move back afterward).
Similarly as above, the set of generators is $\{\lp_s\mid1\le s\le j\}\cup G_j$,
for the L-twist along some L-arcs $\lp_s$, and the relations are precisely
$\Lambda_j$ together with
\[
    \Br^4(\lp_s,\beta_s) \quad\text{and}\quad
    \Co(\lp_s,\alpha), \alpha\in G_j\setminus\beta_s;\quad 1\le s\le j,
\]
where $\beta_s$ is the only arc in $G_j$ that intersect with $\lp_s$ (and with $\Int=1$).

%=========================================================
\paragraph{\textbf{Equivalence between presentations}}\

For $s=p$, we obtain a presentation for $\MT(\sop)$,
which compares with the presentation of $\BT(\Sigma_p)=\BT(\sop)$,
the set of generators has extra $p$ L-twists $\{\lp_s\mid1\le s\le p\}$,
the set of relations has extra relations:
\[
    \Br^4(\lp_s,\beta_s) \quad\text{and}\quad
    \Co(\lp_s,\alpha), \alpha\in G_j\setminus\beta_s;\quad 1\le s\le p.
\]
Since $\surfo\cong\surfo^0$ as decorated surfaces,
we have $\BT(\surfo)\cong\BT(\surfo^0)=\BT(\RT_0^*)=\BT(\T^*_{\CA})$.
As shown in \cite{QZ3}, the presentation of $\BT(\RT_0^*)$ is in fact derived from $\BT(\surfo)$.
Thus, the above presentation is equivalent to the one in \Cref{thm:MT's}.
\qed

%=========================================================
%=========================================================
\section{Flip graphs and categorification}\label{sec:cat}
%=========================================================
%=========================================================
%=========================================================
\subsection{Flip (exchange) graphs}\label{sec:FG}\
%=========================================================

Let $\RT$ be a triangulation of $\surf$.
For any $\gamma$ in $\RT$,
the neighbourhood $N_\RT(\gamma)$ of $\gamma$ are the union of all $\RT$-triangles containing $\gamma$.
If $\gamma$ is not a self-folded edge, there are exactly two $\RT$-triangles containing $\gamma$
and $N_\RT(\gamma)$ is a quadrilateral.
If $\gamma$ is a self-folded edge in $\RT$,
then $N_\RT(\gamma)$ is a self-folded triangle but still can be treated as quadrilateral
with two edges gluing together.

At each non self-folded edge $\gamma$ of $\RT$,
there is an operation $\mu_\gamma\colon \RT-\RT'$,
known as the \emph{usual (unoriented) flip} with respect to $\gamma$,
which produces another triangulation $\RT'$ from $\RT$ by replacing the diagonal $\gamma$ of $N_\RT(\gamma)$
by the unique other diagonal $\gamma'$ of $N_\RT(\gamma)$.

Denote by $\uEG(\surfp)$ the unoriented exchange graph of triangulations of $\surfp$,
whose vertices are triangulations and whose edges are usual flips.
The following is a well-known fact, cf. \cite{Har} or \cite[Thm.~3.10]{FST}.

\begin{theorem}\label{thm:H}
$\uEG(\surfp)$ is connected and its fundamental group is generated
by (unoriented) squares and pentagons, locally looks like the left and middle pictures of \Cref{fig:456}
(forgetting about the orientations of edges and the decorations there).
\end{theorem}

There is an oriented version of $\uEG(\surfp)$.

\begin{definition}\label{def:fflip}
Let $\RT$ be a triangulation of $\surfp$ and $\gamma$ be an open arc in $\RT$.
There is an operation $\mu^\sharp_\gamma\colon \RT\to\RT^\sharp_\gamma$,
known as the \emph{forward flip} with respect to $\gamma$,
which produces another triangulation $\RT^\sharp_\gamma$ from $\RT$
by replacing the arc $\gamma$ with $\gamma^\sharp$.
Here, $\gamma^\sharp$ is the open arc obtained from $\gamma$
by moving each of its endpoint,
anticlockwise along boundaries of $N_\RT(\gamma)$ to the next marked point or puncture.

The forward flip operation also applies to decorated triangulations of $\sop$.
The inverse operation of which is the backward flip, denoted by $\mu^\flat_\gamma=(\mu^\sharp_\gamma)^{-1}$.
We will use the notation $\pm\sharp=\mp\flat$ from now on.
\end{definition}

The three possible cases of forward flips are shown in \Cref{fig:flip}, for both $\sop$ and $\surfp$
with vertical correspondence by forgetful map.
The first two are \emph{usual forward flips}
and the last one (on the right) is the \emph{self-folded forward flip} or \emph{forward L-flip} for short. These oriented flips have been considered by many people before,
e.g. \cite{GMN} and \cite[Fig. 33/34]{IN}.

\begin{figure}[htpb]\centering
    \includegraphics[width=\textwidth]{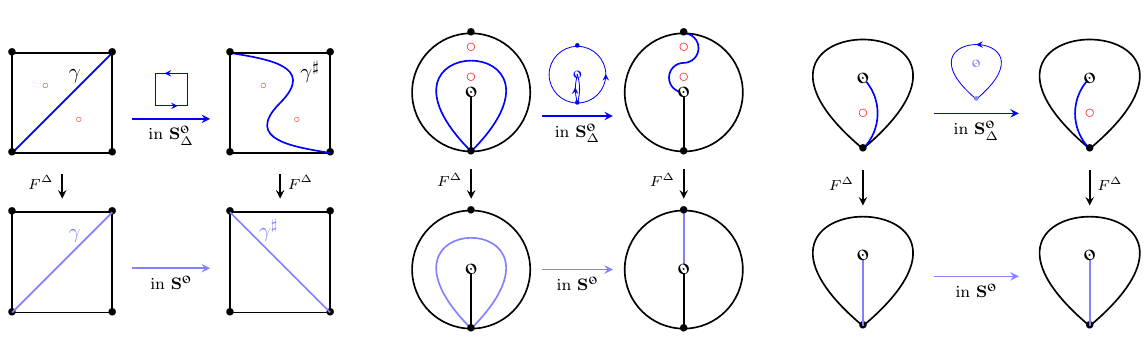}
\caption{The forward flips with forgetful map}
\label{fig:flip}
\end{figure}

Denote by $\FG(\surfp)$ and $\FG(\sop)$ the \emph{flip (exchange) graph} of (decorated) triangulations of $\surfp$ and $\sop$, respectively, where oriented edges are forward flips.
We use $\FG$ instead of $\EG$ to distinguish with the various other exchange graphs.

\begin{figure}[ht]\vskip -.2cm
\centering\makebox[\textwidth][c]{
    \includegraphics[width=14cm]{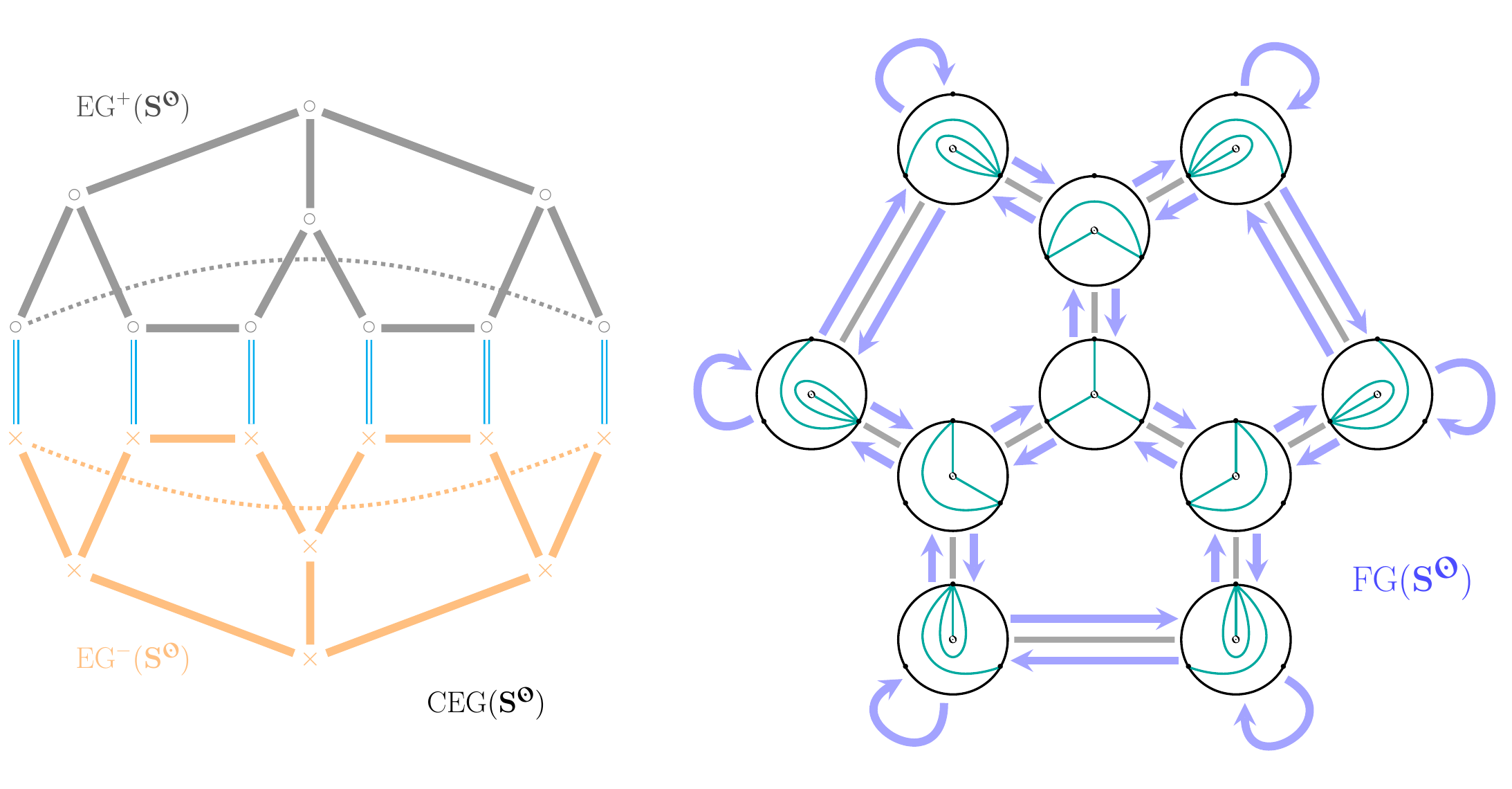}}\vskip -.2cm
\caption{Cluster exchange graph V.S. flip graph for once-punctured triangle}\label{fig:op-tri}
\end{figure}

\begin{example}
If $\surfp$ is a once punctured triangle,
then $\uEG(\surfp)$ and $\FG(\surfp)$ are shown in the right picture of \Cref{fig:op-tri},
where the unoriented/oriented version is in gray/purple, respectively.
\end{example}

\begin{remark}\label{rem:FG}
When there is no punctures, $\FG(\surfp)$ can be obtained from $\uEG(\surfp)$ by replacing each unoriented edge with a 2-cycle, as we did in \cite{KQ2}.
However, with punctures, $\uEG(\surfp)$ is not $n$-regular (i.e. there are $n$ edges at each vertex).
We fix this by including forward L-flip, which is a loop in $\FG(\surfp)$,
so that $\FG(\surfp)$ is still $(n,n)$-regular
(i.e. there are $n$ edges going out and in, respectively, at each vertex).
Moreover, one can obtained $\uEG(\surfp)$ from $\FG(\surfp)$ by deleting loop edges and replacing every 2-cycle by an unoriented edge.

%In \Cref{sec:MSx}, we will study the flip exchange graph with cluster style tagging at some of the punctures.
%For instance, the left picture of \Cref{fig:op-tri} is the cluster exchange graph associated to $\surfp$
%(the famous associahedron).
%To distinguish with the various exchange graphs, we will use flip graphs from now on.
\end{remark}

As in the undecorated case, there a natural covering.
\begin{lemma}\label{lem:covering}
The forgetful map induces a Galois covering $F^\Tri_*:\FG(\sop)\to\FG(\surfp)$ with covering group $\SBr(\sop)$.
\end{lemma}
\begin{proof}
Notice that although some of the edges are loops, the two graphs are both $(n,n)$-regular.
The rest follows the same as \cite[Lem.~3.8]{KQ2}.
\end{proof}

Usually, $\FG(\sop)$ is not connected.
The more tricky business is to determine the covering group when restricted to a connected component.
In the unpunctured case, this is stated properly as \cite[Thm.~4.8]{BMQS}, cf \cite[Rem.~3.10]{QQ}.
In the punctured case, it is even harder (as the cluster theory is not applicable).

Fix any $\T\in\FG(\sop)$.
Denote by $\FGT(\sop)$ the connected component of $\FG(\sop)$ containing $\T$.
So there is a covering
\begin{gather}\label{eq:FGT}
  F^\Tri_*\mid_{\T}\colon\FGT(\sop)\to\FG(\surfp),
\end{gather}
whose covering group is denoted by $\MTT\le\SBr(\sop)$.
We will show that $\MTT$ is in fact $\MT(\sop)$ in \Cref{cor:mt=mt},
which implies the following (using \Cref{thm:QZ+}).

\begin{corollary}\label{cor:pi0}
The connected components of $\FG(\sop)$ is parameterized by $$\SBr(\sop)/\MT(\sop)\cong \Ho{1}(\surf).$$
\end{corollary}

%=========================================================
\subsection{The 3-Calabi-Yau category $\Dsop$}\label{sec:D3}\
%=========================================================

Recall that we assume that \eqref{eq:aleph} which ensures that $\sop$ admits decorated triangulations.
We will fix a triangulation $\T$,
which induces a natural grading on $\sop$ induced by $\T$, cf. \cite[\S~2]{CHQ} for more details.
Note that any triangulation and the dual graph of which carries natural gradings.
Denote by $\wh{\on{CA}}(\sop)$ the set of graded arcs whose underlying arcs are in $\CA(\sop)$. Similarly for $\wh{\on{LA}}(\sop)$ or a graded arc $\wh{\alpha}$.

For the preliminaries on simple tilting theory, see \Cref{sec:prel}.

Let $\Gamma_\T^\sun=\Gamma(Q_\T^\sun,W_\T^\sun)$ be the \emph{Ginzburg algebra} associated to the quiver with potential $(Q_\T^\sun,W_\T^\sun)$
and $\Dsop$ the \emph{finite dimensional derived category} of $\Gamma_\T^\sun$.

The following is well-known (cf. \cite{KVB}):
$\Dsop$ is 3-Calabi-Yau with a canonical heart $\h_\T^\sun$ being (dg) module category of $\Gamma_\T^\sun$.
Denote by $\Sph\Dsop$ the set of \emph{(reachable) spherical objects} in $\Dsop$
and $\EGT(\Dsop)$ the principal component of exchange graph of hearts of $\Dsop$.
Here is a summery of the results in \cite{Chr,CHQ}, specified to our category.

\begin{theorem}\cite{CHQ}\label{thm:CHQ1}
There is an injective map:
\begin{equation}\label{eq:X}
\begin{array}{rcl}
  \Xto\colon \wh{\on{CA}}(\sop)\bigsqcup\wh{\on{LA}}(\sop)&\to& \on{Obj}\Dsop,  \\
    \wh{\alpha} &\mapsto& \Xto_{\wh{\alpha}},
\end{array}
\end{equation}
satisfying
\begin{equation}\label{eq:Int}
    \Int^d(\wh{\alpha},\wh{\beta})=\dim\Hom^d(\Xto_{\wh{\alpha}},\Xto_{\wh{\beta}}).
\end{equation}

Moreover, $\Xto$ induces an isomorphism of oriented graphs:
\begin{equation}\label{eq:iso}
\begin{array}{rcl}
  \Xto_*\colon  \FGT(\sop) &\xrightarrow{\cong}& \EGT(\Dsop),  \\
    \T &\mapsto& \h_\T^\sun,\\
\end{array}
\end{equation}
such that $\Sim\h_\T^\sun=\Xto(\wh{\T^*})$.
\end{theorem}

In particular, for $\wh{\lp}\in\wh{\LA}(\sop)$,
$\Ext^1(\Xto_{\wh{\lp}},\Xto_{\wh{\lp}})=\k$ and when restricted to $\wh{\CA}(\sop)$,
one has an injection
\begin{equation}\label{eq:X2}
\begin{array}{rcl}
  \Xto \colon \wh{\CA}(\sop) &\xrightarrow{1-1}& \Sph^\circ(\Dsop),  \\
    \wh{\eta} &\mapsto& \Xto_{\wh{\eta}},
\end{array}
\end{equation}
Thus, $\Xto_{\wh{\alpha}}$ is rigid if and only if $\alpha$ is in $\CA(\sop)$.

Denote by $\FGC(\sop)$ the subgraph of $\FG(\sop)$ by deleting all L-flip edges.
Equivalently, it consists of flips $\mu^\sharp_\gamma:\T\to\T^\sharp$
such that the dual closed arc of the flipped open arc $\gamma$ is in $\CA(\sop)$ (not in $\LA(\sop)$).
Similarly for other variation of $\FGC$.
We get the following corollary.

\begin{corollary}
The isomorphism $\Xto_*$ in \eqref{eq:iso} restricts to an isomorphism
\begin{equation}\label{eq:iso2}
  \Xto_*\colon  \FGC^\circ(\sop) \xrightarrow{\cong} \REGp(\Dsop).
\end{equation}
\end{corollary}
Note that for any $\T\in\FGC(\sop)$, there are exactly $\#\T^*_{\CA}$ edges coming in and going out.

Two of the key features of the categorical side are
\begin{itemize}
  \item the symmetry triangulated shift $[1]$ and
  \item the natural partial order on hearts.
\end{itemize}
Under the correspondences in \Cref{thm:CHQ1},
we can also use $[1]$ and $\le$ on the topological side for decorated triangulations of $\sop$.

\begin{remark}
For closed arcs, the shift correspond to grading shift.
The map $X$ also extends to open arc and the alternative version of \eqref{eq:iso} is that
$X(\T)$ is in fact a silting set in the perfect derived category $\per(\sop)$,
which is the Koszul dual of $\h_\T^\sun$ in $\Dsop$.
See \cite{QQ2,Chr,CHQ} for more discussion.

On open arcs, the shift corresponds to the \emph{universal rotation} $\wh{\rho}\in\MCG(\sop)$.
More precisely, for any boundary component $\partial\in\partial\sop$,
there is a rotation/operation $\rho_\partial$ that moves each marked point to the next consecutive one:
\[\begin{tikzpicture}[scale=.3]
\draw[thick](0,0) circle (5) ;
\draw[thick,fill=gray!20](0,0) circle (1.5);
  \foreach \j in {1,...,5}{\draw(90-72*\j:5) node{$\bullet$};}
  \foreach \j in {1,...,3}{\draw(-90-120*\j:1.5) node{$\bullet$};}
\draw[teal,-stealth](90+36-15:5.4)arc(90+36-15:90+36+15:5.4);
\draw[teal](126:6)node{$\rho_1$}(90:.5)node{$\rho_2$};
\draw[teal,-stealth](90+36:1.1)arc(90+36:90+36-70:1.1);
\end{tikzpicture}
\label{fig:rotate}\]
These rotations are roots of the Dehn twists, i.e. $\rho_\partial^{m_\partial}=\on{D}_\partial$.
And the universal rotation is the product of all:
\begin{equation}\label{eq:uni-rho}
   \wh{\rho}\colon=\prod_{\partial\subset\partial\sop} \rho_\partial.
\end{equation}
In particular, we have
\[
    \wh{\rho} \big( X(\T) \big) = X(\T)[1].
\]
Notice that in each long dumbbell/symmetric hexagon in \eqref{fig:4hex},
the source triangulated digon is the universal rotation of the sink triangulated digon.
\end{remark}

%=========================================================
\subsection{The fundamental domain}\label{sec:pi0}\
%=========================================================

Denote by $\REGp[\h_0,\h_0[1]]$ the connected component in $\REG[\h_0,\h_0[1]]$ containing $\h_0$.
By the isomorphism \eqref{eq:iso}, we have
\[
    \FGCp[\T_0,\T_0[1]]=(\Xto_*)^{-1}( \REGp[\h_0,\h_0[1]] )
\]
for $\T_0=(\Xto_*)^{-1}(\h_0)$.

Denote by $\dbloop{\FGCp[\T_0,\T_0[1]]}$ the graph obtained from $\FGCp[\T_0,\T_0[1]]$ by adding a reverse edge for each of its edges and adding $\#(\T')^*_{\LA}=\#\yue(\T')$ loops at each vertex $\T'$.

\begin{theorem}\label{thm:pi0}
Suppose that $\T_0$ is admissible.
Then there is a natural isomorphism
\begin{gather}
    \FGT(\sop)/\MT(\sop)\cong\dbloop{\FGCp[\T_0,\T_0[1]]}.
\end{gather}
\end{theorem}
\begin{proof}
Firstly, we claim that $\MT(\sop) \cdot \FGCp[\T_0,\T_0[1]]\subset\FGT(\sop)$.
This is because the braid twist/L-twist $\MT(\sop)$ can be realized as either a composition of flips or a flip on its own.
More precisely, for any $\T'\in \FG(\sop)$, one has (cf. $u^2$ and $s$ in \Cref{fig:4hex}, respectively)
\begin{gather}\label{eq:twist=flip2}
  (\T')_\gamma^{\pm2\sharp}=B_{\eta}^\mp(\T'), \\
  (\T')_\gamma^{\pm\sharp}=L_{\lp}^\mp(\T'). \label{eq:twist=flip}
\end{gather}
Here, $\eta$ or $\lp$ in $(\T')^*$ is the dual of the arc $\gamma\in\T'$
and $(-)^{\pm2\sharp}$ means the iterated flips.
Then one can write down explicitly that the twist $\Psi(T_0)$ can be expressed as a sequence of flips staring from $T_0$ in the same way as the formula \cite[(8.4)]{KQ1}, for any $\Psi\in\MT(\T_0)=\MT(\sop)$.

Secondly, we claim that $\MT(\sop) \cdot \FGCp[\T_0,\T_0[1]]\supset\FGT(\sop)$.
We only need to how that if $\T'=\Xi(\T_1)$ is in $\MT(\sop) \cdot \FGCp[\T_0,\T_0[1]]$ for $\Xi\in\MT(\sop)$ and $\T_1\in \FGCp[\T_0,\T_0[1]]$,
then so are the any flips of $\T'$.

Equivalently, we only need to show that any flips of $\T_1$ are still in $\MT(\sop) \cdot \FGCp[\T_0,\T_0[1]]$ if $\T_1$ is.
This follows as below:
\begin{itemize}
  \item For any $\eta\in(\T_1)^*_{\CA}$ with $\gamma=\eta^*\in\T_1$,
  exactly one of the flips $(\T_1)^{\pm\sharp}_{\gamma}$ is in $\FGCp[\T_0,\T_0[1]]$ by \Cref{lem:ss} (via the identification \eqref{eq:iso2}).
  The other flip will be one of the braid twist $B_\eta^\pm$ of which by formula \eqref{eq:twist=flip2}.
  \item For any $\lp\in\T^*_{\LA}$ with $\gamma=\eta^*\in\T_1$,
  the flips $(\T_1)^{\pm\sharp}_{\gamma}$ are precisely the L-twists by \eqref{eq:twist=flip}.
\end{itemize}
Hence $\MT(\sop) \cdot \FGCp[\T_0,\T_0[1]] = \FGT(\sop)$ as sets.

Finally, the formulae \eqref{eq:twist=flip2} implies that each edge in $\FGT(\sop)$ needs to be doubled and \eqref{eq:twist=flip} implies that there are certain loops need to be added, so that
one matches $\dbloop{\FGCp[\T_0,\T_0[1]]}$ with $\MT(\sop)$ precisely as graphs.
\end{proof}

Recall that there is a covering \eqref{eq:FGT} so that
\begin{gather}\label{eq:MTT}
    \FGT(\sop)/\MTT=\FG(\surfp).
\end{gather}
More precisely, $\MTT$ is the subgroup of $\SBr(\sop)$ consisting of elements $f\in\SBr(\sop)$ satisfying $f(\T)\in\FGT(\sop)$.
Thus, $\MT(\sop)\le\MTT$ and, by the conjugation formula \eqref{eq:Psi}, $\MT(\sop)\lhd\MTT$.
But at the moment, we only have the following commutative diagram,
\begin{equation}\label{eq:3FGs}
\begin{tikzpicture}[xscale=3,yscale=1,baseline=(bb.base)]
\path (0,1) node (bb) {}; % baseline
\draw (-1,0) node (s0) {$\FGT(\sop)$}
 (1,-1) node (s1) {$\FG(\surfp)$}
 (1,1) node[] (s2) {$\dbloop{\FGCp[\T_0,\T_0[1]]}$};
\draw [-stealth, font=\scriptsize]
 (s0) edge node [below] {$\big/ \MTT$} (s1)
 (s0) edge node [above] {$\big/ \MT(\sop)$} (s2)
 (s2) edge node [right] {$\on{Ni}$} (s1);
\end{tikzpicture}
\end{equation}
for any admissible $\T_0$, where $\on{Ni}=\MTT/\MT(\sop)$.

In \Cref{sec:pi1}, we will show that $\MTT=\MT(\sop)$ and $\FG(\surfp)=\dbloop{\FGCp[\T_0,\T_0[1]]}$.

%=========================================================
%=========================================================
\section{Flip groupoids and flip twist groups}\label{sec:fgoupoids}
%=========================================================
%=========================================================

In \cite{KQ2}, we introduce the cluster exchange groupoids as the enhancement of cluster exchange graphs,
using the cluster combinatorics to produce generalization of braid groups.
Naturally, in the unpunctured surface case, as well as Dynkin case, we identified
cluster braid groups with braid twist groups and spherical twist groups.

In this section, we introduce the flip groupoids as refinements of Cayley graphs of mixed twist groups,
with the same nature above.
However, the square/pentagon/dumbbell relations there are not enough here.
New phenomenon appears due to the loop edges in $\FG(\surfp)$.

%=========================================================
\subsection{Relations in cluster exchange graphs}\label{sec:fg0}\
%=========================================================

%=========================================================
\paragraph{\textbf{Classical square and pentagon (relations)}}\

Given a triangulation $\RT$ in $\FG(\surfp)$ with any two vertices $i, j\in Q_\RT$,
there is the following local full subgraph in $\FG(\surfp)$:
\begin{gather}\label{eq:local EG}
\begin{tikzpicture}[scale=1.5, rotate=0,
          arrow/.style={->,>=stealth,bend left=25,very thick}]
    \clip(-1.2,-.5)rectangle(2.2,.5);
    \draw[orange] (1,0) node[] (t3) {{$\RT'$}} (2,0) node (t4) {{$\bullet$}} (-1,0) node(t1) {{$\bullet$}}
        (0,0)node[](t2) {{$\RT$}};
    \draw[font=\scriptsize](0.5,0) node {{$i$}} (-.5,0) node {{$j$}} (1.5,0) node {{$j$}};
    \foreach \j/\k in {1/2,2/3,3/4}{
    \draw[Emerald,arrow] (t\j) edge node[above]{$x$} (t\k);
    \draw[blue!50,arrow] (t\k) edge node[below]{$y$}(t\j);}
\end{tikzpicture}.
\end{gather}
Suppose that there is no arrow from $i$ to $j$.
If there is at most one arrow from $j$ to $i$ in $Q_\RT$,
then such a subgraph completes into a double square or a double pentagon,
depending on if the numbers of arrow is zero or one, as shown below.
\begin{gather}\label{eq:Sqr-Pen}
        \begin{tikzpicture}[scale=.7, rotate=0,
          arrow/.style={->,>=stealth,very thick}]
        \begin{scope}[shift={(0,0)},rotate=45]
        \foreach \j in {1,2,3,4}{
            \draw[orange] (90*\j+45:2) node (t\j) {{$\bullet$}};
            \draw[blue!50] (90*\j:2) node {{$x$}};
            \draw[Emerald] (90*\j:.8) node {{$y$}};}
        \foreach \a/\b in {1/2,2/3,3/4,4/1}{
            \draw[Emerald] (t\a) edge[arrow,bend left=15]  (t\b);
            \draw[blue!50] (t\b) edge[arrow,bend left=15]  (t\a);};
            \draw[white,fill=white](135:2)circle(.1) node[orange]{$\otimes$}node[left,orange,font=\small]{$\RT=\;$};
            \draw[white,fill=white](-45:2)circle(.1) node[orange]{$\odot$};
        \end{scope}
            \draw[](3.5,-.4)node{\huge{\; $\Big/_{x^2=y^2\;,}$}};
%            \draw[font=\scriptsize](0,1.414)node{$a$} (-1.414,0)node{$b$};
        \begin{scope}[shift={(10,0)},rotate=90]
        \foreach \j in {1,2,3,4,5}{
            \draw[orange] (72*\j:2) node (t\j) {{$\bullet$}};
            \draw[blue!50] (72*\j+36:2.1) node {{$x$}};
            \draw[Emerald] (72*\j+36:1) node {{$y$}};}
        \foreach \a/\b in {1/2,2/3,3/4,4/5,5/1}{
            \draw[Emerald] (t\a) edge[arrow,bend left=15]  (t\b);
            \draw[blue!50] (t\b) edge[arrow,bend left=15]  (t\a);};
            \draw[](-.4,-3.8)node{\huge{\; $\Big/_{x^2=y^3\;.}$}};
            \draw[white,fill=white](72:2)circle(.1) node[orange]{$\otimes$}node[left,orange,font=\small]{$\RT=\;$};
            \draw[white,fill=white](-72:2)circle(.1) node[orange]{$\odot$};
        \end{scope}
        \end{tikzpicture}
\end{gather}
The \emph{square/pentgon relation} (at $\RT$ with respect to $i,j$)
is the relations $x^2=y^h$ staring at $\RT$ (where $h=2$ for square and $h=3$ for pentagon).
In fact, by the symmetry, all $x^2=y^h$ are this type of relations.

Note that these relations are well-known, cf. \cite{Kr,Q1}.
They are not exactly the ones in \Cref{thm:H}, but the oriented version.

\begin{figure}[ht]
  \centering\makebox[\textwidth][c]{
    \includegraphics[width=14cm]{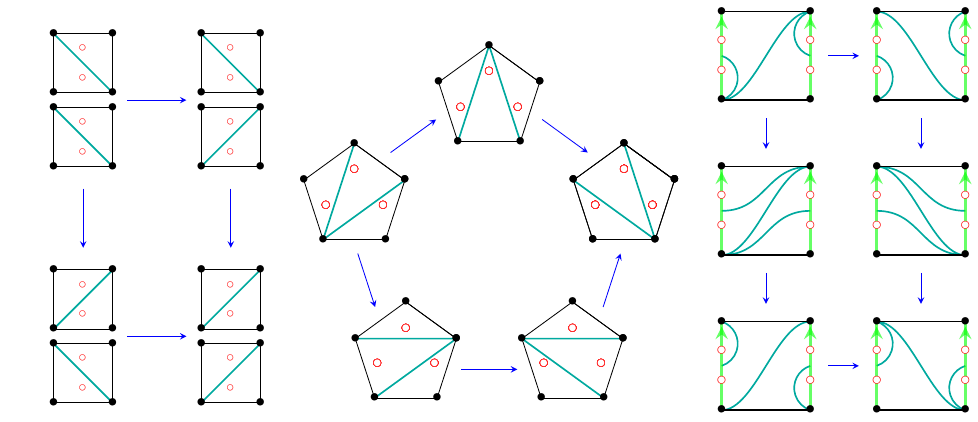}   \qquad    }
\caption{Square, pentagon and S-DB in $\FG(\sop)$}
\label{fig:456}
\end{figure}

%=========================================================
\paragraph{\textbf{Original short dumbbell (S-DB) relations}}\cite[\S~2]{KQ2}\

Given a triangulation $\RT$ in $\FG(\surfp)$ with two vertices $i,j$ in $Q_\RT$
such that the corresponding open arcs in $\RT$ are not self-folded edges.
Suppose that there is no arrow from $i$ to $j$ as above and we still have \eqref{eq:local EG}.
We name the edges/flips $x$ and $y$ (omitting the subscripts for simplicity).
When there are arrows in one direction (between $i$ and $j$),
then $x$ and $y$ can be distinguished as being forward mutation at either the tail or head of arrows.
\begin{gather}\label{eq:fat DB}
\begin{tikzpicture}[scale=1.5,arrow/.style={->,>=stealth,very thick}]
    \clip(-1.2,-.7)rectangle(2.2,.7);
    \draw[orange] (1,0) node (t3) {{$\bullet$}} (2,0) node (t4) {{$\bullet$}} (-1,0) node(t1) {{$\bullet$}}
        (0,0)node[white](t2) {{$\bullet$}};\draw(t2)node[orange]{$\RT$};
    \draw[Emerald,arrow] (t2.30) to node[above]{$x$} (t3.150);
    \draw[Emerald,arrow] (t3.60) to[bend left=60] node[above]{$x$} (t4.90);
    \draw[Emerald,arrow] (t1.-90) to[bend left=-60] node[below]{$x$} (t2.-120);
    \draw[Emerald,arrow] (t2.-30) to node[below,white]{$x$} (t3.-150);
    \draw[blue!50,arrow] (t4.-90) to[bend left=60] node[below]{$y$} (t3.-60);
    \draw[blue!50,arrow] (t2.120) to[bend left=-60] node[above]{$y$} (t1.90);
\end{tikzpicture}.
\end{gather}
The \emph{short dumbbell (S-DB) relation} (at $\RT$ with respect to $i,j$) is the hexagonal relations shown in \eqref{eq:fat DB},
or $x^2y=yx^2$ staring at $\RT$.
Let $t_i=yx$ and $t_j=xy$ be the \emph{flip twists (i.e. 2-cycle)} at $\RT$ and $\RT'$, respectively.
Then the relation can be written as conjugation relation $t_j=t_i^x$.

%=========================================================
\subsection{Extra relations in the punctured case}\label{sec:fg}\
%=========================================================

\begin{figure}[ht]\centering
 \includegraphics[]{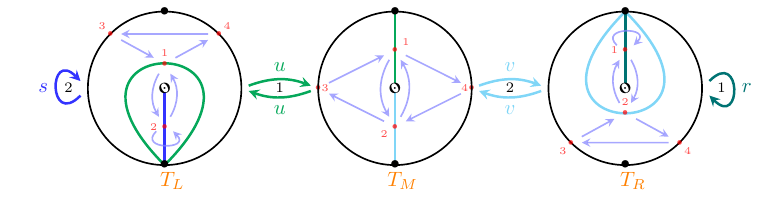}
\caption{$\FG(\surfp)$ for a once-punctured digon as local flip graph}\label{fig:D2}
\end{figure}

Given a triangulation $\RT$ in $\FG(\surfp)$ with two vertices $i,j$ in $Q_\RT$
such that they are in a digon as shown in the middle picture of \Cref{fig:D2}.
They there is a full subgraph of $\FG(\surfp)$ as shown there.
We name three triangulations $\RT_L,\RT_M$ and $\RT_R$, respectively (cf. \Cref{fig:D2} and the third picture of \Cref{fig:Hex}).
We also name usual flips/edges $u,v$ and L-flips/loop edges $s,r$ as shown there,
where the color matches the open arc that they flip at.

%=========================================================
\paragraph{\textbf{Long dumbbell (L-DB) relations}}\

The \emph{long dumbbell relations} (at $\RT_M$ with respect to $i,j$)
are
\begin{itemize}
  \item $r=s^{uv}$ ($\Leftrightarrow suv=uvr$) at $\RT_L$ and
  \item $s=r^{vu}$ ($\Leftrightarrow vus=rvu$) at $\RT_R$,
\end{itemize}
as shown in the left two picture of \Cref{fig:Hex}.

\begin{figure}[ht]\centering
    \includegraphics[width=\textwidth]{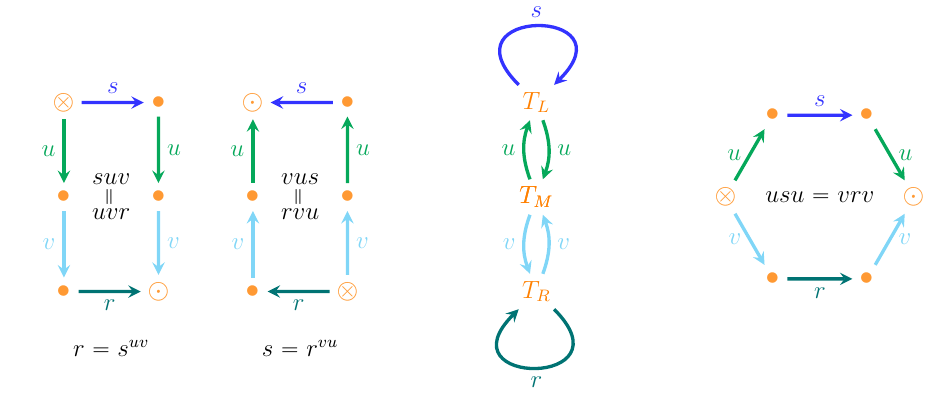}\vskip -.5cm
\caption{Two L-DB and a Sym-Hex}\label{fig:Hex}
\end{figure}

\paragraph{\textbf{Symmetric hexagon (Sym-Hex) relations}}\

Continue the set up as above (for long dumbbell relation).

The \emph{symmetric hexagon relations} (at $\RT_M$ with respect to $i,j$) is $usu=vrv$,
as shown in the right picture of \Cref{fig:Hex}.
Note that they will also play a role in the conjugation formula later.

\begin{definition}\label{def:fg}
The \emph{flip groupoid} $\fg(\surfp)$ is the quotient of
the path groupoid of $\FG(\surfp)$
by all (square/pentagon, S-DB, L-DB, Sym-Hex) relations above.
\end{definition}

These extra relations have been observed by some experts, cf. \cite{Hai,HKQ,IN}.

%=========================================================
\subsection{Covering groupoids}\label{sec:covering}\
%=========================================================

Similar to $\fg(\surfp)$, we have the groupoid version of $\FG(\sop)$.
\begin{figure}[ht]
  \centering\makebox[\textwidth][c]{
    \includegraphics[width=16cm]{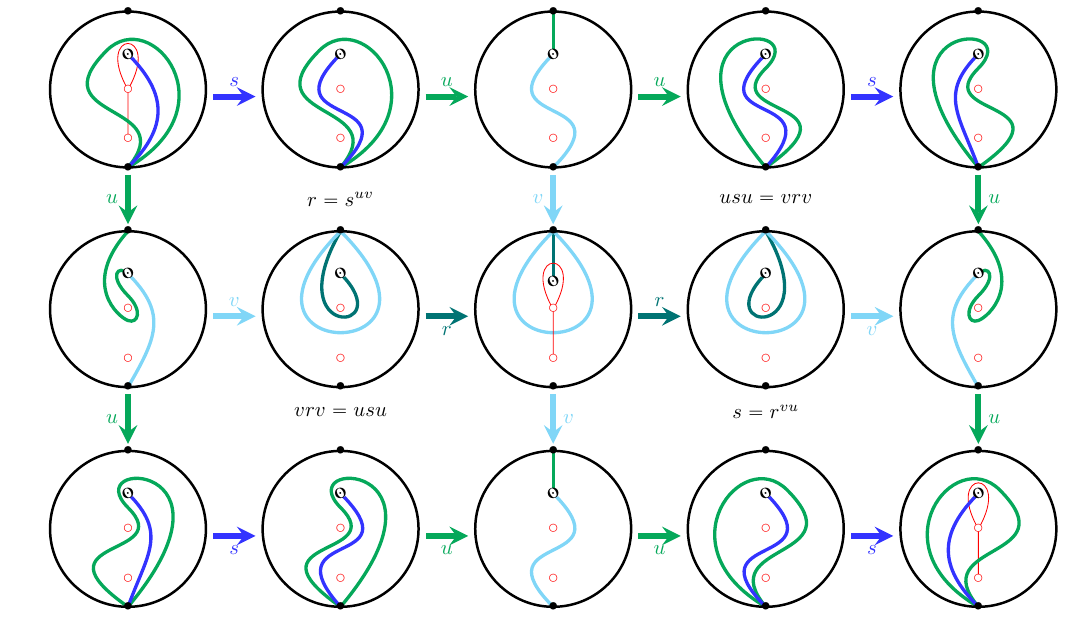}       \quad}
\caption{Higher braid relation $\eta\delta\eta\delta=\delta\eta\delta\eta$ decomposes into 4 hexagons}
\label{fig:4hex}
\end{figure}

\begin{definition}\label{def:gp.sop}
There are the following lifts in $\FG(\sop)$:
\begin{itemize}
  \item A square is shown in the left picture of \Cref{fig:456}.
  \item A pentagon is shown in the middle picture of \Cref{fig:456}.
  \item An S-DB is shown in the right picture of \Cref{fig:456}.
  \item Two L-DB are shown as the top-left/bottom-right hexagons in \Cref{fig:4hex}.
  \item Two Sym-Hex are shown as the top-right/bottom-left hexagons in \Cref{fig:4hex}.
\end{itemize}
The \emph{flip groupoid} $\fg(\sop)$ is the quotient of
the path groupoid of $\FG(\sop)$ by the five kinds of relations above.
\end{definition}

As the relations in $\fg(\surfp)$ lifts to relations in $\fg(\sop)$,
the isomorphism in \Cref{thm:pi0} upgrades to an isomorphism between groupoids:
\[
    F^\Tri_*\colon\fgt(\sop)/\MTT\cong\fg(\surfp).
\]

%=========================================================
\subsection{Local flip twists}\label{sec:FT}\
%=========================================================

There are two types \emph{(local) flip twists} at a triangulation $\RT\in\FG(\surfp)$:
the \emph{2-cycle (twist)} $t_i$ corresponding flipping twice a non self-folded edge $i$ of $\RT$
(as mentioned above) and the \emph{loop (twist)} $l_s=s$ corresponding flipping once a self-folded edge $s$ of $\RT$.
In total, there are $n$ flip twists at each $\T$.
\begin{definition}\label{def:LT}
Define the \emph{flip twist group} $\FT(\RT)$ to be
the subgroup of $\fg(\surfp)$ generated by flip twists at $\RT$.
\end{definition}

\begin{lemma}\label{lem:decompo}
The following (higher) braid relations hold in any local flip group $\FT(\RT)$:
\begin{itemize}
  \item $\Co(t_i,t_j)$ between two 2-cycles if there is no arrow between $i$ and $j$ in $Q_\RT$.
  \item $\Br(t_i,t_j)$ between two 2-cycles if there is exactly 1 arrow between $i$ and $j$ in $Q_\RT$.
  \item $\Co(l_s,l_r)$ between two loops if there is no arrow between $s$ and $r$ in $Q_\RT$.
  \item $\Co(t_i,l_s)$ between a 2-cycle and a loop if there is no arrow between $i$ and $s$ in $Q_\RT$.
  \item $\Br^4(t_i,l_s)$ between a 2-cycle and a loop if there are two arrows between $i$ and $s$ in $Q_\RT$
  (that forms a two-cycle).
\end{itemize}
\end{lemma}
\begin{proof}
The first two relations are shown in \cite[Lem.~2.7, cf. Fig.~2]{KQ2}.
The next two commutation relations follow the same way as the first commutation relation,
(as the flips are apart).

We only need to check the last higher braid relation.
It follows by assembling 2 long dumbbells and 2 symmetric hexagons
as shown in the left picture of \Cref{fig:decomposition}.
More precisely, at $\T_s$ in \Cref{fig:Hex,fig:Hex}, we have $l_s=s$ and $t_i=u^2$.
The corresponding algebraic calculation is:
\[\begin{array}{rl}
    suusuu=&(suv)\cdot\iv{v}\cdot(usu)\cdot u\\
    =&(uvr)\cdot\iv{v}\cdot(vrv)\cdot u\\
    =&uvrrvu\\
    =&u\cdot(vrv) \cdot\iv{v} \cdot (rvu)\\
    =&u\cdot(usu) \cdot\iv{v} \cdot (vus)=uusuu.
\end{array} \qedhere \]
\end{proof}

\begin{remark}
The commutation, braid and higher braid relations decomposes into 4 squares, 6 pentagons and 4 hexagons, respectively.
\Cref{fig:6pen} and \Cref{fig:4hex} shows the lifting of braid and higher braid relations in $\fg(\sop)$, respectively.
\begin{figure}[ht]
  \centering\makebox[\textwidth][c]{
    \includegraphics[width=17cm]{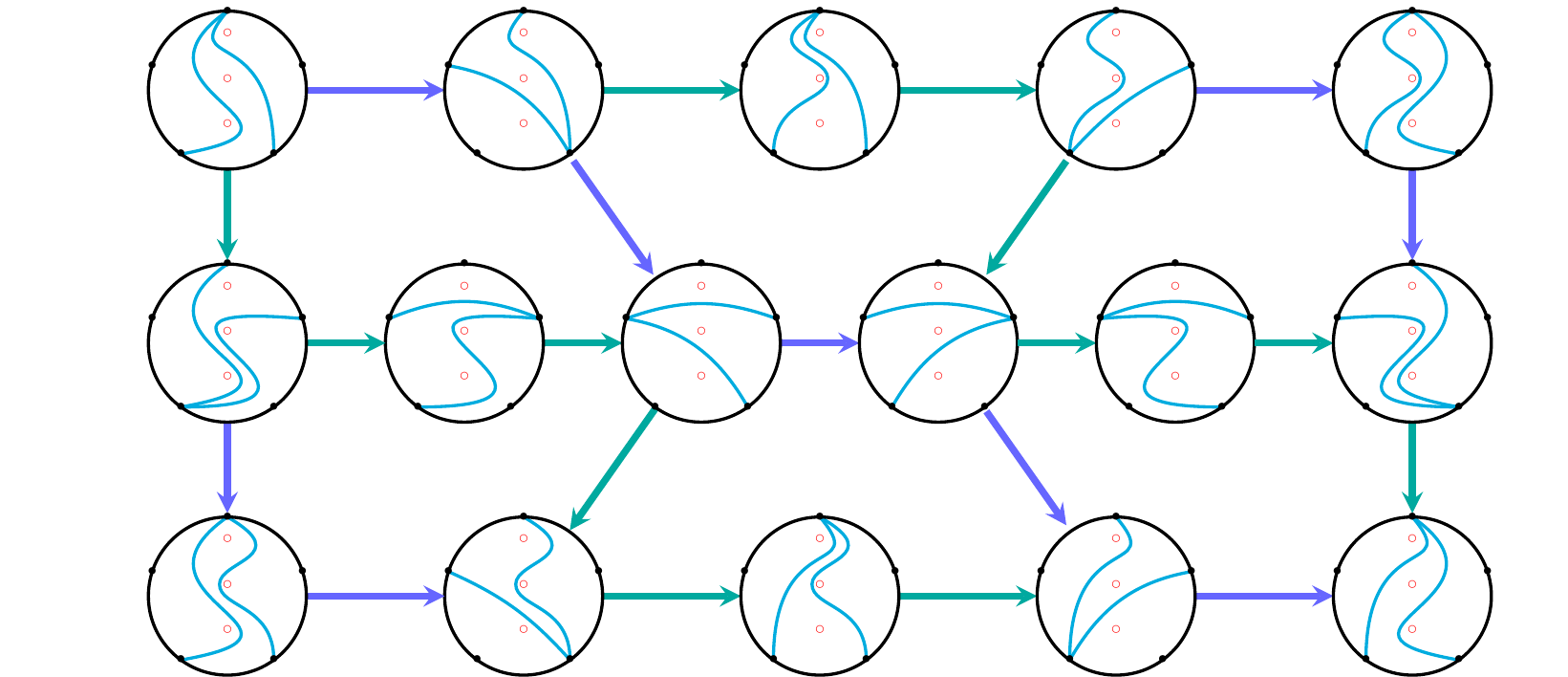}       \qquad}
\caption{Braid relation decomposes into 6 pentagons}
\label{fig:6pen}
\end{figure}
\end{remark}

\begin{figure}[ht]\centering
    \includegraphics[width=\textwidth]{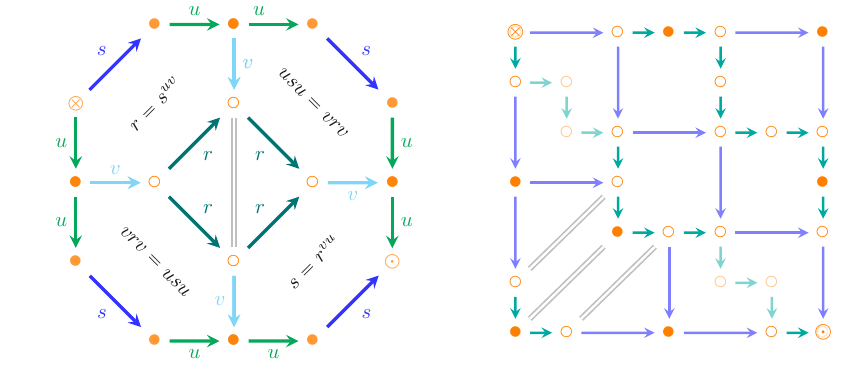}
\caption{Two ways of decomposing of higher braid relation $\Br^4$.\\
Left: 2 L-DB and 2 Sym-Hex;\quad Right: 8 Non-Sym-Hex}
\label{fig:decomposition}
\end{figure}

\begin{remark}[Alternative decomposition of $\Br^4$ via non-symmetric hexagon relations]\

There is another type of (non--symmetric) hexagon relation, appearing in the cluster braid groupoids of type $I_2(4)=B_2=C_2$ Dynkin diagrams,
see the picture below.
\begin{gather}\label{eq:h}
        \begin{tikzpicture}[scale=.7, rotate=0,
          arrow/.style={->,>=stealth,very thick}]
        \begin{scope}[shift={(0,0)},rotate=90]
        \foreach \j in {1,2,3,4,5,6}{
            \draw[orange] (60*\j:2) node (t\j) {{$\bullet$}};
            \draw[blue!50] (60*\j+30:2.2) node {{$x$}};
            \draw[Emerald] (60*\j+30:1.2) node {{$y$}};}
        \foreach \a/\b in {1/2,2/3,3/4,4/5,5/6,6/1}{
            \draw[Emerald] (t\a) edge[arrow,bend left=15]  (t\b);
            \draw[blue!50] (t\b) edge[arrow,bend left=15]  (t\a);};
            \draw[white,fill=white](60:2)circle(.1) node[orange]{$\otimes$}node[left]{$\RT=\;$};
            \draw[white,fill=white](-60:2)circle(.1) node[orange]{$\odot$};
        \end{scope}
            \draw[](3.8,-.4)node{\huge{\; $\Big/_{x^2=y^4}$ .}};
        \end{tikzpicture}
\end{gather}
These relations clearly different from our new long dumbbell/symmetric hexagon relations.
Nevertheless, they both produce $\Br^4$ relation in different fashion.
See the right picture of \Cref{fig:decomposition} that $\Br^4$ relation decomposes into 8 non-symmetric hexagons
(which is the same style as \Cref{fig:6pen}).

In general, there are polygonal (i.e. $(h+2)$-gon) relations (cf. \cite{HHQ}) of the form $x^2=y^h$,
which essentially come from weighted folding for type $I_2(h)$ quiver/diagram
and produces $\Br^h$ relations.
\end{remark}

Recall that we have chosen an initial triangulation $\T\in\FG(\sop)$.

By \eqref{eq:MTT}, we have a short exact sequence:
\begin{equation}\label{eq:ses MTT}
    1\to \pi_1\fgt(\sop,\T) \to \pi_1 \fg(\surfp,\RT) \xrightarrow{\man_\RT} \MTT \to 1
\end{equation}
for $\RT=F_*^\Tri(\T)$.
By formula \eqref{eq:twist=flip2} and \eqref{eq:twist=flip}, there is a natural homomorphism
\begin{equation}\label{eq:MT=FT}
\begin{array}{rcl}
    \man_\RT\colon \FT(\RT) &\twoheadrightarrow& \MT(\T) \; \big(\lhd \MT(\sop) \lhd \MTT \big)\\
        t_i &\mapsto& B_\eta, \\
        l_j &\mapsto& L_\delta,
\end{array}
\end{equation}
for $\eta\in\T^*_{\CA}$ dual to a non self-fold arc in $\RT$ that corresponds to vertex $i$ in $Q_\RT$ and,
for $\delta\in\T^*_{\LA}$ dual to a self-fold arc in $\RT$ that corresponds to vertex $j$ in $Q_\RT$.

%=========================================================
%=========================================================
%\section{Properties of flip twist groups}\label{sec:ftg}
%=========================================================
%=========================================================
%We study relations between flip twist groups associated to different triangulations and their relations with the corresponding mixed twist groups in this section.
\def\CFT{\CT}

%=========================================================
\subsection{Conjugation formulae}\label{sec:conj.FT}\
%=========================================================

%=========================================================
\paragraph{\textbf{2-cycle twist groups}}\

Let $\RT$ be a triangulation of $\surfp$ in $\FG(\surfp)$ and $P\in\yue(\RT)$ a $\RT$-isolated puncture,
locally shown as $\RT_L$ in \Cref{fig:D2}.
There is a loop twist $s$ associated to the self-folded edge (or vertex $2$ there)
and a 2-cycle twist $u^2$ associated to the enclosing edge.
Call the conjugated twist
\[
    (u^2)^{\iv{s}}=su \cdot u\iv{s} \xlongequal{usu=vrv} \iv{u}vrv \cdot u\iv{s}
        \xlongequal{r=s^{\iv{uv}}}  \iv{u}vvu=(v^2)^u
\]
the \emph{2-cycle (twist)} associated to
the loop $s$ at $\RT$ of $\fg(\surfp)$.
Then to each triangulation of $\fg(\surfp)$, there are $n$ associated 2-cycle twists.

\begin{definition}
The 2-cycle twist group $\CFT(\RT)$ is the subgroup of $\FT(\RT)$ generated by
all 2-cycle twists at $\RT$.
\end{definition}

%=========================================================
\paragraph{\textbf{Twist groups related by a flip}}\

We now study the relationship between the flip twist groups of triangulations that differ by a flip.
Recall that $\yue(\RT)$ denotes the set of punctures that are enclosed in some self-folded $\RT$-triangles.

\def\RTc{\RT_\circ}
\def\RTb{\RT_\bullet}
Consider a pair of forward flips between two triangulations (which may coincide):
\[
\begin{tikzpicture}[xscale=1.8, rotate=0,
          arrow/.style={->,>=stealth,bend left=25}]
    \clip(-1.2,-.5)rectangle(2.2,.5);
    \draw[] (1,0) node[] (t3) {{$\RTb$}}node[right]{$\colon\mu_\bullet$}
        (0,0)node[](t2) {{$\RTc$}}node[left]{$\mu_\circ\colon$};
    \draw[font=\scriptsize](0.5,0) node {{$i$}};
    \foreach \j/\k in {2/3}{
    \draw[arrow] (t\j) edge node[above]{$z_\circ$} (t\k);
    \draw[arrow] (t\k) edge node[below]{$z_\bullet$}(t\j);}
\end{tikzpicture}
\]
in $\fg(\surfp)$ at arcs $z_\circ$ and $z_\bullet$ respectively.
Denote by $i$ the corresponding vertex in $Q_{\RTc}$ and $Q_{\RTb}$.
%We use $\mu$ to emphasize that it is a morphism in $\fg(\surfp)$ instead of an arc.
Note that $\yue(\RTc)$ and $\yue(\RTb)$ differ by at most one puncture.
Without loss of generality, suppose that $\yue(\RTc) \subset \yue(\RTb)$.
The flip twist of $\RT_*$ with respect to a vertex $i$ of $Q_{\RT_*}$ will be denoted by $t^*_i$ or $l^*_i$.

\begin{lemma}\label{lem:conj}[Conjugation formulae]\

\begin{itemize}
\item  If the flip $\mu_i$ is a loop in $\FG(\surfp)$, i.e. as
$\mu_\circ=\mu_2=s$ at $\RTc=\RT_L=\RTb$ in \Cref{fig:D2}.
Then there is a conjugation/isomorphism
\[\begin{array}{rcl}
    \conj_{\mu_\circ}=\conj_{s} \colon \FT(\RTc)&\to&\FT(\RTb),\\
        f &\mapsto& f^s.
\end{array}\]

\item  If $\yue(\RTc)=\yue(\RTb)$ and $\RTc\ne\RTb$, then
there is a natural isomorphism
\[\begin{array}{rcl}
    \conj_{\mu_\circ}\colon \FT(\RTc)&\xrightarrow{\cong}&\FT(\RTb),\\
            t^\circ_i &\mapsto& t^\bullet_i , \\
            l^\circ_j &\mapsto& l^\bullet_j, \\
            t^\circ_k &\mapsto& \iv{t^\bullet_i} t^\bullet_k t^\bullet_i,
\end{array}\]
for any $j$ in $Q_{\RTc}$ corresponding to a self-fold arc
for any $k\ne i$ in $Q_{\RTc}$ corresponding to a non self-fold arc.
More precisely, one has $\conj_{\mu_\circ} ( t^\circ_k )=t^\bullet_k$ if and only if there is no arrow
from $k$ to $i$ in $Q_{\RTc}$.

\item  If $\yue(\RTb)=\yue(\RTc)\cup\{P\}$, i.e. as
$(\RTc,\RTb)=(\RT_M,\RT_L)$ and $\mu_\circ=v,\mu_\bullet=u$ in \Cref{fig:D2}.
Then there is an injection
\begin{equation}\label{eq:tricky}
\begin{array}{rcl}
    \conj_{\mu_\circ}=\conj_v\colon \FT(\RT_M)&\to&\FT(\RT_L),\\
        t^M_1 &\mapsto& t^L_1 , \\
        t^M_2 &\mapsto& l^L_2 t^L_1 \iv{l^L_2}, \\
        t^M_3 &\mapsto& \iv{t^L_1} t^L_3 t^L_1, \\
        f^M_k &\mapsto& f^L_k,
\end{array}
\end{equation}
with vertices $1,2,3$ of $Q_{\RT_M}$ shown in \Cref{fig:D2} and any other vertex $k$ with flip twist $f^?_k$ in $\RT_?$.
In fact, we have $\FT(\RT_L)=\< \FT(\RT_M)^\mu,l_2^L \>$ for $l_2^L=s$.
\end{itemize}
When restricted to $\CFT$, we always have a natural isomorphism
\begin{equation}\label{eq:CFT}
\begin{array}{rcl}
    \conj_{\mu_\circ} \colon \CFT(\RTc)&\to&\CFT(\RTb),\\
        f &\mapsto& f^s.
\end{array}
\end{equation}
\end{lemma}

\begin{proof}
For the first case, the formula for $\conj_\mu(t_z)$ follows from the higher braid relation $\Br^4(s,u)$
as we have seen in \Cref{fig:decomposition}.

For the second case, it follows essentially the same way as \cite[Prop.~2.8]{KQ2}.
The only difference is that for $j$ corresponding to a self-fold arc,
there is no edge between $j$ and $i$ -- otherwise $\yue(\RTc)\ne\yue(\RTb)$.
Hence $(l^\circ_j)^{\mu_i} = l^\bullet_j$ is a square-commutation relation.

For the third case,
we have
\begin{itemize}
  \item For $1$, $(t^M_1)^u=(u^2)^u=u^2=t^L_1$.
  \item For $2$, $(t^M_2)^u=(v^2)^u=\iv{u}v^2u$ and $l^L_2 t^L_1 \iv{l^L_2}=su^2\iv{s}$.
  Then
    \[
        usuu\xlongequal{usu=vrv}vrvu\xlongequal{r=s^{\iv{uv}}}vvus
    \]
    (or the combination of the two hexagon relations on the right in \Cref{fig:4hex})
    implies that $(t^M_2)^u=l^L_2 t^L_1 \iv{l^L_2}$.

  \item For $3$, $(t^M_3)^u=\iv{t^L_2} t^L_3 t^L_2$ follows the same way as in the second case.
  \item For $k=4$, $(t^M_4)^u=t^L_4$ is an S-DB relation.
  \item For other $k$, $(f^M_k)^u=f^L_k$ is a square=commutation relation (as there is no arrows between $i=2$ and $k$).\qedhere
\end{itemize}
\end{proof}

Note that the flip twist groups are NOT all isomorphic to each other by \eqref{eq:tricky}.

As a consequence, the 2-cycle twist groups are all isomorphic to each other.
In fact, a more precise statement can be made as follows.
\begin{corollary}\label{cor:CFT}
$\{ \CFT(\RT) \mid \RT\in\fg(\surfp) \}$ is a normal subgroupoid of $\pi_1\fg(\surfp)$.
\end{corollary}
\begin{proof}
This is the analogue of the result of first paragraph \cite[Prop.~2.9]{KQ2},
which follows the same way since we have the conjugation isomorphisms \eqref{eq:CFT}.
\end{proof}

%=========================================================
\paragraph{\textbf{The $\lozenge$-flips}}\

\begin{definition}\cite{FST}
A (forward) $\lozenge$-flip is the composition of locally flipping twice at a enclosing edge of a self-folded triangle of a triangulation $\RT$ of $\surfp$.
%For instance, the $\lozenge$-flip $uv$/$vu$ of the left/right triangulated digon in \Cref{fig:D2} is the right/left one, respectively.
\end{definition}

\begin{lemma}\label{lem:conj2}
Consider a pair of $\lozenge$-flips
$\begin{tikzcd}
    \RT_L \ar[r,"uv",shift left=1,bend left=15] & \RT_R \ar[l,"vu",shift left=1,bend left=15]
\end{tikzcd}$
between admissible triangulations as in \Cref{fig:D2}.
Then there is a natural isomorphism
\begin{equation}\label{eq:tricky*}
\begin{array}{rcl}
    \conj_{uv}\colon \FT(\RT_L)&\to&\FT(\RT_R),\\
        t^L_1 &\mapsto& l^R_1 t^R_2 \iv{l^R_1} , \\
        l^L_2 &\mapsto& l^R_1, \\
        t^L_4 &\mapsto& (t^R_4)^ {t^R_2 l^R_1 t^R_2 \iv{l^R_1} }, \\
        f^L_k &\mapsto& f^R_k,\end{array}
\end{equation}
\end{lemma}
\begin{proof}
%Firstly, we have
%\[\begin{cases}
%    (t^L_1)^{uv}= (u^2)^{uv} =\iv{v}u(uvr)\iv{r}\xlongequal{uvr=suv}\iv{v}(usu)\iv{r}
%        \xlongequal{usu=vrv}rv^2\iv{r}=l^R_1 t^R_2 (l^R_1)^{-1} \\
%    (l^L_2)^{uv}=s^{uv}=r=l^R_1.
%\end{cases}\]
%Secondly,
Applying \eqref{eq:tricky} twice, we have
\begin{equation}\label{eq:tricky2}
\begin{array}{ccccc}
    \FT(\RT_L) &\xleftarrow{(-)^u}& \FT(\RT_M)&\xrightarrow{(-)^v}&\FT(\RT_R),\\
      t^L_1 &\mapsfrom& t^M_1 &\mapsto&   l^R_1 t^R_2 \iv{l^R_1}, \\
      l^L_2 t^L_1 \iv{l^L_2} &\mapsfrom& t^M_2 &\mapsto& t^R_2, \\
      \iv{t^L_1} t^L_3 t^L_1 &\mapsfrom& t^M_3 &\mapsto& t^R_3, \\
      t^L_4 &\mapsfrom& t^M_4 &\mapsto& \iv{t^R_2} t^R_4 t^R_2, \\
      f^L_k &\mapsfrom& f^M_k &\mapsto& f^R_k, \\
\end{array}
\end{equation}
Thus, we deduce that
\[\begin{cases}
    (t^L_1)^{uv} %(t^L_1)^{uv}= (u^2)^{uv} =\iv{v}u(uvr)\iv{r}\xlongequal{uvr=suv}\iv{v}(usu)\iv{r}
        %\xlongequal{usu=vrv}rv^2\iv{r}=l^R_1 t^R_2 \iv{l^R_1}, \\
        =\big( (t^L_1)^{\iv{u}} \big) ^{u^2v}
        =(t^M_1)^{u^2v} \xlongequal{u^2=t^M_1} (t^M_1)^{v} = l^R_1 t^R_2 \iv{l^R_1}, \\
    (l^L_2)^{uv}=s^{uv}=r=l^R_1,\\
    (t^L_3)^{uv}=\Big( (t^L_1)^{\iv{u}} \cdot
        \big(\iv{t^L_1} t^L_3 t^L_1\big)^{\iv{u}} \cdot (\iv{t^L_1})^{\iv{u}}   \Big)^{u^2v}
        = (t_1^M t_3^M \iv{t_1^M} )^{ t^M_1 v}=(t_3^M)^v=t_3^R,\\
    (t^L_4)^{uv}=\big( t_4^M \big) ^{u^2v} = \big( (t_4^M)^{t^M_1} \big) ^{v}
        = \big( \iv{t^R_2} t^R_4 t^R_2 \big)^ { l^R_1 t^R_2 \iv{l^R_1} }\\
    (f^L_k)^{uv}= \big( (f_k^M)^{t^M_1} \big) ^{v}    \xlongequal{\Co(f_k^M,t^M_1)}     (f_k^M)^{v}=f_k^R,
\end{cases}\]
as required.
\end{proof}

Combining with the fact, \cite[Lem.~A.2]{QZ1}, that
any two admissible triangulations in $\FG(\surfp)$ are related by a sequence of $\lozenge$-flips.
We know that all $\FT(\RT)$ are isomorphic for admissible triangulations.
We proceed to show that flip twist groups at admissible triangulations are the whole point groups/fundamental groups of $\fg(\surfp)$.

%=========================================================
\subsection{Simply connectedness}\label{sec:pi1}\
%=========================================================

\begin{proposition}\label{pp:whole}
For any admissible triangulation $\RT_0$, $\FT(\RT_0)=\pi_1(\fg(\surfp),\RT_0)$.
\end{proposition}
\begin{proof}

We first show that, for any triangulation $\RT$, there is a path/morphism
\begin{equation}\label{eq:path w}
\begin{tikzcd}
    p\colon\RT=\RT_1 \ar[r,-]& \RT_2 \ar[r,-]& \cdots \ar[r,-]& \RT_m=\RT_0
\end{tikzcd}
\end{equation}
in $\fg(\surfp)$, such that $\yue(\RT_i)\subset\yue(\RT_{i+1})$ for any $i$.
To do this, we first keep all the self-folded edges in $\RT$ (or equivalently cut them),
and (repeatedly) flip $\RT$ to an admissible triangulation $\RT'$, such that along the way,
$\yue(\RT_i)$ is not decreasing.
This can be done due to connectedness of flip graphs of triangulations.
Then by \cite[Lem.~A.2]{QZ1} again, we can $\lozenge$-flip $\RT'$ to $\RT_0$ to get a path $p$ as claimed.
Note that along the way, we can get rid of any loops in $p$ since they are redundant.
By formulae in \Cref{lem:conj} and \Cref{lem:conj2},
the path $p$ induces a conjugation $\conj_p\colon\FT(\RT)\to\FT(\RT_0)$, which is an injection.

Now, take any path/morphism $w$ from $\RT$ to $\RT_0$ in $\fg(\surfp)$.
We claim that the conjugation $\conj_{w}$ above is still an injection.
By \Cref{cor:CFT}, $\conj_{w}$ restricts to an isomorphism $\CFT(\RT)\cong\CFT(\RT_0)$.
We only need to show that for each loop twist $s\in \FT(\RT)$, $\conj_{w}(s)$ is in $\FT(\RT_0)$.
Use induction on the length of ${w}$ starting with the trivial case with $|{w}|=0$.
If there is a loop $r$ in the path $w$, i.e. $w=w_1rw_2$,
then $\conj_{w_2}(r)$ is in $\FT(\RT_0)$ by induction.
As
\[
    \conj_{w}(s)=\iv{w_2rw_1}sw_1rw_2=\iv{r}^{w_2}s^{w_1w_2}r^{w_2}=
        \iv{\conj_{w_2}(r)} \cdot \conj_{w_1w_2}(s) \cdot \conj_{w_2}(r),
\]
we only need to show that $\conj_{w_1w_2}(s)$ is in $\FT(\RT_0)$.
In other words, we can get rid of all loop twists in ${w}$.
Since $\conj_{w}(s)=(\iv{w}p)\cdot\conj_p(s)\cdot(\iv{p}w)$,
we only need to show that $\iv{p}w$ is in $\FT(\RT_0)$.
Note that $\iv{p}w$ does not contain any loop twist.
By \Cref{thm:H}, when regarding $\iv{p}w$ as a path in $\uEG(\surfp)$,
it decomposes into (unoriented) squares and pentagons.
Taking account of the orientation, this means that $\iv{p}w$ in $\FG(\surfp)$ can be written a product of oriented squares/pentagons, together with 2-cycles.
On the other hand, all the 2-cycles are in the 2-cycle twist subgroupoid
by \Cref{cor:CFT}. Thus, $\iv{p}w$ is in $\CFT(\RT_0)\le\FT(\RT_0)$ as required.

So we have shown that the conjugations of any 2-cycles and loops $\fg(\surfp)$ by any paths are in $\FT(\RT_0)$.
On the other hand, let $\ueg(\surfp)$ be the groupoid obtained from $\uEG(\surfp)$ by composing square/pentagon relations.
Consider the canonical functor $\fg(\surfp)\to\ueg(\surfp)$,
which sets all 2-cycles and loops equal to 1.
By \Cref{thm:H}, $\pi_1\ueg(\surfp)$ is trivial and hence $\fg(\surfp)$ is generated by 2-cycles and loops.
This implies that $\FT(\RT_0)=\pi_1(\fg(\surfp),\RT_0)$.
\end{proof}

An immediate consequence is the following.
\begin{corollary}\label{cor:mt=mt}
$\MTT=\MT(\sop)$ and $\dbloop{\FGCp[\T_0,\T_0[1]]}=\FG(\surfp)$ for admissible $\T_0$.
\end{corollary}
\begin{proof}
Choose admissible triangulations $\T$ and $\RT$,
\eqref{eq:MT=FT} implies that $\MT(\T)=\MTT$, which implies the first claim.
Combining with \eqref{eq:3FGs}, the second claim follows.
\end{proof}

Now we can prove the main theorem of this section.
\begin{theorem}\label{thm:FT}
$\fgt(\sop)$ is a universal cover of $\fg(\surfp)$ with covering group $\MT(\sop)$ and %$$\man_\RT\colon\pi_1(\fg(\surfp))\cong\MT(\sop).$$ Moreover,
$\man_\RT$ in \eqref{eq:MT=FT} restricts to a natural isomorphism $\FT(\RT)\cong\MT(\T)$,
which further restricts to an isomorphism $\CT(\RT)\cong\BT(\sop)$ .
\end{theorem}
\begin{proof}
In the case when $\RT$ (and $\T$) is admissible,
$\man_\RT$ in \eqref{eq:MT=FT} becomes $\FT(\RT)\twoheadrightarrow\MT(\T)$,
sending the standard generators to the standard ones.
On the other hand, by \Cref{cor:MT's2},
there is a presentation of $\MT(\T)$, with three type of relations (e.g. $\Co, \Br$ and $\Br^4$).
Each of which, also holds for (the standard generators of) $\FT(\RT)$, as shown in \Cref{lem:decompo}.
Thus, $\man_\RT$ admits a well-defined inverse that forces it to be an isomorphism.
In particular, in \eqref{eq:ses MTT}, we have $\pi_1\fgt(\sop)$ is trivial, i.e. $\fgt(\sop)$ is simply connected and is a universal cover of $\fg(\surfp)$.

In the case when $\RT$ (and $\T$) is not admissible,
the isomorphism $\man_\RT$ restricts to an isomorphism \eqref{eq:MT=FT} between subgroups of
$\pi_1\FG(\surfp,\RT) \cong \MT(\sop)$.
\end{proof}

Another related consequence is that we obtain \Cref{cor:pi0} therefore.

%=========================================================
%=========================================================
\section{On fundamental groups of moduli spaces}\label{sec:uni}
%=========================================================
%=========================================================

%=========================================================
\subsection{Quadratic differentials}\label{sec:quad}\
%=========================================================

Let $\rs$ be a compact Riemann surface and $\omega_\rs$ be its holomorphic cotangent bundle.
A \emph{meromorphic quadratic differential} $\phi$ on $\rs$ is a meromorphic section
of the line bundle $\omega_{\rs}^{2}$.
In terms of a local coordinate at $z$ on $\rs$,
$\phi$ can be written as $\phi(z)=g(z)\, \dd z^2$, for a meromorphic function
$$g(z)=z^d(a_d+\cdots)$$ for $a_d\ne0$.
If $d>0$, $z$ is a \emph{zero} of \emph{order} $d$; if $d<0$, $z$ is a \emph{pole} of \emph{order} $-d$;
if $d=0$, $z$ is a \emph{smooth} point.
The zeros and poles are singularities/critical points of $\phi$.
For $d\in\ZZ_{>0}$,
denote by $\wt_{d}(\phi)$ the set of zeroes of order $d$ and $\wt_{-d}(\phi)$ the set of poles of order $d$.
There are two type of singularities:
\begin{itemize}
\item the finite type, including all zeroes and simple (i.e. order 1) poles.
Let $$\wty=(\wt_{-1},\wt_0,\wt_1,\cdots)$$ be the collection of them.
\item the infinite type, including all poles with order at least two.
Let $$\wtw=(\wt_{-2},\wt_{-3},\cdots)$$ be the collection of them.
\end{itemize}
Denote by $\rs^\circ=\rs^\circ(\phi)=\rs\setminus(\wty\cup\wtw)$.

In general, one could also allow a set $\wt_0(\phi)$ of special smooth markings of $\phi$ in $\rs$.
%However, we will not consider them here.

%=========================================================
\paragraph{\textbf{Foliation structure}}\

The key structure of quadratic differentials is the \emph{foliation} structure.
At a smooth point of $\rs^\circ$,
there is a  distinguished local coordinate $\omega=\int \sqrt{g(z)}\dd z$,
uniquely defined up to transformations: $\omega \mapsto \pm\, \omega+\operatorname{const}$,
with respect to which, one has $\phi(\omega)=\dd \omega \otimes \dd \omega$.
Then by locally pulling back the Euclidean metric on $\CC$ using a distinguished coordinate $\omega$,
$\phi$ induces the $\phi$-metric and geodesics on $\rs$.
Each geodesics have a constant phase with respect to $\omega$.

A \emph{trajectory} of a $\phi$ on $\rs$ is a maximal geodesic $\gamma\colon(0,1)\to\surp$,
with respect to the $\phi$ metric (i.e. corresponds to the equation $\on{Im}\omega\equiv\on{const}$).
When $\lim\gamma(t)$ exists in $\rs$ as $t\to0$ (resp. $t\to1$),
the limits are called the left (resp. right) endpoint of $\gamma$.
The \emph{horizontal} trajectories (with $\on{Im}\omega\equiv0$) of a meromeorphic quadratic differential $\phi$
provide the \emph{horizontal foliation} on $\rs$.

\begin{figure}[ht]\centering
    \includegraphics[width=\textwidth]{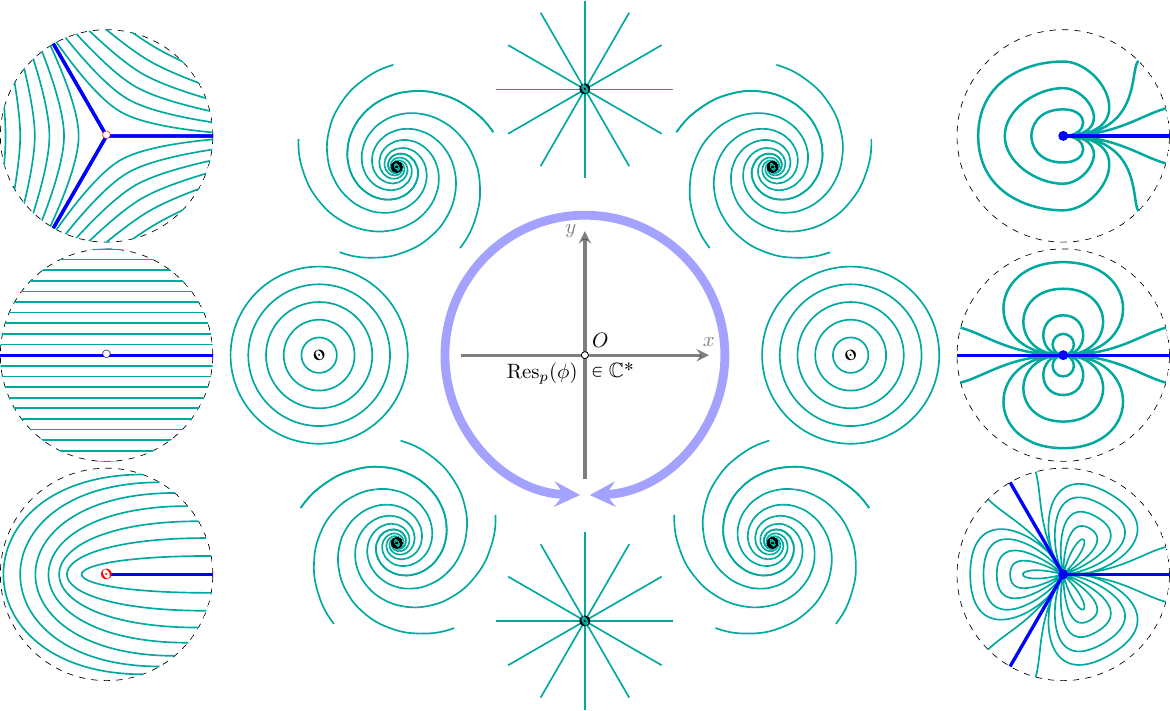}
\caption{Foliation structure around a point of a Riemann surface}\label{fig:foli}
\end{figure}

In \Cref{fig:foli}, one has examples of (horizontal) foliations. Namely:
\begin{itemize}
  \item On the left column, the pictures show the foliation at a simple zero/smooth point/simple pole (up to bottom). In general, there are $d+2$ distinguished directions at a finite singularity (conical point) of order $d\ge-1$
      that divides the neighbourhood of such a point into $d+2$ (local) half planes at zero.
  \item In the middle column, the pictures show the foliation at a double pole $p$ depending on the residue
  \begin{gather}\label{eq:residue}
    \on{Res}_p(\phi)\colon=\pm2\int_{\wh{\delta_p}}\sqrt{\phi},
  \end{gather}
  of $\phi$ at $p$, where the integration is along the lift $\wh{\delta_p}$ of a small loop $\delta_p$
  on the spectral cover (w.r.t. $\phi$) of $\rs$.
  \item On the right column, the pictures show the foliation at a pole of order 3/4/5 (up to bottom).
  In general, there are $|d|-2$ distinguished directions at a pole of order $d\le-3$
      that divides the neighbourhood of such a point into $d+2$ (local) half planes at infinity.
\end{itemize}

%=========================================================
\paragraph{\textbf{Horizonal strip decomposition}}\

There are the following types of trajectories of a quadratic differential:
\begin{itemize}
\item \emph{saddle trajectories}, either whose both ends are in $\wty$.
\item \emph{separating trajectories} with one endpoint in $\wty$ and the other in $\wtw$;
\item \emph{generic trajectories} whose both ends are in $\wty$;
\item \emph{recurrent} if at least one of its directions is recurrent.
\end{itemize}
But from now on, we now fix the horizonal direction to unless specified otherwise,
i.e. trajectories means horizontal trajectories.
A \emph{saddle connection} is a trajectory in some arbitrary direction
and a saddle trajectory means a horizontal saddle connection.

Removing from $\rs$ the separating trajectories and saddle trajectories,
it decomposes the surface into connected components,
each of which is of the following types:
\begin{itemize}
\item a \emph{horizontal strip}, i.e. is isomorphic to
$\{z\in \CC\mid a<\on{Im}(z)<b\}$
equipped with the differential $\dd z^{2}$ for some $a<b \in \RR$.
It is swept out by generic trajectories connecting two (not necessarily distinct) poles.
\item a \emph{half-plane}, i.e. is isomorphic to
$\{z\in \CC \mid \on{Im}(z)>0\}$
equipped with the differential $\dd z^{2}$.
It is swept out by generic trajectories which connect a fixed pole to itself.
\item a \emph{ring domain} or \emph{cylinder}, i.e. is foliated by closed trajectories.
\item a \emph{spiral domain}, i.e. is the interior of the closure of a recurrent trajectory.
\end{itemize}
We call this union the \emph{horizontal strip decomposition} of $\rs$ w.r.t. $\phi$.

%\begin{figure}[hbt]
%\centering\makebox[\textwidth][c]{
% \includegraphics[width=9cm]{New-collision.pdf}
%}   \caption{A loop in $\FQuad(\surfp)$ with a missing center in the partial compactification}\label{fig:loop}
%\end{figure}

A quadratic differential $\phi$ on $\rs$ is \emph{saddle-free}, if it has no saddle trajectory.
Then there are only half-planes and horizontal strips by \cite[Lem.~3.1]{BS}.
Note that in such a case:
\begin{itemize}
\item In each horizontal strip, the trajectories are isotopic to each other.
\item The boundary of any component consists of separating trajectories.
\item In each horizontal strip, there is a unique geodesic,
the \emph{saddle connection}, connecting the two zeroes on its boundary.
\end{itemize}

\begin{example}
The right picture of \Cref{fig:2loop} shows eight local horizontal strip decompositions on $\rs$ corresponding to the motion of a L-flip at a self-folded triangle (in the top disk),
which are around a collision state--the middle disk.

In particular, there is a `wall crossing' state (the bottom disk)
when there is a saddle trajectory (red cycle),
which bounds a \textrm{degenerate} ring domain (meaning one of its boundary components is a double pole).
Note that, in the disks
\begin{itemize}
\item the emerald/blue vertices (bullets) are poles or marked points,
\item the black/red suns are double/simple poles respectively,
\item the red vertices (small circles) are simple zeroes,
\item the emerald arcs are (generic) geodesics,
% and the darker ones (in the top picture) are representatives forming a triangulation,
\item the blue arcs are separating trajectories,
\item the red solid arc/cycle is a saddle trajectory (in the bottom disk)
    and the red dashed arc is a saddle connection (in the top disk).
\end{itemize}

Moreover,
the right picture of \Cref{fig:2loop} also demonstrates a loop in the moduli space of quadratic differentials which corresponds to the L-twist (as the bottom picture of \Cref{fig:T}) of the saddle connection in these octagons.

For comparison, see the left picture of \Cref{fig:2loop} (or \cite[Fig.~4.3]{KQ2})
shows eight horizontal strip decompositions of another loop in the moduli space,
which corresponds to a braid twist (as the top picture of \Cref{fig:T}),
which are around a collision state--the middle octagon.
\end{example}

The \emph{real blow-up} of $\rs$ with respect to $\phi$ is a weighted DMSp $\rs^\phi_{\wty}$ obtained from $\rs$ by replacing each puncture $p$ in $\wt_{\le-3}$ with a boundary component $\partial_p$ with $\on{ord}_\phi(p)-2$ marked points
and regarding $\wty=\Tri$ as decoration set with weights given by the degree of the finite singularities, $\wt_{-2}=\sun$ as puncture set.
One can also forget about the decorations and denote by $\rs^\phi$ the corresponding weighted marked surface with punctures.

%=========================================================
\paragraph{\textbf{Moduli spaces and framing}}\

The \emph{moduli space} $\Quad(g)$ of quadratic differentials on genus $g$ Riemann surfaces decomposes into many strata,
parameterized by their prescribed singularity type, i.e.
\[
    \Quad(g)=\bigsqcup_{|\wt|=4g-4} \Quad(g;\wty,\wtw),
\]
where
\begin{equation}\label{eq:4g-4}
    \|\wt\|=\| \wty \|+\| \wtw \|\colon=\sum_{k\in\ZZ} k|\wt_k|
\end{equation}
satisfying $\|\wt\|=4g-4$,
for $|\wt_k|$ the number of points in $\wt_k$.

The \emph{numerical data} of singularity type of $\phi$ on $\rs$ is $(\wty,\wtw)$,
which can be encoded by the associated weighted DMSp $\rs^\phi_{\wty}$.
We will only consider the strum $\Quad(g;\wt_1,\wtw)$ for $\wty=\wt_1$ (i.e. only simple zeros as finite singularities) in this section and
$\wtw\ne \wt_{-2}$ (i.e. not only just double poles as infinite singularities).
In such a case, a quadratic differential in $\Quad(g;\wt_1,\wtw)$ is parameterized by a DMSp $\sop$
(with trivial weight).

More precisely, a point $(\rs,\phi,\psi)$ in $$\Quad(\surfp)=\Quad(\sop)=\Quad(g;\wt_1,\wtw)$$
consists of a Riemann surface $\rs$ with quadratic differential $\phi$,
equipped with a diffeomorphism $\psi\colon\surfp \to\rs^\phi$, preserving the marked points and punctures.
Two such quadratic differentials $(\rs_1,\phi_1,\psi_1)$ and $(\rs_2,\phi_2,\psi_2)$
are equivalent, if there exists a biholomorphism $f\colon\rs_1\to\rs_2$
such that $f^*(\phi_2)=\phi_1$ and furthermore $\psi_2^{-1}\circ f_*\circ\psi_1\in\Diff(\surfp)$,
where $f_*\colon\rs_1^{\phi_1}\to\rs_2^{\phi_2}$ is the induced diffeomorphism.

As discussed in \cite[\S~4.1]{KQ2}, there are two natural way to frame on $\Quad(\surfp)$.
Namely, in the definition above, if we change $\Diff(\surfp)$ to $\Diff_0(\surfp)$ to identify equivalent elements,
then we obtain the $\surfp$-framed moduli space $\FQuad(\surfp)$.
Similarly, we can also use $\sop$ to frame, i.e. replacing the diffeomorphism $\psi\colon\surfp \to\rs^\phi$
by a diffeomorphism $\psi\colon\sop \to\rs^\phi_{\wt_1}$ preserving the marked points, punctures and decorations.
These three moduli spaces fits in the following diagram:

\begin{equation}\label{eq:3Moduli}
\begin{tikzpicture}[xscale=1.6,yscale=0.8,baseline=(bb.base)]
\path (0,1) node (bb) {}; % baseline
\draw (0,2) node (s0) {$\FQuad(\sop)$}
 (0,0) node (s1) {$\FQuad(\surfp)$}
 (2.5,1) node (s2) {$\Quad(\surfp)$};
\draw [-stealth, font=\scriptsize]
 (s0) edge node [left] {$\SBr(\sop)$} (s1)
 (s0) edge node [above] {$\MCG(\sop)$} (s2)
 (s1) edge node [below] {$\MCG(\surfp)$} (s2);
\end{tikzpicture}
\end{equation}

For a $\sop$-framed
the generic trajectories on $\rs$ (with respect to $\phi$)
are inherited by $\surf$, for any $\psi\colon\surf\to\rs^\phi$, and
all trajectories on $\rs$ (with respect to $\phi$)
are inherited by $\surfo$, for any $\Psi\colon\surfo\to\rs^\phi$.

So generic trajectories become open arcs on $\surf$ (as well as on $\surfo$)
and saddle trajectories/connections becomes closed arcs on $\surfo$.

\begin{example}\label{ex:D2}
The foliations in the right picture of \Cref{fig:2loop} are precisely a loop in the once-punctured monogon case,
i.e. $\Quad(\surfp)=\Quad(0;1,-2,-3)$.
One can write down the quadratic differentials globally (on $\CC\mathbb{P}^1$):
\[
    \Quad(0;1,-2,-3)=\{ c\in\CC^* \mid \phi(z)=\frac{z+c}{z^2} \dd z ^2 \}.
\]
More precisely
\begin{itemize}
  \item The residue at the double pole (i.e. zero) is $s=\on{Res}_0(\phi)=4\pi\mathbf{i}\sqrt{c}$.
  \item The central foliation in the right picture of \Cref{fig:2loop} correspond to the $c=0$ case:
  then the quadratic differential $\phi(z)=z^{-1} \dd z ^2$ is in $\Quad(0;-1,-3)$.
\end{itemize}
Moreover, in this case, we have
\begin{itemize}
  \item $\FQuad(\surfp)=\Quad(\surfp)\cong\CC^*$.
  \item $\FQuad(\sop)=\{\on{log}(c)\mid c\in\CC^*\}\cong\CC$ is the $\ZZ$-cover/universal cover of $\FQuad(\surfp)$.
%  \item In \Cref{sec:}, we will consider the double cover $\Quad^\pm(\surfp)$ of $\Quad(\surfp)$ that corresponds to choosing a square-root of $s=\sqrt{c}$ (cf. \cite[Ex.~12.3]{BS})      and the orbifold
\end{itemize}
\end{example}

%=========================================================
\subsection{Cell structures}\label{sec:cell}\
%=========================================================

%=========================================================
\paragraph{\textbf{Stratification}}\

Let $\Psi\in\FQuad(\sop)$.
One has the following numerical equality:
\[
    r_\Psi+2s_\Psi+t_\Psi=\| \wty \|+2|\wty|=\sum_{k=-1}^{\infty}(\wt_k+2)=\colon K
\]
where $s_\Psi/r_\Psi/t_\Psi$ are the numbers of saddle/recurrent/separating trajectories, respectively.

There is a natural filtration of $\FQuad(\sop)$ given by
\[
    B_0 = B_1 \subset B_2 \subset \cdots \subset B_K = \Quad(\sop),
\]
for
\begin{gather}\label{eq:B}
    B_p = B_p(\sop) = \{ \Psi \in  \Quad(\sop):
        r_\Psi + 2s_\Psi \leq p\},
\end{gather}
observing that $B_0 = B_1$ is the set of saddle-free differentials.
The corresponding \emph{stratification} is given by $F_p(\sop) = B_p \setminus B_{p-1}$.

The space~$B_2$ is called the space of \emph{tame differentials},
We observe that $F_0$ is dense, $F_1$ is empty, and $F_2$ consists of differentials with exactly one saddle trajectory.
Here are two more precisely statements, which was proved in \cite{BS} and are generalized in \cite{BMQS}.

\begin{lemma}\cite[Lem.~4.1]{BMQS}\label{le:B0dense}
$\FQuad(\sop)=\CC\cdot B_0(\sop)$ and hence $B_0(\sop)$ is dense in $\FQuad(\sop)$.
Moreover, $F_2(\sop)$ has codimension 1 and $B_2(\sop)$ has codimension 2.
\end{lemma}

\begin{lemma}\cite[Cor.~4.3]{BMQS}\label{lem:pi0}
The connected components of $\FQuad(\sop)$ naturally correspond to the connected components of $\FG(\sop)$.
\end{lemma}

All the statements above also apply to the $\surfp$-framed version.

%=========================================================
\paragraph{\textbf{WKB-triangulations}}\

Let
\[
    \UHP=\{z\in\CC\mid \on{Im}(z)>0\}\subset\CC
\]
be the strict upper half plane and $\uhp=\UHP\cup\RR_{<0}$ the half-open-half-closed one.

Let $\Psi$ be a saddle-free $\sop$-framed quadratic differential.
Then there is an associated \emph{(WKB-)triangulation} $\T_\Psi$ of $\sop$,
where the arcs are (isotopy classes of inherited) generic trajectories.
The dual graph $\T_\Psi^*$ consisting of saddle connections.
Denote by $\cubc(\RT)$ be the subspace in $\FQuad(\surfp)$ consisting of
those saddle-free $\Psi$ support on $\RT$ (i.e. whose associated triangulation is $\RT$).
Then $\cub(\RT)\isom\uhp^{\RT}$ and
\begin{gather}\label{eq:Squad*}
    B_0(\surfp)=\bigcup_{\RT\in\FG(\surfp)} \cubc(\RT).
\end{gather}
for $\cubc(\RT)\isom\UHP^{\RT}$.
The coordinates $(u_\gamma)_{\gamma\in \RT}$ give the complex modulus of the horizontal strip with generic trajectory in the isotopy class $\gamma$.
Thus the $\cubc(\RT)$ are precisely the connected components of $B_0(\surf)$.
The boundary of $\cubc(\RT)$ meets $F_2(\surf)$ in $2\numarc$ connected components,
which we denote $\partial^\sharp_\gamma\cubc(\RT)$ and $\partial^\flat_\gamma\cubc(\RT)$,
where the coordinate $u_\gamma$ goes to the negative or positive real axis, respectively.
Note that $u_\gamma$ cannot go to zero since it means the collision of two zeros.

By Walls-have-Ends property (\cite[Prop~5.8]{BS} for simple zero case and \cite[Cor.~4.3]{BMQS} in general),
we have the following, showing that $\FG(\sop)$ is a skeleton for $\FQuad^{\T}(\sop)$.

\begin{lemma}\label{lem:f.d.1}
There is a (unique up to homotopy) canonical embedding
\begin{equation}\label{eq:Sembed}
\skel_{\sop}\colon\FGT(\sop)\to\FQuad^{\T}(\sop)
\end{equation}
whose image is dual to $B_2(\sop)$
and which induces a surjective map
\begin{equation}\label{eq:skel}
  \begin{tikzcd}
    \skel_*\colon \pi_1\FGT(\sop) \ar[r,twoheadrightarrow]& \pi_1\FQuad^{\T}(\sop).
  \end{tikzcd}
\end{equation}
\end{lemma}

Again, these statements also apply to the $\surfp$-framed version.
For the later use, we record them as well:
there is a canonical embedding
\begin{equation}\label{eq:Sembed2}
    \skel_{\surfp}\colon\FG(\surfp)\to\FQuad(\surfp)
\end{equation}
that induces a surjective map
\begin{equation}\label{eq:skel2}
  \begin{tikzcd}
    \skel_*\colon \pi_1\FG(\surfp) \ar[r,twoheadrightarrow]& \pi_1\FQuad(\surfp).
  \end{tikzcd}
\end{equation}

Parallel to the cell structure of moduli spaces of quadratic differentials,
one has stability structures on triangulated categories.
% see \Cref{sec:stab}.
We refer to \cite{BS,BMQS}, \cite[\S~4]{KQ2} and \cite[App.~A]{CHQ} for more details
on stability conditions.
We recall a upgraded version of \Cref{thm:CHQ1}.

\begin{theorem}\cite[Thm.~1.1]{CHQ}\label{thm:CHQ2}
The isomorphism \eqref{eq:iso} can be complexified to an isomorphism between complex manifolds
\begin{gather}\label{eq:CHQ2}
\begin{array}{rcl}
  \Xto_*\colon\FQuad^{\T}(\sop)&\cong&\Stab^\T(\Dsop),\\
  s\cdot \Psi&\mapsto& s\cdot(Z,\h_{\T_{\Psi}}),
\end{array}
\end{gather}
for any $s\in\CC$ and saddle-free $\sop$-framed quadratic differential $\Phi$, where
\begin{equation}
    Z(X_\eta)=\int_{\wh{\eta}}\sqrt{\Psi}.
\end{equation}
\end{theorem}

In \Cref{sec:MSx}, we will study a quotient version of \eqref{eq:CHQ2}:
\begin{gather}\label{eq:CHQ2+}
  \Quad(\surfp)=\FQuad(\surfp)/\MCG(\surfp) \cong \Stab^\T(\Dsop)\Autp,
\end{gather}
where $\Autp$ is a quotient group of the autoequivalence group of $\Dsop$, cf. \eqref{eq:Autp}.
%=========================================================
\subsection{The kernel of AJ map as the fundamental group}\label{sec:main}
%=========================================================

\begin{proposition}\label{pp:SQH}
Under the embedding \eqref{eq:Sembed} and the isomorphism \eqref{eq:CHQ2},
any square, pentagon, S-DB, L-DB or Sym-Hex in $\FG(\sop)$ is contractible in $\Stab^\T(\Dsop)$.
\end{proposition}
\begin{proof}
The square, pentagon and the S-DB follows precisely as in the proof of \cite[Prop.4.14]{KQ2}.
The other two new relations (L-DB and Sym-Hex) follows in a similar fashion.
We sketch the proof as follows.

By \Cref{thm:CHQ1},
a L-DB or Sym-Hex in $\FGT(\sop)$ induces a corresponding hexagon in $\EGT(\Dsop)$
as shown in \Cref{fig:C-1} or \Cref{fig:C-2} respectively.
\begin{figure}[hb]\centering
    \includegraphics[width=14cm]{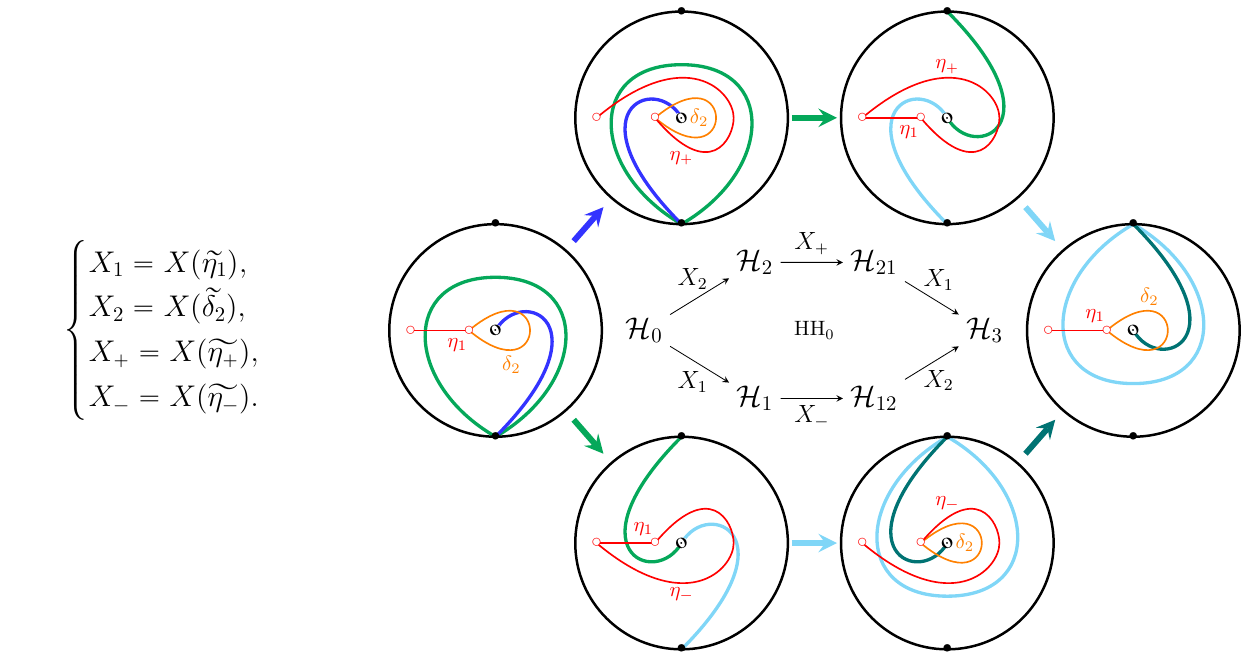}
    \caption{L-DB in $\EGT(\Dsop)$}\label{fig:C-1}
\end{figure}
\begin{figure}[ht]\centering
  \includegraphics[width=14cm]{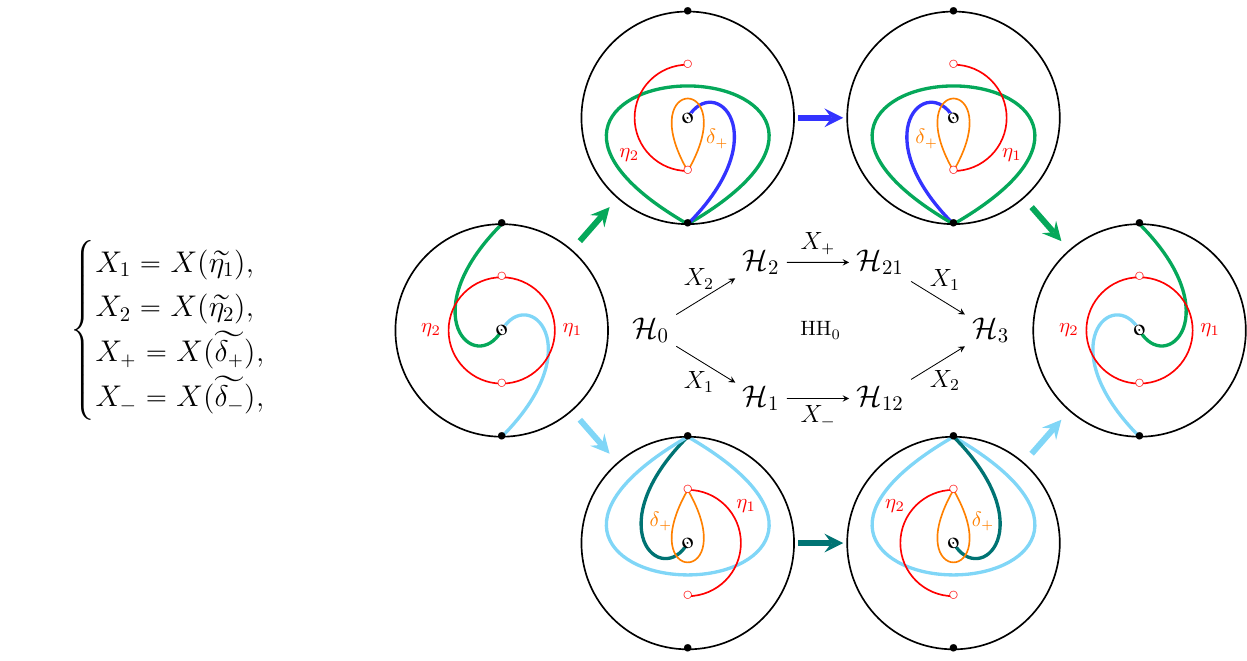}
    \caption{Sym-Hex in $\EGT(\Dsop)$}\label{fig:C-2}
\end{figure}

In fact, such an induced hexagon can be described in the same way in $\EGT(\Dsop)$
by setting the simples $X_1,X_2$ in the source heart $\h_0$ correspond to the proper (graded) closed arcs
as the formulae shown in the left of the figures.
Then there are always two short exact sequences in $\h_0$
\begin{gather}\label{eq:xpm}
    0\to X_2\to X_+\to X_1\to0,\\
    0\to X_1\to X_-\to X_+\to0.
\end{gather}
and the same hexagon of hearts (=$\on{HH}_0$) in $\FGT(\Dsop)$ with
\begin{gather*}
  \begin{cases}
    \h_i=(\h_0)_{X_i}^\sharp,\; i=1,2 \\
    \h_{21}=(\h_2)_{X_+}^\sharp, \\
    \h_{12}=(\h_{1})_{X_-}^\sharp , \\
    \h_{3}=(\h_{21})_{X_1}^\sharp=(\h_{12})_{X_2}^\sharp,
  \end{cases}
\end{gather*}
which consists of two paths:
\[\begin{cases}
    \on{HH}_0^+ \colon \h_0\to\h_2\to\h_{21}\to\h_3,\\
    \on{HH}_0^- \colon \h_0\to\h_1\to\h_{12}\to\h_3.
\end{cases}\]
Note that the rigidities of the simple tiltings are different in these two cases.

Now we use the $\CC$-action to show that such induced hexagons are contractible.
Namely, we can choose $\sigma_\pm$ in $\cub(\h_0)$ such that
\begin{gather}
    \begin{cases}
        Z_+(X_1[1]) = \exp 3\pi\mai\delta, &\\
        Z_+(X_2[1]) = \exp \pi\mai\delta, &
    \end{cases}
    \quad
    \begin{cases}
        Z_-(X_1[1]) = \exp \pi\mai\delta, &\\
        Z_-(X_2[1]) = \exp 3\pi\mai\delta, &
    \end{cases}
\\
        Z_\pm(X)=M\cdot \exp \textfrac{\mai\pi}{2}, \qquad X\in\Sim\h_0\setminus\{X_1,X_2\}
\end{gather}
for some real number $M\gg1$.
for some small $\delta\in\RR_{>0}$ and large $M\in\RR_{>0}$ as shown in \Cref{fig:Z}.

\begin{figure}\centering
\begin{tikzpicture}
\begin{scope}[shift={(6,0)},scale=3,xscale=1]
\draw[thick](0,0)edge[-stealth](15:.4) (0,0)edge[-stealth](5:.4) (0,0)edge[-stealth]($(5:.4)+(15:.4)$)
    (60:.6)edge[orange,-stealth,bend left=30] node[above]{$\cdot\epsilon$}(30:.6)
    (0,0) edge[-stealth](0,1.4);
\draw[thick,dashed,gray](-.2,0)edge[-stealth](.8,0)(0,-.2)edge[-stealth](0,.5);
\draw[gray](.8,0)node[right]{$x$}(0,.5)node[left]{$y$}(0,0)\nn;
\draw[font=\small]    ($(5:.4)+(15:.4)$)node[above]{$Z_+$}
    (15:.4)node[above]{$Z_2$}(5:.4)node[below]{$Z_1$}
    (0,1.4) node[above]{$Z(X)$};
\end{scope}
\begin{scope}[shift={(0,0)},scale=3,xscale=1]
\draw[thick](0,0)edge[-stealth](15:.4) (0,0)edge[-stealth](5:.4) (0,0)edge[-stealth]($(5:.4)+(15:.4)$)
    (60:.6)edge[orange,-stealth,bend left=30] node[above]{$\cdot\epsilon$}(30:.6)
    (0,0) edge[-stealth](0,1.4);
\draw[thick,dashed,gray](-.2,0)edge[-stealth](.8,0)(0,-.2)edge[-stealth](0,.5);
\draw[gray](.8,0)node[right]{$x$}(0,.5)node[left]{$y$}(0,0)\nn;
\draw[font=\small]    ($(5:.4)+(15:.4)$)node[above]{$Z_-$}
    (15:.4)node[above]{$Z_1$}(5:.4)node[below]{$Z_2$}
    (0,1.4) node[above]{$Z(X)$};
\end{scope}
\end{tikzpicture}
\caption{Central charges of $\sigma_+,\sigma_-$}\label{fig:Z}
\end{figure}
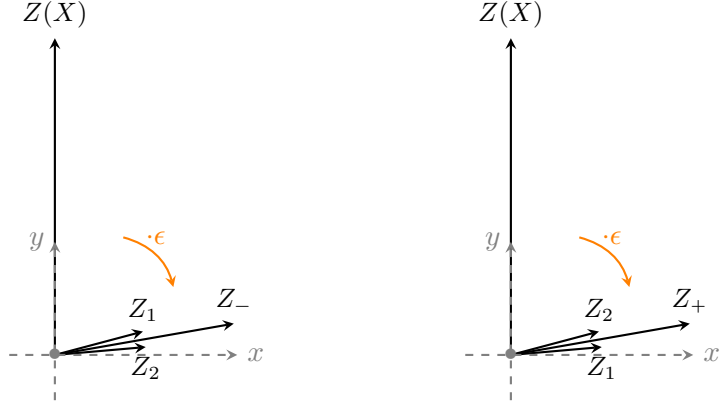

Then we can deduce from \eqref{eq:xpm} that
$X_\pm$ is stale with respect to $\sigma_\pm$ but unstable with respect to $\sigma_\mp$.
Furthermore, by the simple tilting formulae in \cite[App.~A]{CHQ},
we deduce that $k\delta\cdot\sigma_+$, for $k=1,2,3$, are in the three tilting walls of the path
$\on{HH}_0^+$ while $(k\delta-\delta,k\delta)\cdot\sigma_+$ is in the corresponding cube $\cubc(\h_{l})$
for $k/l=1/0,2/2,3/21,4/3$.
Hence we deduce that $[0,4\delta]\cdot\sigma_+$ is homotopy to $\skel_*(\on{HH}_0^+)$.
Similarly, we have $[0,4\delta]\cdot\sigma_-$ is homotopy to $\skel_*(\on{HH}_0^-)$.
Then if we connect $\sigma_\pm$ with a line segment $l$ in $\cubc(\h_0)$,
we see that $4\delta\cdot l$ will be in $\cubc(\h_3)$ and
the boundary of the rectangle $[0,4\delta]\cdot l$ is homotopy to $\skel_*(\on{HH}_0)$.
Thus, $\skel_*(\on{HH}_0)$ is contractible as required.
\end{proof}

By \Cref{eq:Sembed} and \Cref{pp:SQH}, the simply connectedness of $\fgt(\sop)$ in \Cref{thm:FT}
implies the simply connectedness of $\Stab^\T(\Dsop)$.
\begin{theorem}\label{thm:stab}
$\Stab^\T(\Dsop)$ is the universal cover of $\FQuad(\surfp)$.
\end{theorem}

Combining \Cref{thm:QZ+} (that $\MT(\sop)=\ker\AJ$), \Cref{thm:FT} (that $\FT(\RT)\cong\MT(\T)$)
and the characterization of connected components (in \Cref{cor:pi0} and \Cref{lem:pi0}),
we obtain our main theorem.

\begin{theorem}\label{thm:main}
Let $\sop$ be a DMSp. Then we have
\begin{gather}\label{eq:long}
    \pi_1\FQuad(\surfp)\cong\pi_1\fg(\surfp)\cong\MT(\sop)=\ker\AJ
\end{gather}
and
%that fits into the short exact sequence
%\begin{gather}\label{eq:ses1*}
%    1\to \pi_1\FQuad(\surfp)\to \SBr(\sop)\xrightarrow{ \AJ } \Ho{1}(\surf)\to1.
%\end{gather}
%Moreover, we have
\begin{equation}\label{eq:main*}
  \begin{cases}
     \pi_1\FQuad^{\T}(\sop)=1,\\
     \pi_0\FQuad(\sop)\cong\Ho{1}(\surf).
  \end{cases}
\end{equation}
\end{theorem}

%=========================================================
%=========================================================
\section{On orbifold fundamental groups}\label{sec:xxx}
%=========================================================
%=========================================================
\subsection{Partial compactification}\
%=========================================================

%Simple poles from collision
Consider a $\{1,2\}$-valued function $\uk\colon \sun\to\{1,2\}$.
%We write $\uk_i=\uk^{-1}(i)$, $i=1,2$, for simplicity.
Let $\uk_i=\uk^{-1}(i)$, $i=1,2$.

\begin{definition}
Denote by $\FQuad_{\uk}(\surfp)$ the moduli space of
\emph{$\surfp$-framed quadratic differentials with parameter $\uk$},
where a point $(\rs,\phi,\psi)$ in it consists of a quadratic differential $\phi$ on $\rs$
with a diffeomorphism/framing $\psi\colon\surfp \to \rs^\phi_{\wt_1(\phi)}$, sending
\begin{itemize}
  \item the punctures in $\uk_2\subset\sun$ to the double poles in $\wt_2(\phi)$ and
  \item the `collided points' $\uk_1\subset\sun$ to simple poles in $\wt_1(\phi)$.
\end{itemize}
In particular, the number $|\wt_1(\phi)|$ of simple zeros of $\phi$ equals $|\Tri|-|\uk_1|$,
determined by the equation $\|\wt\|=4g-4$.
Two such data $(\rs_i,\phi_i,\psi_i)$ are equivalent if there exists a biholomorphism $f\colon\rs_1\to\rs_2$
such that $f^*(\phi_2)=\phi_1$ and $\psi_2^{-1}\circ f_*\circ\psi_1\in\Diff_0(\surfp)$.
\end{definition}

Now we fix a partition $\sun=\yue\cup\kui$ and make the punctures in $\kui$ collidable.
Define the \emph{moduli space of quadratic differentials on $\surfk$} to be the union
\begin{gather}\label{eq:FQuad-k}
    \FQuad(\surfk)=\coprod_{\uk(\yue)=2} \FQuad_{\uk}(\surfp),
\end{gather}
which is a partial compactification of $\FQuad(\surfp)$.
As in \cite[\S~6.5]{BS}, the period coordinate of $\FQuad(\surfp)$ extends to $\FQuad(\surfk)$.

\begin{remark}\label{rem:k.vs.x}
The mapping class group $\MCG(\surfk)\lhd\MCG(\surfp)$ naturally acts on $\FQuad(\surfk)$
(by left composing with the framing).
However, $\MCG(\surfp)$ does not act on it since one can not permute a usual puncture with a collidable one,
except for the cases $(\yue,\kui)=(\sun,\emptyset)$ and $(\yue,\kui)=(\emptyset,\sun)$,
where $\MCG(\surfp)=\MCG(\surfk)$.
\end{remark}

$\FQuad(\surfp)\subset\FQuad(\surfk)$ is called the \emph{complete} part,
consisting of the differentials without simple poles.
The incomplete part is denoted by
$$ \on{IC}(\surfk)\colon=\FQuad(\surfk)\setminus\FQuad(\surfp).$$

%=========================================================
\subsection{Killing loops}\
%=========================================================

Let $\mu(\kui)$ be the set of (forward) L-flips in $\FG(\surfp)$ with the corresponding isolated punctures in $\kui$.
We will regard L-flips in $\mu(\kui)$ as loops/elements in $\pi_1\FQuad(\surfp)$
under the map \eqref{eq:skel2}.
By the right picture of \Cref{fig:2loop}, we see that any L-flip in $\mu(\kui)$ becomes contractible in $\FQuad(\surfk)$ after the partial compactification.
We claim that this is essentially the only type of loops $\FQuad(\surfp)$ that get killed in $\pi_1\FQuad(\surfk)$.

\def\Loop{\on{c}}
\begin{proposition}\label{pp:tech}
Let $\Loop$ be a loop in $\FG(\surfp)$. If it is contractible in $\FQuad(\surfk)$,
then, it decomposes to a finite sum of L-flips in $\mu(\kui)$
in $\pi_1\FQuad(\surfp)$ up to homotopy.
\end{proposition}
\begin{proof}
$\FQuad_{\uk}(\surfp)$ has complex co-dimension $|\uk_1|$ in $\FQuad(\surfk)$.
Let $R(\surfk)$ be the subspace of $\FQuad(\surfk)$, consisting of saddle-free quadratic differentials
with exactly one simple pole.
These are the `top strata' in the incomplete part $\on{IC}(\surfk)$
with real co-dimension two in $\FQuad(\surfk)$.
The rest of $\on{IC}(\surfk)$ has real co-dimension at least three in $\FQuad(\surfk)$.

Let $D$ be a disk with $\partial D=\Loop$ that is contractible in $\FQuad(\surfk)$.
By perturbation, we can make $D$ only intersect $\on{IC}(\surfk)$ in $R(\surfk)$
and intersect $R(\surfk)$ transversally.
By dimension reason, the intersection consists of isolated points.
If $D$ does not intersect $\on{IC}(\surfk)$, then it is trivial and there is nothing to prove.
Otherwise, we can decompose $D$ into many small disks such that each of which
only intersect $R(\surfk)$ at one point.
We only need to show that the assertion holds for each disk.
In other words, we may assume that $D$ intersects $\on{IC}(\surfk)$ at one point $\Psi_0$
and we only need to show that it is homotopy to an L-flip in $\mu(\kui)$.

As $\Psi_0$ has exactly one simple pole and is saddle-free,
it admits a WKB-triangulation with exactly one triangle containing a simple pole with foliation shown in the center octagon of the right picture of \Cref{fig:2loop}, and all other triangles are once-decorated (i.e. with a simple zero).
Then there is a saddle-free differential $\Psi$ in $\FQuad(\surfp)$, cf. the top disk of the right picture of \Cref{fig:2loop}, such that $\Psi_0$ is the closure (taken in $\FQuad(\surfk)$) of the cell
\[
    \cubc(\RT_\Psi)\subset B_0(\surfp)\subset \FQuad(\surfp),
\]
where $\RT_\Psi$ is the WKB-triangulation associated to $\Psi$.
Note that $\RT_\Psi$ is the same WKB-triangulation for $\Psi_0$ (after forgetting about the simple zeros and simple poles) and let $T_0$ be the self-folded triangle in $\RT_\Psi$ containing the simple pole of $\Psi_0$.
Recall from \Cref{sec:cell} that $\cubc(\RT_\Psi)\isom\UHP^{\RT_\Psi}$ with coordinate $(u_\gamma)_{\gamma\in \RT_\Psi}$.
Let $u_0$ be the one corresponding to the saddle connection in $T_0$
(dual to self-folded edge $\gamma_0$, cf. dashed red arc in the top octagon of the right picture of \Cref{fig:2loop})
and then $\Psi_0$ is in the boundary
\[
    \partial_0=\{ (u_\gamma)_{\gamma\in \RT_\Psi} \mid u_0=0 \} \subset \partial \UHP^{\RT_\Psi}.
\]
In fact, a neighbourhood of $\Psi_0$ in $\on{IC}(\surfk)$ lies in $\partial_0$.
Therefore,
we deduce that when taking a neighbourhood $N_D(\Psi_0)$ of $\Psi_0$ in $D$, we have
\[
    \left( N_D(\Psi_0)\setminus\{\Psi_0\} \right) \subset
        \left( \UHP^{\RT_\Psi}\bigsqcup \partial^\sharp_{\gamma_0} \UHP^{\RT_\Psi} \right),
\]
where $\partial^\sharp_{\gamma_0} \UHP^{\RT_\Psi}$ is the wall containing differential with exactly one saddle trajectory in the triangle $T_0$, cf. the bottom disk in the right picture of \Cref{fig:2loop}.
Hence we see that $\partial N_D(\Psi_0)$ is homotopy to (up to sign) the L-flip $$\mu_{\gamma_0}^\sharp\colon\RT_\Psi\to\RT_\Psi.$$
By the assumption that $D$ is contractible in $\FQuad(\surfk)$,
we see that
\[
    \Loop=\partial D\backsimeq \partial N_D(\Psi_0)\backsimeq \pm \mu_{\gamma_0}^\sharp
\]
in $\pi_1\FQuad(\surfp)$ as required.
\end{proof}

Recall that $\LA(\sop,\kui)$ consists of L-arcs of $\sop$ enclosing exactly a puncture in $\kui$
and $\on{L}(\kui)$ is the normal subgroup in $\SBr(\sop)$ generated by all L-twists along L-arcs in $\LA(\sop,\kui)$.
We have the following corollary of \Cref{thm:main}.

\begin{theorem}\label{thm:pi1-kui}
The fundamental group of $\FQuad(\surfk)$ is
\begin{gather}\label{eq:long-k}
  \pi_1\FQuad(\surfk)=\pi_1\FQuad(\surfp)/\mu(\kui)\cong\MT(\sop)/\on{L}(\kui)
    .%=\MT(\surfo^\yue),
\end{gather}
%where $\surfo^\yue$ is the DMSp obtained from $\surfp$ by deleting punctures in $\kui$.
\end{theorem}
%Note that $\surfo^\yue$ does not satisfy \eqref{eq:aleph}.
\begin{proof}
Denote by $\fg(\surfk)$ the groupoid obtained from $\fg(\surfp)$ by adding the relations
$\mu=1$ for any L-flip $\mu$ in $\mu(\kui)$.

Choose a saddle-free different $\Psi$ in $\FQuad(\surfp)$ with admissible WKB-triangulation $\RT$.
By \eqref{eq:long}, we have that
\[
    \pi_1\FQuad(\surfp,\Psi)\cong\pi_1\fg(\surfp,\RT) \; (=\FT(\RT)\cong\MT(\T)=\MT(\sop)),
\]
where $\T$ is a decorated triangulation satisfying $\RT=F_*^\Tri(\T)$.
By \Cref{pp:tech}, we know that
\[
    \pi_1\FQuad(\surfk,\Psi)=\pi_1\FQuad(\surfp)/\mu(\kui)\cong \pi_1\fg(\surfk,\RT)
        =\MT(\sop)/\on{L}(\kui),
\]
so that we have \eqref{eq:long-k}.
%except for the last equation,
%which then follows from the presentation of $\MT(\sop)$ in \Cref{thm:MT's},
%comparing to the one of $\MT(\surfo^\yue)$.
\end{proof}

%=========================================================
\subsection{Orbifolding}\
%=========================================================

%The extra strata play the role of hyperplanes, which should get a $\ZZ_2$-orbifolding in the quotient space.%, cf. \Cref{sec:ex} for examples.
%=========================================================
\paragraph{\textbf{Signed quadratic differentials}}\

A quadratic differential on $\surfx$ is a \emph{signed quadratic differential} $(\Psi,\epsilon)$,
consisting of $\Psi=[\rs,\phi,\psi]\in\FQuad(\surfk)$ and
together with a choice $\epsilon(V)$ of signs of the residue $\on{Res}_V(\phi)$ at each $V\in\uk_2\cap\kui$,
cf. \eqref{eq:residue}, and a random choice of signs at each $V\in\uk_1$.
Two such differentials $(\Psi_i,\epsilon_i)$ are equivalent if $\Psi_1=\Psi_2$ in $\FQuad(\surfk)$
and $\epsilon_1(V)=\epsilon_2(V)$ for all $V\in\uk_2\cap\kui$.

Denote by $\FQuads(\surfx)$ the moduli space of quadratic differentials on $\surfx$.
Then the $\MCG(\surfk)$-action on $\FQuad(\surfk)$ upgrades to a $\MCG(\surfx)$-action on $\FQuads(\surfx)$,
which is free only when restricted to the complete part.
Here, $\MCG(\surfx)$ is defined as in \eqref{eq:MCGx}.
%where $\MCG(\surfx)$ is defined as
%\begin{equation}\label{eq:MCGx}
%    \MCG(\surfx)=\MCG(\surf^{\yue,\kui})\ltimes\ZZv.
%\end{equation}

%\begin{definition}\label{def:Quadh}
The \emph{moduli space of signed quadratic differentials on $\surfx$} is defined to be the orbifold
\begin{gather}
    \Quad(\surfx)\colon=\FQuads(\surfx)/\MCG(\surfx).
\end{gather}
%\end{definition}
Another framed version we will consider is
%Instead of calculating the fundamental group of $\Quad(\surfx)$,
%let us look at a
the $\MCG(\surfk)$-covering of $\Quad(\surfx)$, that fits into the
%\begin{gather}\label{eq:FQuadh}
%  \FQuad(\surfx)\colon=\FQuads(\surfx)/\ZZv.
%\end{gather}
%So we have the
following commutative diagram
\begin{equation}\label{eq:3Orbifold}
\begin{tikzpicture}[xscale=1.6,yscale=0.8,baseline=(bb.base)]
\path (0,1) node (bb) {}; % baseline
\draw (0,2) node (s0) {$\FQuads(\surfx)$}
 (0,0) node (s1) {$\FQuad(\surfx)$}
 (2.5,1) node (s2) {$\Quad(\surfx)$};
\draw [-stealth, font=\scriptsize]
 (s0) edge node [left] {$\ZZv$} (s1)
 (s0) edge node [above] {$\MCG(\surfx)$} (s2)
 (s1) edge node [below] {$\MCG(\surfk)$} (s2);
\end{tikzpicture}
\end{equation}

Similar to a comment in \cite[\S~6.3]{BS}, that
the only difference between the spaces $\FQuad(\surfx)$ and $\FQuad(\surfk)$ is some extra orbifolding along the incomplete locus $\on{IC}(\surfk)$.

\begin{example}\label{ex:D2+}
Following \Cref{ex:D2}, we have
\[
    \FQuads(\surf^{\emptyset,\vortex})=\big(\ZZ_2\times\FQuad(0;1,-2,-3)\big)\bigsqcup\FQuad(0;-1,-3).
\]
The complete part $\ZZ_2\times\FQuad(0;1,-2,-3)$ is a double cover of $\FQuad(\surfp)$.
For instance, the purple loop in the right picture of \Cref{fig:2loop} lifts to the purple loop in \Cref{fig:foli}.
\end{example}
%=========================================================
%\subsection{Orbifold fundamental groups}\
%=========================================================

%=========================================================
\paragraph{\textbf{Orbifold fundamental groups}}\

Let $\mu^2(\kui)$ be the set of squares of L-flips in $\mu(\kui)$ and
\begin{gather}\label{eq:L-yue2}
  \LLk\colon=\< L_{\lp}^2 \mid \lp\in \LA(\sop,\kui)  \> \quad (\lhd \on{L}(\kui) \lhd \SBr(\sop)).
\end{gather}

\begin{theorem}\label{thm:pi1-v}
The fundamental group of $\FQuads(\surfk)$ is
\begin{gather}\label{eq:long3}
  \pi_1\FQuad(\surfx)=\pi_1\FQuad(\surfp)/\mu^2(\kui)\cong\MT(\sop)/\LLk.
\end{gather}
\end{theorem}
\begin{proof}
The top strata $R(\surfk)$ in the incomplete locus $\on{IC}(\surfk)$ has precisely extra $\ZZ_2$-orbifolding.
Other part of $\on{IC}(\surfk)$ will have higher $\ZZ_2^{\oplus k}$-orbifolding, for some $k\ge2$,
but they don't effect the fundamental group due to co-dimension reason.

By the proof of \Cref{pp:tech}, we know that $\mu^2(\kui)$ are all the loops in $\FQuad(\surfp)$
get killed after partial compactification and orbifolding. Thus the theorem follows as \Cref{thm:pi1-kui}.
\end{proof}

%=========================================================
\subsection{Deformed AJ maps}\label{sec:dmsx}\
%=========================================================

\paragraph{\textbf{Turning some punctures into vortices}}\

Let $\sok$ be the DMSp with a fixed partition $\sun=\yue\cup\kui$ of punctures,
which differs from $\sop$ when considering mapping class groups,
similar to \Cref{rem:k.vs.x}.
Namely, $\MCG(\sok)$ requires to preserve $\yue$ and $\kui$ setwise separately and
$\MCG(\sop)$ only requires to preserve $\sun$ setwise. Hence, we have
$\MCG(\sok)\lhd\MCG(\sop)$.
Note that such a change does not effect surface braid group and mixed twist groups, e.g. $$\SBr(\sok)=\SBr(\sop) \quad\text{and}\quad \MT(\sok)=\MT(\sop).$$

Let $\sox$ be the DMS with mixed punctures \& vortices (DMSx),
obtained from $\sok$ by turning the punctures in $\kui$ into vortices $\vortex$.
The difference (when considering symmetry groups) between $\sox$ and $\sok$ is
we include the following extra relations for $\sox$:
\begin{itemize}
  \item $L_{\lp}^2=1$ for any $\lp\in \LA(\sop,\kui)$.
\end{itemize}
For $\on{Sym}^*$ being $\MCG,\SBr$ or $\MT$, the subgroup $\LLk$ is normal in $\on{Sym}^*(\sok)$
and thus we define
\[
    \on{Sym}^*(\sox)\colon=\on{Sym}^*(\sok)/\LLk.
\]
In the case when $\LLk$ is not a normal subgroup of some symmetry groups associated to $\sok$,
we need to modify the quotient a bit. More precisely:
\begin{itemize}
  \item Let $\cok=\LLk\cap\MTsoy$ and define
  \begin{gather}
        \MTsoyk\colon=\MTrel{\yue}/\cok.
  \end{gather}
  \item Let $\Ho{1}^2(\kui)=\LLk/\cok$ be the abelian subgroup of $\Ho{1}(\kui)$,
  generated by the square of the standard generators of $\Ho{1}(\kui)$.
  Define $$\Ho{1}(\vortex)=\Ho{1}(\kui)/\Ho{1}^2(\kui)\cong\ZZv.$$
  We will write $\Ho{1}(\surfv)\cong\Ho{1}(\surf)\oplus\Ho{1}(\vortex)$.
\end{itemize}
Now, we can quotient the whole octahedron in \Cref{fig:3x3} by $\LLk$ and obtain a short exact sequence of octahedrons in \Cref{fig:3Octa}.
The groups in the quotient octahedron are the ones for DMSx $\sox$.
%which we will use to analysis the fundamental groups of quadratic differentials on $\surfx$.

\begin{figure}[htb]
\pgfmathsetmacro{\r}{43}
\centering\makebox[\textwidth][c]{
\begin{tikzpicture}[scale=1.4,font=\scriptsize,rotate=-\r]
\begin{scope}[shift={(\r:0)}]
  \draw[]
        (150:2)node(a1){$\cok$}(90:1)node[red,,](a2){$\LLk$}
        (30:2)node(a3){$\Ho{1}^2(\kui)$}
        (0,0)node(b1){$\LLk$}(-30:1)node(b2){$\Ho{1}^2(\kui)$}(0,-2)node(b3){1};
\end{scope}
\begin{scope}[shift={(\r:4)}]
  \draw[] (150:2)node(c1){$\MTsoy$}(90:1)node[red,,](c2){$\MT(\sop)$}
        (30:2)node(c3){$\Ho{1}(\kui)$}
        (0,0)node(d1){$\SBr(\sop)$}(-30:1)node(d2){$\Ho{1}(\surf^{\kui})$}
        (0,-2)node(d3){$\Ho{1}(\surf)$};
\end{scope}
\begin{scope}[shift={(\r:8)}]
  \draw[] (150:2)node[red,,](e1){$\MTsoyk$}(90:1)node[red,,](e2){$\MT(\sox)$}
        (30:2)node[red,,](e3){$\Ho{1}(\vortex)$}
        (0,0)node(f1){$\SBr(\sox)$}(-30:1)node(f2){$\Ho{1}(\surf^{\vortex})$}
        (0,-2)node(f3){$\Ho{1}(\surf)$};
\end{scope}
\draw[{Hooks[right]}-stealth]
    (a1)edge(a2)edge(b1)    (a2)edge(b1) (a3)edge(b2)
    (c1)edge(c2)edge(d1)    (c2)edge(d1) (c3)edge(d2)
    (e1)edge(e2)edge(f1)    (e2)edge(f1) (e3)edge(f2);
\draw[->>,>={Stealth[scale length=.4]}]
    (a2)edge(a3)    (b1)edge(b2)    (b1)edge(b3)(b2)edge(b3)
    (c2)edge(c3)    (d1)edge(d2)    (d1)edge(d3)(d2)edge(d3)
    (e2)edge node[red,above]{$\pi_\varepsilon$}(e3)    (f1)edge(f2)    (f1)edge(f3)(f2)edge(f3);
\foreach \j in {1,...,2}
{\draw[{Hooks[right]}-stealth,cyan,thick,opacity=.8]
    (a\j)edge(c\j) (b\j)edge(d\j);
    \draw[->>,>={Stealth[scale length=.4]},cyan,thick,opacity=.8](c\j)edge(e\j) (d\j)edge(f\j);}
    \draw[{Hooks[right]}-stealth,cyan,thick,opacity=.8](a3)edge(c3) ;
    \draw[->>,>={Stealth[scale length=.4]},cyan,thick,opacity=.8](c3)edge(e3) ;
    \draw[cyan,thin](d3)edge[double](f3);\draw[cyan](b3)edge[dotted](d3);
\end{tikzpicture}
}
\caption{A short exact sequence of octahedrons}\label{fig:3Octa}
\end{figure}
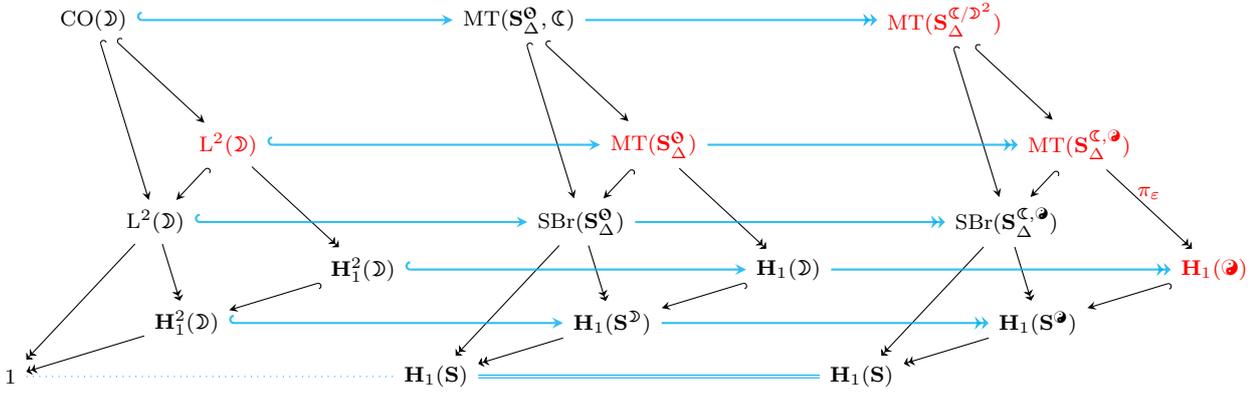

In particular, we have a short exact sequence
\begin{equation}\label{eq:ses H1-xx}
  \begin{tikzcd}[column sep=17,row sep=10]
    1 \ar[r] & \MTsoyk \ar[r] & \SBr(\surfx) \ar[r,"\AJ^\vortex"] & \Ho{1}(\surfv) \ar[r]& 1.
  \end{tikzcd}
\end{equation}
When $\surfp$ is a twice-punctured $(n-2)$-gon,
this short exact sequence specializes to \cite[Thm.~1.1 and Thm.~1.2]{All}
for Euclidean Artin braid groups. More precisely, let $\sun=\{P_1,P_2\}$.
Then
\begin{equation}\label{eq:BCD}
  \MTsoyk=
    \begin{cases}
      \Br(\widetilde{C_n}), & \mbox{if $(\yue,\kui)=(\sun,\emptyset)$}, \\
      \Br(\widetilde{B_n}), & \mbox{if $(\yue,\kui)=(\{P_1\},\{P_2\})$}, \\
      \Br(\widetilde{D_n}), & \mbox{if $(\yue,\kui)=(\emptyset,\sun)$}.
    \end{cases}
\end{equation}

\paragraph{\textbf{Deformed main theorem}}\

Now we can reformulate \Cref{thm:pi1-v} as a deformation of \Cref{thm:main}.
\begin{theorem}\label{thm:Dmain}
Let $\sox$ be a DMSx. Then
\begin{gather}\label{eq:Dmain}
    \pi_1\FQuad(\surfx)=\MT(\sox)=
        \ker\left( \SBr(\sox)\xrightarrow{\AJ}\Ho{1}(\surf) \right).
\end{gather}
\end{theorem}
\begin{remark}\label{rem:final}
Equivalently, we have
\[
    \pi_1\FQuads(\surfx)=\MTsoyk=
        \ker\left( \SBr(\sox)\xrightarrow{\AJ^\vortex}\Ho{1}(\surfv) \right).
\]
Then any non-exceptional Euclidean Artin braid groups can be realized as $\pi_1\FQuads(\surfx)$
by \eqref{eq:BCD} and by \cite[(8.6)]{QQ} for $\Br(\widetilde{A_n})$.
\end{remark}

\appendix
%=========================================================
%=========================================================
\section{Tilting and stability structures}
%=========================================================
%=========================================================
\subsection{Simple tilting theory}\label{sec:prel}\
%=========================================================

We will first recall basic notions and notations about triangulated categories,
tilting theory and the 3-Calabi-Yau categories associated to a DMSp.
Then we review the categorification of flip graphs from \cite{CHQ}, cf. \cite{QQ,Chr}.

We work over algebraically closed field $\k$ and $\k$-linear additive categories.
A category usually means an abelian or triangulated category.
Most stuffs in this subsection are classical, we refer to \cite{B1,KQ1,Q1} for general references.

We write $\hua{C}=\hua{T}\perp\hua{F}$ for a torsion pair in a category $\hua{C}$.
An object $S$ in $\hua{C}$ is \emph{rigid} if $\Ext^1(S,S)=0$.

For an abelian category $\h$, denote by $\Sim\h$ the set of simples of $\h$.
Such an $\h$ is called \emph{finite} if $\Sim\h$ is finite and generates $\h$
(so that $\h$ is a length category).

Let $\D$ be a triangulated category.
A \emph{t-structure} $\sli$ of $\D$ is the torsion part of a torsion pair $\D=\sli\perp\sli^\perp$
such that $\sli[1]\subset\sli$.
It is bounded if
\[
  \hua{D}= \displaystyle\bigcup_{i,j \in \ZZ} \hua{P}^\perp[i] \cap \hua{P}[j],
\]
The \emph{heart} of a bounded t-structure $\sli$ is $\h=\sli^\perp[1]\cap\sli$, which
is abelian and determines $\sli$ uniquely.
%The heart induces a homology on $\D$. Namely,
For any object $E\in\D$, there is a HN-filtration
\begin{equation}\label{eq:HN}
\begin{tikzcd}[column sep=.8pc]
    0=E_0 \ar[rr] && E_1 \ar[rr]\ar[dl] && E_2 \ar[r]\ar[dl] & \ldots \ar[r] & E_{l-1} \ar[rr] && E_l=E \ar[dl] \\
        &F_1 \ar[ul,dashed]&&F_2 \ar[ul,dashed]&&&& F_l \ar[ul,dashed]
\end{tikzcd}
\end{equation}
with factors $F_i\in\h[k_i]$ for $k_1>k_2>\cdots>k_l$.

Given a torsion pair $\h=\hua{T}\perp\hua{F}$ in $\h$,
there is a new heart $\h^\sharp$, known as the \emph{forward (HRS) tilt} of $\h$ with respect to the torsion pair
(cf. \cite{HRS}), which admits a torsion pair $\hua{F}[1]\perp\hua{T}$ (and hence determined by them).
The dual notion of forward tilting is backward tilting.
For $\h=\hua{T}\perp\hua{F}$, one has $\h^\flat=\h^\sharp[-1]=\hua{F}\perp\hua{T}[-1]$.
%Here, HRS stands for Happel-Reiten-Smal{\o}.

There is a natural partial order on the set of hearts, i.e.
\begin{equation}\label{eq:order}
  \h_1 \leq \h_2\; \Longleftrightarrow\; \hua{P}_2\subset\hua{P}_1 \;\Longleftrightarrow\;
    \hua{P}^\perp_1\subset\hua{P}^\perp_2.
\end{equation}
Another way to characterize forward tilts of a heart is that they are precisely all heart in the interval
$[\h,\h[1]]$, i.e. (cf. \cite[\S ~3]{KQ1})
\begin{equation}\label{eq:[1]}
  \h'=\h^\sharp \text{ w.r.t. $\h=\hua{T}\perp\hua{F}$ if and only if}  \h\le \h'\le\h[1].
\end{equation}
Indeed, in such a case, one has $\hua{T}=\h'\cap\h$ and $\hua{F}=\h'[-1]\cap\h$.

\begin{definition}\cite{KQ1}
A forward tilting $\h\to \h^\sharp$ with respect to $\h=\hua{T}\perp\hua{F}$ is \emph{simple} if the corresponding torsion part $\hua{F}$ is generated by a single simple $S\in\Sim\h$.
Denote by $\h^\sharp_S$ the simple forward tilt in this case and by
$\mu_S^\sharp\colon \h\to \h^\sharp_S$ the simple forward tilting.

A backward tilting is simple if the corresponding torsion part is generated by a single simple $S$
and the backward tilt will be denoted by $\h^\flat_S$.
The backward tilting is the inverse operation of forward tilting and we have
$\mu_{S[-1]}^\sharp\colon \h^\flat_S\to \h$.

The \emph{exchange graph (of hearts)} $\EG(\D)$ of a triangulated category $\D$ is the directed graph whose vertices are hearts and whose edges are simple forward tilting.
When $\D$ admits a canonical heart $\h$,
denote by $\EGp(\D)$ the principal component of $\EG(\D)$, i.e. the connected component containing $\h$.
\end{definition}

We often interested in intervals of in $\EGp(\D)$.
Denote by $\EGp[\h_1,\h_2]$ the connected component of the full subgraph $\EGp(\D)\cap[\h_1,\h_2]$ containing $\h_1$.

There are couple of simple tilting formulae,
e.g \cite[Prop.~5.4]{KQ1} for the rigid case and \cite[\S~A.2]{CHQ} for a slight generalization.

In many good scenarios, $\EGp(\D)$ will consists of finite hearts and the formula apply to any of its edges.
In that case, $\EGp(\D)$ is $(n,n)$-regular, e.g. when $\D=\D^b(Q)$ for an acyclic quiver $Q$ or $\D=\D_N(Q)$ for the $N$-Calabi-Yau version of $\D^b(Q)$.

As rigid simple tilting is very nice and we introduce the rigid exchange graph
$\REG(\D)$ to be the subgraph of $\EG(\D)$ whose edge are simple tilting with respect to rigid simples.
Similar for other variation of $\EG$, e.g. $\EGp$, $\EGp[\h_1,\h_2]$, etc.
Note that the formula above only ensures the finiteness of the tilted heart, not the rigidity of the tilted simples.
Hence, $\REGp(\D)$ may not be equal $\EGp(\D)$ even if it is $(n,n)$-regular.

We recall another technical result from \cite[Lem.~3.8]{KQ1}. %, Lem.~5.6 and Cor.~5.7
\begin{lemma}\label{lem:ss}
Let $\h\in\EGp[\h_0,\h_0[1]]$. % (for some $m\in\NN$) and
%$\Ho{\bullet}$ the homology with respect to $\h_0$.
Then, for any rigid $S\in\Sim\h$,
it must be in one of $\h_0$ or $\h_0[1]$ and
\begin{itemize}
\item $\tilt{\h}{\sharp}{S}\in \EGp[\h_0,\h_0[1]]$
    if and only if $S\in\h_0$ if and only if $S\notin\h_0[1]$.
\item $\tilt{\h}{\flat}{S}\in \EGp[\h_0,\h_0[1]]$
    if and only if $S\notin\h_0$ if and only if $S\in\h_0[1]$.
\end{itemize}
In other words, exactly one of $\tilt{\h}{\pm\sharp}{S}$ is still in $\EGp[\h_0,\h_0[1]]$.
\end{lemma}

%=========================================================
\subsection{Stability structures on triangulated categories}\label{sec:stab}\
%=========================================================

Let $\D$ be a triangulated category.
Recall the basic notion and result from \cite{B1}.

\begin{definition}\cite{B1}\label{def:stab}
A \emph{stability condition} $\sigma=(Z,\sli)$ consists of a \emph{central charge} $Z\in\Hom(\Grot(\D),\CC)$ and a $\RR$-family $\sli_\RR=\{\sli(\phi)\mid \phi\in\RR\}$ of additive (in fact, abelian) subcategories (known as the \emph{slicing}), satisfying (cf. \cite{B1} for more details):
\begin{itemize}
  \item if $0 \neq E \in \sli(\phi)$,
  then $Z(E) = m(E) \exp(\phi  \pi \mathbf{i} )$ for some $m(E) \in \RR_{>0}$;
  \item $\sli(\phi+1)=\sli(\phi)[1]$, for all $\phi \in \RR$;
  \item $\Hom(\sli(\phi_1),\sli(\phi_2))=0$ if $\sli(\phi_1)>\sli(\phi_2)$;
  \item each nonzero object $E \in \hua{D}$ admits a \emph{Harder-Narashimhan} filtration \eqref{eq:HN}
  with factors $F_i\in\sli(\phi_i)$ for $\phi_1>\cdots>\phi_l$.
  The upper and lower phases of $E$ are $\phi_1$ and $\phi_l$ respectively.
  \item the so-called \emph{support property}.
\end{itemize}
\end{definition}

Objects in the slice $\sli(\phi)$ are called \emph{semistable} objects with phase $\phi$;
simple objects in $\sli(\phi)$ is called \emph{stable} objects with phase $\phi$.
For an interval $J$, denote by $\sli(J)$ the subcategory consisting of objects whose upper and lower phases
are in $J$.
The \emph{heart} $\h_\sigma$ of a stability condition $\sigma=(Z,\sli)$ is the abelian category $\sli(0,1]$.

\begin{theorem}\cite[Thm.~1]{B1}
All stability conditions on a triangulated category $\D$
form a complex manifold, denoted by $\Stab(\D)$, with local coordinate given by the central charge.
\end{theorem}

There are two natural actions on $\Stab(\D)$, the $\CC$-action given by
\[
    s \cdot (Z,\sli)=(Z \cdot e^{-\mathbf{i} \pi s},\sli_{\on{Re}(s)}),
\]
for $\sli_x(\phi)=\sli(\phi+x)$ and the $\Aut\D$-action given by
\[
    \Phi  (Z,\sli)=\big(Z \circ \Phi^{-1}, \Phi (\sli) \big).
\]

%=========================================================
\paragraph{\textbf{Generic-finite component}}\

For a finite heart $\h$,
the stability conditions supported in $\h$ (i.e. with $\h_\sigma=\h$) is parameterized by
the central charges $Z(S)\in\uhp$ for $S\in\Sim\h$.
Thus there is a cell/subspace $\cub(\h)\isom\uhp^{\Sim\h}$ in $\Stab(\D)$ consisting of
those stability conditions supported on $\h$.

For each tilting $\h\to\h'=\tilt{\h}{\sharp}{S}$,
there is co-dimensional one wall where $Z(S)\in \RR_{>0}$,
\[
 \partial^\sharp_S\cub(\h) = \partial^\flat_{S[1]}\cub(\h')
 = \overline{\cub(\h)}\cap\overline{\cub(\h')}.
\]

For a connected component $\EGp$ in $\EG(\D)$, define
\[
    B_0(\EGp)=\bigcup_{\h\in\EGp} \cubc(\h)
\]
and
\[
    B_2(\EGp)=B_0(\Gamma)\cup\bigcup_{
        \begin{smallmatrix}\h\in\EGp\\ S\in\Sim\h\end{smallmatrix}}
        \partial^\sharp_S\cub(\h).
\]

A connected component $\Stap$ in $\Stab(\D)$ is \emph{generic-finite}
if $\Stap=\CC\cdot B_2(\EGp)$ for some connected component $\EGp$ of $\EG(\D)$
consisting of finite hearts.

%=========================================================
%=========================================================
\section{Marked surfaces with mixed punctures \& vortices}\label{sec:MSx}
%=========================================================
%=========================================================
Recall that there is a partition $\sun=\yue\cup\kui$.
We add extra $\ZZ_2$-symmetry to punctures in $\kui$, call them \emph{vortices}
and use $\vortex$ to denote the set of vortices.
Let $\surfx$ be the marked surface with mixed punctures \& vortices (MSx).
More precisely, what this means is:
\begin{itemize}
  \item Categorically, we will add terms around each vortex in $\vortex$ for quivers with potential associated to triangulations and deform $\Dsop$ into $\Dsox$.
  \item Geometrically, we will allow collisions between simple zeros in $\Tri$ and double poles in $\vortex$,
  so that we partially compactify the moduli space $\FQuad(\surfp)$ into $\FQuads(\surfx)$ by adding strata with simple poles from collision (with orbifolding).
\end{itemize}
Then we show that the corresponding quotient space $\Stab\Dsox/\Autp$ is isomorphic to
the unframed version of $\FQuads(\surfx)$.
When $\Dsox$ is fully deformed, i.e. taking $\vortex=\sun$,
We will call the corresponding MSx by MSv (with notation $\surfv$).
In such a case, $\Quad(\surfv)$ becomes the moduli space in \cite{BS}.

\begin{remark}\label{rem:assumption}
We assume the following when considering MSx:
\begin{itemize}
  \item $\partial\surfx\neq\emptyset$.
  \item $n\ge2$, where $n$ is defined in \eqref{eq:n}.
\end{itemize}
The main results are covered in \cite{KQ2} when $\sun=\emptyset$
and hence we usually assume $\sun\neq\emptyset$ in the proofs.
\end{remark}
%=========================================================
\subsection{Cluster-ish combinatorics for MSx}\label{sec:cluster}
%=========================================================

%=========================================================
\paragraph{\textbf{Signed triangulations of MSx}}\

A \emph{signed triangulation} $\xT$ of a MSx $\surfx$ consists of
an (ideal) triangulation $\RT$ of MSp $\surfp$ with a \emph{sign}
\begin{equation}\label{eq:epsilon}
    \epsilon\colon\vortex\to\{\pm1\}\cong\ZZ_2.
\end{equation}
The set of signs can be identified with $\ZZv$.
Two signed triangulations $\xT_1=(\RT_1,\epsilon_1)$ and $\xT_2=(\RT_2,\epsilon_2)$ are equivalent if
\begin{itemize}
  \item $\RT_1=\RT_2$ and
  \item $\epsilon_1(V)=\epsilon_2(V)$, for any vortex $V$ that is not in a self-folded triangle.
\end{itemize}
Namely, we have the (local) equivalence shown in \Cref{fig:equiv}.
In other words, the sign at a vortex which is in a self-folded triangulation can be chosen arbitrarily
without changing the equivalent class of the signed triangulation.
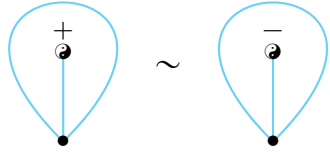
\begin{figure}[hb]
\begin{tikzpicture}[scale=1.2];
\clip(-.7,-1.1)rectangle(3,.7);
\begin{scope}[shift={(0,0)}]
\draw[\qlan,thick](0,0)to(0,-1)
    .. controls +(45:2.9) and +(135:2.9) .. (0,-1);
%\draw(0,.25)\ww(0,0.75)\ww;
\draw[thick](0,-1)\nn
    (0,0)node[white]{\tiny{$\blacksquare$}} node[]{\scriptsize{\vortex}}node[above]{$+$};
\end{scope}
\draw(2.3/2,0)node[font=\Large,below]{$\sim$};
\begin{scope}[shift={(2.3,0)}]
\draw[\qlan,thick](0,0)to(0,-1)
    .. controls +(45:2.9) and +(135:2.9) .. (0,-1);
%\draw(0,.25)\ww(0,0.75)\ww;
\draw[thick](0,-1)\nn
    (0,0)node[white]{\tiny{$\blacksquare$}} node[]{\scriptsize{\vortex}}node[above]{$-$};
\end{scope}
\end{tikzpicture}
\caption{Equivalent signed triangulations (local)}
\label{fig:equiv}
\end{figure}

There is a forward flip between signed triangulations $\mu^\sharp\colon \xT_1\to\xT_2$
if one can choose the (representatives of) signs $\epsilon_1=\epsilon_2$,
such that there is a usual forward flip between the corresponding triangulations $\mu^\sharp\colon \RT_1\to\RT_2$
or there is a forward L-flip between the corresponding triangulations, where the corresponding isolated puncture is not a vortex.
For more details, cf. \cite{FST}.
Then we have the corresponding notation of (oriented) cluster-ish exchange graph $\FG^\pm(\surfx)$ for MSx.

Alternatively/concretely, $\FG^\pm(\surfx)$ can be constructed as follows.
Note that when we write $\kui$ instead of $\vortex$,
we mean there is no $\ZZ_2$-symmetry (i.e. signs) attached.
\begin{itemize}
  \item Denote by $\FG(\surfk)$ the flip graph for $\surfk$, which is a subgraph of $\FG(\surfp)$,
      by deleting L-flips where the corresponding enclosed puncture is in fact a vortex.
  \item For each sign $\epsilon\in\ZZv$, take a copy $\FG^\epsilon(\surfk)=\FG(\surfk)\times\{\epsilon\}$.
  \item We take the union of $\FG^\epsilon(\surfk)$, for all $\epsilon\in\ZZv$,
  and identify vertices/edges by the equivalent relation $\sim$ above.
\end{itemize}
So we have
\begin{equation}\label{eq:EG x}
  \FG^\pm(\surfx)=\bigcup_{\epsilon\in\ZZv}\FG^{\epsilon}(\surfk)
    = \big( \FG(\surfk)\times\ZZv \big) \big/ \sim.
\end{equation}

There is also an unoriented version of $\FG^\pm(\surfx)$,
which can be obtained from it by deleting all loop edges and replacing each 2-cycle with an unoriented edge.
Equivalently, one has
\begin{equation}\label{eq:uEG x}
  \uEG^\pm(\surfx)=\bigcup_{\epsilon\in\ZZv}\uEG^{\epsilon}(\surfp)
    = \big( \uEG(\surfp)\times\ZZv \big) \big/ \sim.
\end{equation}

For instance, the left picture of \Cref{fig:op-tri} shows the building process of $\FG(\surfv)$
for a once-puncture triangle with $\vortex=\sun$.
When $\vortex=\sun$, we have $\FG^\pm(\surfv)=\CEG(\surfp)$ and $\uEG^\pm(\surfv)=\uCEG(\surfp)$ are the
oriented and unoriented cluster exchange graph associated to $\surfp$ respectively.

Building on \Cref{thm:H},
we have the following result as a variation of \Cref{thm:H}, using the same proof of \cite[Thm.10.2]{FST}.

\begin{proposition}\label{pp:FST}
$\pi_1\uEG^\pm(\surfx)$ is generated by squares and pentagons,
which locally look like the (unoriented) squares/pentagons in \Cref{fig:456},
or the one in \Cref{fig:new SQ}.
\end{proposition}

\begin{figure}[bth]\centering
  \includegraphics[width=5cm]{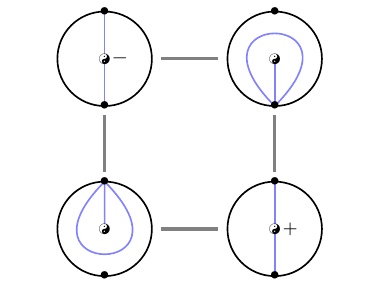}
\caption{The square involving vortex}\label{fig:new SQ}
\end{figure}

%=========================================================
\paragraph{\textbf{Mutation, reduction and right-equivalence}}\

There are the notions of \emph{right-equivalence}, \emph{reduction} and \emph{mutation} introduced by \cite{DWZ}, cf. \cite[\S~2.1]{KY} for more details.
\begin{itemize}
  \item Two quivers with potential are right-equivalent
    if there is an algebra isomorphism between the corresponding completed path algebras
    such that it is the identity when restricted to the vertices
    and induces cyclically equivalence between the potentials.
  \item The reduction is an operation on right-equivalent classes of quivers with potential.
    It reduces a quiver with potential by killing some 2-cycles (and change the potential according).
    A quiver with potential $(Q,W)$ is called \emph{reduced} if $\partial_a W$ is contained in $\mathfrak{m}^2$ for any $a\in Q_1$.
    By \cite[Thm.~4.6]{DWZ}, the reduction is well-defined on the right-equivalence classes of quivers with potential.
  \item The mutation $\mu_i$ is an operation on a quiver with potential $(Q,W)$ at a vertex $i$ of $Q$;
    it produces another quiver with potential.
     By \cite[Thm.~5.2 and Thm.~5.7]{DWZ}, the mutation is also well-defined up to right-equivalence,
     which is an involution.
\end{itemize}

%=========================================================
\paragraph{\textbf{Quivers with potential associated to triangulated MSx}}\

Let $\xT=(\RT,\epsilon)$ be a signed triangulation.
Define a quiver with potential $(\widehat{\Qx},\widehat{\Wx})$ as follows:
\begin{itemize}
  \item The vertices are open arcs in $\RT$.
  \item For each triangle $T$ of $\RT$, which is not a self-folded triangle enclosing a vortex,
  there is a clockwise 3-cycle between the edges of $T$ (so that the arrow corresponds to anticlockwise angles in $T$).
  \item For each vortex $V$ that is enclosed in a self-folded triangle $T$
  with self-folded edge $i$ and loop edge $l$, we call these two edges twins.
  Add an arrow between $i$ and (from/to) $j$ whenever there is an arrow between $l$ and (from/to) $j$, for $j\ne l$ in $\RT$.
  So besides the usual (type P-1) puzzle piece (i.e. a triangle without involving a vortex in a self-folded triangle--see any triangle in \Cref{fig:QP}),
  there are three types of other puzzle pieces in a triangulation $\RT$ shown in \Cref{fig:type IV}.
    \begin{figure}[ht]\centering
        \includegraphics[width=\textwidth]{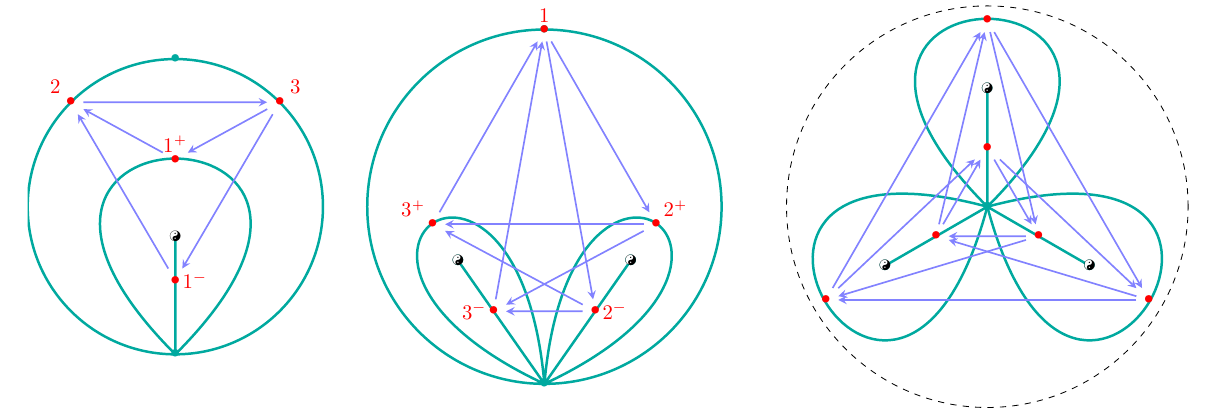}
    \caption{The quivers with potential associated to types P-2/P-3/P-4 puzzle pieces with vortices}
    \label{fig:type IV}
    \end{figure}
  \item The unreduced potential is
  \[
    \widehat{\Wx}=\sum_T W_T^\vortex- \sum_V W_V^\vortex,
  \]
  where
    \begin{itemize}
      \item the first sum is over triangles which is not a self-folded triangle enclosing a vortex,
        each team $W_T^\vortex$ is the sum of all corresponding 3-cycles (one can replacing an edge of $\RT$ with its twin).
        This means in type P-$k$, there are $2^{k-1}$ 3-cycle terms in $W_T^\vortex$ for the non self-folded triangle $T$.
      \item the second sum is over all vortices which are not enclosed in a self-folded triangle $T$,
        each term $W_V^\vortex$ is an anticlockwise cycle around $V$, which needs to skip some of the arrows.
        More precisely, if there are twins incident at $V$, which corresponding to another vortex $V_0$,
        then
        \begin{itemize}
          \item $W_V^\vortex$ skip the self-folded edge at $V_0$ and only count loop-edge once if $\epsilon(V_0)=1$;
          \item $W_V^\vortex$ skip the loop-edge and only count self-folded edge if $\epsilon(V_0)=-1$.
        \end{itemize}

    \end{itemize}
\end{itemize}

\begin{definition}
Let $(\Qx,\Wx)$ be the quiver with potential associated to a signed triangulation $\xT=(\RT,\epsilon)$ of $\surfx$, which is the reduced part of $(\widehat{\Qx},\widehat{\Wx})$.
\end{definition}

\begin{example}
For instance, the term $W_V^\vortex$ in the potential is
\begin{itemize}
  \item a 5-cycle for the left picture of \Cref{fig:W_V};
  \item a 10-cycle $7\to 8\to4\to1\to2\to2\to1\to3\to5\to6\to7$ for the middle picture of \Cref{fig:W_V};
  \item a 7-cycle $7\to 8\to3\to1^{\epsilon(V)}\to2\to5\to6\to7$ for the right picture of \Cref{fig:W_V},
  where $\epsilon(V)$ is the sign of the vortex $V_0$ (in a self-folded triangle) adjacent to $V$.
\end{itemize}
\begin{figure}[ht]\centering
    \includegraphics[width=\textwidth]{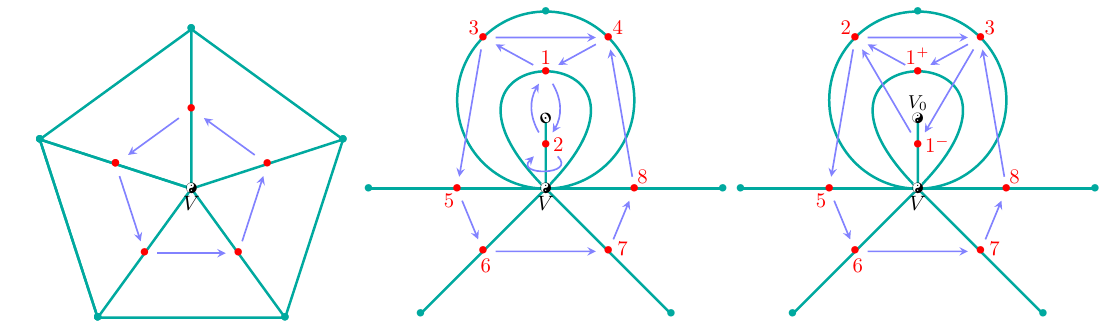}
\caption{The extra term in the potential for each vortex (not in a self-folded triangle)}
\label{fig:W_V}
\end{figure}

\Cref{fig:reduction} shows a reduction of quiver with potential,
where one needs to kill a 2-cycle around a vortex (and change the potential accordingly).
For more examples, see \cite{LF}.
\begin{figure}[ht]\centering
    \includegraphics[width=10cm]{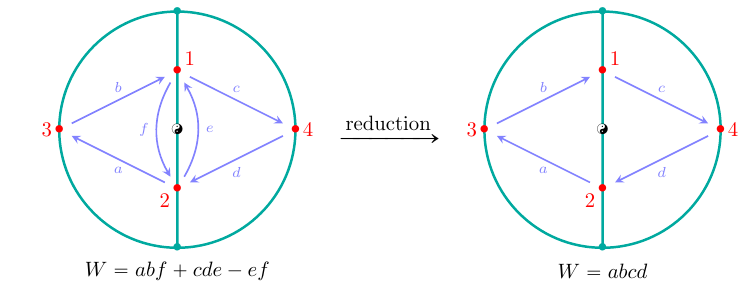}
\caption{Reduction of a quiver with potential}
\label{fig:reduction}
\end{figure}
\end{example}

Note that:
\begin{itemize}
  \item the quiver with potential is the one we considered in \Cref{def:QP} if $\vortex=\emptyset$;
  \item the quiver with potential is non-degenerate if and only if $\vortex=\sun$; in which case,
  it is the one constructed in \cite{LF} for classical cluster theory.
\end{itemize}

\begin{lemma}\label{lem:right-eq}
Any usual flip of signed triangulations induces mutations
between the associated quivers with potential (up to right-equivalence).
\end{lemma}
\begin{proof}
This is a local issue. We only need to check the following.
\begin{itemize}
  \item For the usual mutation corresponding to flipping diagonals of a quadrilateral,
  we know that the corresponding quivers with 3-cycles potential are related by mutation
  up to right equivalent as below.
  \begin{equation}\label{eq:QP mut}
    \includegraphics[width=7.5cm]{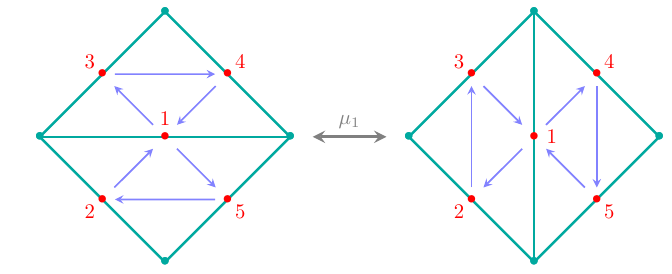}
  \end{equation}
  \item For the flip $\mu_1$ in \Cref{fig:QP}, we cut the self-folded edge
  (corresponding to vertex 2 in the two pictures in the middle)
  and the mutation becomes the case above.
  Conversely, by identifying vertex $2$ and $5$ in \eqref{eq:QP mut},
  we obtain the mutation/flip $\mu_1$ in \Cref{fig:QP} .
  \item For other cases, we can always cut all the self-folded edge (provided the corresponding enclosed puncture is a not a vortex) as above, and reduce to the classical cluster case.
      Then the lemma follows from \cite{LF}.
\qedhere
\end{itemize}
\end{proof}

%=========================================================
\paragraph{\textbf{The symmetry group}}\

Denote by
\begin{equation}\label{eq:MCGx}
    \MCG(\surfx)=\MCG(\surf^{\yue,\kui})\ltimes\ZZv
\end{equation}
the \emph{signed mapping class group} of $\surfx$,
where $\MCG(\surf^{\yue,\kui})$ is the mapping class group of $\surfp$
(that preserves $\M,\yue$ and $\kui$ setwise, respectively)
and $\ZZv$ acts on the signs of vortices.

Note that we have rule out all the non-amenable cases (in the sense of \cite[Def.~9.3]{BS})
by the assumptions in \Cref{rem:assumption}.
Then we have the analogue of \cite[Prop.~8.5 and Prop.~8.6]{BS}.

\begin{lemma}\label{lem:isoT}
Two signed triangulations $\xT_i$ are in the same $\MCG(\surfx)$-orbit if and only if
  the associated quivers $Q_{\xT_i}^{\mix}$ are isomorphic.
In particular, $\MCG(\surfx)$ acts freely on $\FG^\pm(\surfx)$.
\end{lemma}
\begin{proof}
For the first claim, we only need prove the \emph{if} part.
In other words, we can reconstruct the triangulation from the associated quiver.
In our case, a vertex $i$ of $Q$ with loops precisely corresponds to
a self-folded edge of the triangulation, whose corresponding isolated puncture is in $\yue$.
Moreover, there is exactly one vertex $j\ne i$ in $Q$ that is incident with $i$
(more precisely there is a 2-cycle between them).
Deleting all such vertices (with loops) from the quiver corresponds to cut the surface along all those self-folded edges.
Then the reconstruction of the triangulation of cut surface can be done by \cite[Prop.~8.5]{BS}.
More, one can add back those vertices and gluing back the original surface accordingly.
Thus, the first claim follows.

The second follows as \cite[Prop.~8.6]{BS} similarly.
\end{proof}

\begin{remark}\label{rem:iff}
We also have the following:
two signed triangulations are in the same $\MCG(\surfx)$-orbit if and only if
  the associated quivers with potential are right-equivalent.
This is because the isomorphic class of quivers (associated to a fixed MSx)
and the right-equivalent class of quivers with potential (associated to the same fixed MSx)
determines each other.
\end{remark}

%=========================================================
\subsection{Deformed Calabi-Yau categories}\
%=========================================================

As in \Cref{sec:D3},
let $\Gx$ be the Ginzberg dg algebra of $(\Qx,\Wx)$ and
$\on{pvd}(\Gx)$ its the (Calabi-Yau-3) perfectly valued derived category with canonical heart $\hx$.

Let $\h$ be a finite heart.
Recall that the \emph{Ext quiver} $\hua{Q}(\h)$ is the (positively) graded quiver
whose vertices are the simples of $\h$ and
whose degree $k$ arrows $S_i \to S_j$ correspond to a basis of
$\Hom_{\D}(S_i, S_j[k])$.
Equivalently,
denote by $\hua{E}(\h)$ the dg-endomorphism algebra
$\on{\hua{E}nd}\left(\bigoplus_{S\in\Sim\h}S\right)$ for $\h$.
Then $\hua{Q}(\h)$ is the \emph{Gabriel quiver }of $\hua{E}(\h)$.
Further, for a graded quiver $Q$, define
\begin{itemize}
  \item the \emph{Calabi-Yau-$N$ double},
  denoted by $\double{Q}{N}$, to be the quiver obtained from $Q$
  by adding an arrow $j\to i$ of degree $N-k$ for each arrow $i\to j$ of degree $k$
  and adding a loop of degree $N$ at each vertex.
  \item the grading shift dual, denote by $Q^\vee$, to be the same quiver with
  grading changes $d\mapsto1-d$.
\end{itemize}

The lemma below follows from the standard simple-projective duality.

\begin{lemma}\label{lem:Ext-quiver}
The Ext-quiver of $\hx$ is the Calabi-Yau-3 double of $(\Qx)^\vee$,
i.e. $$\hua{Q}(\hx)=\double{(\Qx)^\vee}{3}.$$
\end{lemma}

\begin{remark}\label{rem:KVB}
In fact, we can also read of the information of the associated potential $\Wx$ from $\hua{E}(\h)$,
cf. \cite[App.A]{KVB}.
In other word, the right equivalent classes of QP determines and is determined
by the corresponding 3-Calabi-Yau category or its canonical heart, cf. proof of \cite[Thm.~9.9]{BS}.
\end{remark}

%=========================================================
\subsection{Triangle equivalences}\
%=========================================================

We proceed to show that flips of signed triangulations induce triangle equivalences,
which is a slightly upgraded of \cite{KY}.
Recall that an object $S$ is \emph{spherical} if
$\Hom^{\bullet}(S, S)=\k \oplus \k[-3]$
and that it then induces a \emph{spherical twist functor} $\twi_S$, defined by
\begin{equation*}%\label{eq:phi}
    \twi_S(X)=\Cone\left(S\otimes\Hom^\bullet(S,X)\to X\right).
\end{equation*}
%with inverse
%\[\twi_S^{-1}(X)=\Cone\left(X\to S\otimes\Hom^\bullet(X,S)^\vee \right)[-1].\]

\begin{proposition}\label{pp:KY}
For any flip of signed triangulation $\RT^\times_1\xrightarrow{\gamma_1}\RT^\times_2$,
there is a triangle equivalence
\begin{gather}\label{eq:KY}
    \Upsilon_{\gamma_1}\colon\on{pvd}(\Gamma_{\RT^\times_1}^{\mix})\cong\on{pvd}(\Gamma_{\RT^\times_2}^{\mix})
\end{gather}
satisfying
\begin{equation}\label{eq:tiltU}
    \Upsilon_{\gamma_1}^{-1}\left( \h_{\RT^\times_2}^\mix \right)
    =\left( \h_{\RT^\times_1}^\mix \right) ^{\sharp}_{S_{\gamma_1}}.
\end{equation}
Moreover, for a 2-cycle
$\begin{tikzcd}
    \RT^\times_1 \ar[r,"\gamma_1",shift left=1] & \RT^\times_2 %(\ne \RT^\times_1)
        \ar[l,"\gamma_2",shift left=1]
\end{tikzcd}$,
we have
$\Upsilon_{\gamma_2}\circ\Upsilon_{\gamma_1}=\twi^{-1}_{S_{\gamma_1}}$ and % \quad\text{and}\quad
$\Upsilon_{\gamma_1}\circ\Upsilon_{\gamma_2}=\twi^{-1}_{S_{\gamma_2}}.$
\end{proposition}
\begin{proof}
If the flip is non self-folded, then the proposition follows from \cite[Thm.~3.2]{KY}
(that mutation implies derived equivalence).
Thus, we only need to consider the L-flip case, where $\RT_1=\RT=\RT_2$.

By Kozsul duality (e.g. \cite[\S~5.1]{QQ}), $\on{pvd}(\Gx)$ can be realized as
the perfect derived category of the dg algebra $\per\hua{E}(\h)$.
Let $\h'$ the tilted heart in the right hand side of \eqref{eq:tiltU} and
we only need to show that $\hua{E}(\h)\cong\hua{E}(\h')$.
For the local case, i.e. $\mu_1$ for the QP in the rightmost picture of \Cref{fig:QP},
we have an isomorphism
\[  \hua{E}(\h)=\on{\hua{E}nd}(S_1\oplus S_2)\cong\on{\hua{E}nd}(S_1[1]\oplus S_2')=\hua{E}(\h')  \]
by \cite[Cor.~4.21]{CHQ}.
Moreover, $S_2'$ admits a filtration with three factors $S_1,S_1$ and $S_2$ (cf. \cite[Ex.~A.3]{CHQ}).
In general, $S_2\in\Sim\h$ is still the only simple that will change during
the forward tilting $\h\xrightarrow{S_1}\h'$ with the same formula (cf. \cite[App.~A.2]{CHQ})
and $S_2'$ is built by the same filtration as above.
As $\Hom(S_1,S)=0$ for any simple $S\in\Sim\h\setminus\{S_1,S_2\}$,
hence there is an isomorphism
\[
    \on{\hua{E}nd}(\bigoplus_{ \begin{smallmatrix}S\in\Sim\h\\S\neq S_1\end{smallmatrix}}S)\cong
    \on{\hua{E}nd}(\bigoplus_{ \begin{smallmatrix}S\in\Sim\h'\\S\neq S_1[1]\end{smallmatrix}}S).
\]
Patching together these isomorphisms, we obtain the required one $\hua{E}(\h)\cong\hua{E}(\h')$.
\end{proof}

\begin{remark}\label{rem:auto}
There is an autoequivalence of $\on{pvd}(\Gx)$
associated to each of the simple of $\hx$.
More precisely, if the simple corresponds to a vertex of $\Qx$ without loops,
then it is spherical and induces a spherical twist;
if it corresponds to a vertex with loops, then there is an autoequivalence $\Upsilon_{\gamma}$ in \eqref{eq:KY} constructed as above.
\end{remark}

By abuse of notations, we denote $\Dsox=\on{pvd}(\Gx)$, where $\RT$ is the initial triangulation.
So $\{\Dsox\}$, for various partition $\sun=\yue\cup\vortex$,
are a class of (Calabi-Yau) deformations of $\Dsop$.
The fully deformed one, denoted by $\Dsov$, is the one studied in \cite{BS}.

\begin{remark}
In DMSp case, we can identified $\on{pvd}(\Gamma_\T^\sun)$ consistently for any $\T\in\FGT(\sop)$ by \Cref{thm:CHQ1}.
In DMSx case, we are not sure if we can do the same at the moment.
Without decorations, one certainly can not identify $\on{pvd}(\Gx)$ consistently for all $\xT\in\FG^\pm(\surfx)$
(as there are monodromies).
For instance, the spherical twists in \Cref{pp:KY} are non-trivial monodromies.
\end{remark}

Denote by $\EGp(\Dsox)$ the principal component of the exchange graph of hearts of $\Dsox$.
Let $\hua{A}ut^\circ\Dsox$ be the group of autoequivalences of $\Dsox$,
consisting of those that preserve $\EGp\Dsox$
and $\on{Ng}\Dsox$ its subgroup consisting of \emph{negligible} ones, i.e. acting trivially on $\EGp\Dsox$.
So what we really need is the their quotient:
\begin{equation}\label{eq:Autp}
    \Autp\Dsox \colon= \hua{A}ut^\circ\Dsox/\on{Ng}\Dsox.
\end{equation}

%=========================================================
\subsection{Flip graphs as exchange graphs}\
%=========================================================

In this section, we prove a quotient/weak version of \eqref{eq:iso} for $\Dsox$.

Take any heart $\h$ in $\EGp(\Dsox)$,
then there is a simple tilting sequence
\[\begin{tikzcd}
    \hx=\h_1 \ar[r,-]& \h_2 \ar[r,-]& \cdots \ar[r,-]& \h_m=\h.
\end{tikzcd}\]
By \Cref{pp:KY}, it corresponds to flip sequence
\[\begin{tikzcd}
    \xT=\xT_1 \ar[r,-]& \xT_2 \ar[r,-]& \cdots \ar[r,-]& \xT_m
\end{tikzcd}\]
in $\FG(\surfp)$ so that $\h_i$ is equivalent to the standard heart $\h_{\RT^\times_i}^\mix$.
This implies that for each heart in $\EGp(\Dsox)$:
\begin{itemize}
  \item there is an associated quivers with potential (up to right-equivalence);
%  so that it is isomorphic to the (dg) module category of the corresponding Ginzburg dg algebra.
  \item there are $n$ associated autoequivalences of $\Dsox$ corresponding to its simples by \Cref{rem:auto}.
\end{itemize}
Denote by $\hua{M}T(\Dsox)$ the autoequivalence group generated by all such associated autoequivalences
and $\MT(\Dsox)$ its image in $\Autp\Dsox$.

\begin{proposition}\label{pp:tilting}
There is a natural isomorphism
\begin{gather}\label{eq:def-EG}
  \on{X}^\mix_*\colon  \FG^\pm(\surfx)/\MCG(\surfx) \xrightarrow{\cong} \EGp(\Dsox)/\Autp,
\end{gather}
sending the $\MCG(\surfx)$-orbit of a signed triangulation $\xT$
to the $\Autp\Dsox$-orbit of the corresponding heart $\hx$.
\end{proposition}
We expect that $\FG^\pm(\surfx) \xrightarrow{\cong} \EGp(\Dsox)/\MT$ as analogue of \cite[Thm.~2.10]{KQ2}
(=\cite[Thm.~5.6]{K2}), which requires further studies.
\begin{proof}
Fix a signed triangulation $\xT_0$ in $\FG^\pm(\surfx)$ and
make a spanning tree $\on{Tr}_0$ (which is a subgraph with the same vertex set and is also a tree) of $\FG^\pm(\surfx)$.

For each $\xT$ in $\FG^\pm(\surfx)$, let $p$ be the unique path in $\on{Tr}_0$ connecting $\xT$ and $\xT_0$.
By \Cref{pp:KY}, there is tilting sequence induced by $\on{Tr}_0$ so that we can associated a heart $X(\xT)=\hx$ to $\xT$.
Therefore, we obtain a graph $X(\on{Tr}_0)$ consisting of hearts in $\EGp(\Dsox)$ with simple tilting and
and a natural (graph) isomorphism $X\colon\on{Tr}_0\to X(\on{Tr}_0)$.

Next, we replace each heart $\h$ in $X(\on{Tr}_0)$ by its $\MT(\Dsox)$-orbit $[\h]$ and complete the graph
into $\overline{X(\on{Tr}_0)}$ by adding edges $[\h_1]\to[\h_2]$ whenever there is a simple forward tilting
from a representative of $[\h_1]$ to a representative of $[\h_2]$.
We claim that the isomorphism $X$ extend to $X\colon\FG^\pm(\surfx)\to \overline{X(\on{Tr}_0)}$.
By \Cref{pp:KY}, each edge in $X(\on{Tr}_0)$ will be doubled to be a 2-cycle that corresponds to the 2-cycle of $\FG^\pm(\surfx)$.
Moreover, there are (oriented) squares and pentagons in $\EGp(\Dsox)$ as in \cite[Lem.~2.12]{KQ2},
which correspond to squares and pentagons in $\FG^\pm(\surfx)$.
Hence, by \Cref{pp:FST}, we can partially complete $\on{Tr}_0$ to match all non-loops edges of $\FG^\pm(\surfx)$.
Furthermore, \Cref{pp:KY} implies that we also need to add the corresponding loops
(for those auto-equivalences induced by L-flips) in $\overline{X(\on{Tr}_0)}$.
Thus we have already arrived at an injection $X\colon\FG^\pm(\surfx)\to \overline{X(\on{Tr}_0)}$.
It is in fact an isomorphism since both graphs are $(n,n)$-regular, as claimed.

Finally, it is clear that $\MT(\Dsox)$ is a subgroup of $\Autp$ and hence $\overline{X(\on{Tr}_0)}$
is a covering of $\EGp(\Dsox)/\Autp$.
Namely, orbits of hearts in $\overline{X(\on{Tr}_0)}$ could be further identified under $\Autp$.
To show that $X$ induces the required isomorphism, we only need to show that
two hearts $\h_1,\h_2$ are related by $\Autp$ if and only if the corresponding triangulation $\RT_1$ and $\RT_2$ are related by an element in $\MCG(\surfx)$.
The former condition implies the associated QPs are right-equivalent (cf. \Cref{rem:KVB})
and hence implies the latter condition by \Cref{rem:iff}.
The inverse direction is straightforward.
\end{proof}

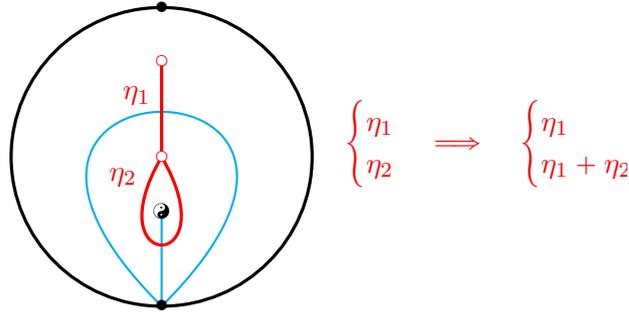
\begin{figure}[htb]\centering
\begin{tikzpicture}[scale=-1.8]\clip(1.2,1.3)rectangle(-3.4,-1.1);
\coordinate (p)at(0,0.5);
\draw[cyan,thick](p)to(0,1.2);
\draw[cyan,thick](0,1.2).. controls +(-45:2.7) and +(-135:2.7) .. (0,1.2);
\draw[red, very thick](0,-.6)to(0,.1).. controls +(120:1) and +(60:1) .. (0,.1) %closed arcs
    (0,-.35)node[left]{$\eta_1$} (.1,.25)node[left]{$\eta_2$};
\draw[very thick](0,0.1)circle(1.1)(0,1.2)\nn(0,-1)\nn
    (p)\Vpole(0,-.6)\ww(0,.1)\ww;
\draw[red](-2.5,0)node{$\begin{cases} \eta_1\\\eta_2 \end{cases} \Longrightarrow\quad
    \begin{cases} \eta_1\\ \eta_1+\eta_2 \end{cases}$};
\end{tikzpicture}
\caption{Modification of dual basis for self-folded triangle}\label{fig:last}
\end{figure}

%=========================================================
\subsection{Categorification}\
%=========================================================

As in \Cref{sec:cell}, $\FQuads(\surfx)$ admits the same style of stratification and cell structure.
Note that the cells will correspond to sign triangulations of $\surfx$.
Moreover, when there are self-folded with a vortex as isolated puncture,
one needs to change the local coordinate for the cell, cf. \cite[\S~10]{BS}.
More precisely, instead of using the periods $z_i$ of the actual saddle connections $\eta_i$ as shown in \Cref{fig:last}, one replace $z_2$ with $z_1+z_2$.
In this way, $z_2$ is allowed to be zero, which corresponds to the collision of the vortex with the simple zero in the self-folded triangle.

Then we have an analogue of \Cref{lem:f.d.1}, for the $\surfp$-framed version, i.e. \eqref{eq:Sembed2} and \eqref{eq:skel2}.
\begin{lemma}\label{lem:last}
There is a canonical embedding
\begin{equation}\label{eq:Sembed3}
    \skel_{\surfx}\colon\FG^\pm(\surfx)\to\FQuads(\surfx)
\end{equation}
whose image is dual to $B_2(\surfx)$ and which induces a surjective map
\begin{equation}\label{eq:skel3}
  \begin{tikzcd}
    \skel_*\colon \pi_1\FG^\pm(\surfx) \ar[r,twoheadrightarrow]& \pi_1\FQuads(\surfx).
  \end{tikzcd}
\end{equation}
\end{lemma}

Following the same proof of \cite[Thm.~11.2]{BS},
we can upgrade the graph isomorphism \eqref{eq:def-EG} into the following isomorphism between orbifolds.

\begin{theorem}\label{thm:app}
There is an isomorphism between complex orbifolds
\[
    \Quad(\sox)\cong\Stap(\Dsox)/\Autp
\]
\end{theorem}

%=========================================================
%=========================================================
\section{Examples: Euclidean Artin braid groups}\label{sec:ex}
%=========================================================
%=========================================================
\subsection{Artin braid groups and Coxeter groups}\label{sec:ABG}\
%=========================================================

Let $\ELam$ a diagram (without loops) with an (edge-) weight function $\mathbf{v}\colon\ELam_1\to\ZZ_{\ge3}$.

\begin{definition}
The Artin braid group $\Br_{\ELam}$ associated to $\ELam$ is defined by the presentation
\begin{gather}\label{eq:ABG}
    \Br(\ELam)\colon= \< b_i\mid i\in\ELam_0 \>\Big/
        (\Br^{m_{ij}}(b_i,b_j)\mid \forall i,j\in\ELam_0),
\end{gather}
where
\[
  m_{ij}=
    \begin{cases}
        2, & \mbox{if there is no edge beteween $i$ and $j$};\\
        \mathbf{v}(\alpha), & \mbox{if there is exactly one edge $i \frac{\alpha}{\quad} j$}; \\

        \infty,& \mbox{if there are multi edges between $i$ and $j$}.
    \end{cases}
\]
\end{definition}
The associated Coxeter/Weyl group $W_{\ELam}$ is the quotient group of $\Br(\ELam)$
by adding the square to one relation (i.e. $b_i^2=1$ for any $i\in\ELam_1$) to the presentation.
It is well-known that there are four infinity families of Euclidean Artin braid groups,
i.e. affine types $\wh{A},\wh{B},\wh{C}$ and $\wh{D}$.
%For completeness, we will also consider the spherical case for finite Dynkin diagrams.
In this section, we show that all of them are arise as fundamental groups of
moduli spaces of framed quadratic differentials studied above.
%The explicit calculation uses \cite{All} or the results of the twin paper \cite{QZ4}.

%=========================================================
\subsection{Hyperplane arrangements}\label{sec:hyper}\
%=========================================================

We restrict $\ELam$ to ba a Euclidean Dynkin diagram of $n+1$ vertices
and $\RootS$ be the corresponding finite Dynkin diagram.
Denote by $\Phi_{\RootS}$ the root system associated to $\RootS$, living in $\RR^n$, and
$\Phi_{\ELam}=\Phi_{\RootS}\oplus\ZZ\<\delta\>$ the root system associated to $\ELam$ living in $\RR^n\oplus\RR_\delta$,
where $\delta$ is the basic imaginary root.
Denote by $\fkhd^*=\fkh^*_\RootS\oplus\CC_\delta$ the complexification of $\RR^n\oplus\RR_\delta$ and $\fkhd\cong\CC^{n+1}$ its dual complex space.
Each root $(\alpha,k\delta)$ in $\Phi_{\ELam}\subset\fkhd^*$ determines an affine hyperplane
\[
    H_{\alpha,k}=\{\varsigma\in\fkhd\mid \<\varsigma,\alpha\>=k\}
\]
and let $\hreg_{\ELam}\subset\fkhd$ be the complement of the these affine hyperplanes, i.e.
\begin{equation}\label{eq:hreg}
    \hreg_{\ELam}\colon= \fkhd\setminus
    \big(\bigcup_{(\alpha,k\delta)\in\Phi_{\ELam}}H_{\alpha,k}\big).
\end{equation}
This is the space of hyperplane arrangements and
it is well-known that $W_{\ELam}$ can be geometrically realized as the reflection group on $\hreg_{\ELam}$,
where the generators $r_{s_{i}}$ are the reflections with respect to simple roots of $\ELam$. Moreover, $W_{\ELam}$ acts freely on $\hreg_{\ELam}$ and one has
$$
    \pi_1(\hreg_{\ELam}/W_{\ELam})=\Br(\ELam).
$$
A recent breakthrough of \cite{PS} proves the $K(\pi,1)$-conjecture for Euclidean braid groups.

\begin{theorem}\cite{PS}\label{thm:PS}
The universal cover of $\hreg_{\ELam}$ is contractible.
\end{theorem}

\def\coox{\mathbf{x}}
For the later use, we recall some facts from \cite{All}.
Let
\begin{gather}\label{eq:V0}
    V_0=\{\coox=(x_1,\ldots,x_n)\in\CC^n \mid x_j\ne\pm x_k, \forall j\ne k\}
\end{gather}
and
\begin{equation}\label{eq:V0CB}
  \begin{cases}
    V_{\wh{C_n}}=\{\coox\in V_0\mid \forall x_j\notin\frac{1}{2}\ZZ\},\\
    V_{\wh{B_n}}=\{\coox\in V_0\mid \forall x_j\notin\ZZ\},\\
    V_{\wh{D_n}}=V_0.
  \end{cases}
\end{equation}
Then there are homotopy equivalences $V_{\ELam}/W_{\ELam}\simeq\hreg_{\ELam}/W_{\ELam}$
for $\ELam=\wh{C_n},\wh{B_n}$ or $\wh{D_n}$.

\begin{remark}\label{rem:ELam}
Consider the following map
\[\begin{array}{rcccl}
    \iota \colon\CC &\xrightarrow{e^{2\pi\mathbf{i}(?)}}& \CC^* &\xrightarrow{\quad\eta\quad}& \CC^{0,1},\\
     x& \mapsto & y &\mapsto& \frac{1}{4}(1-y)(1-\frac{1}{y}),
\end{array}\]
where $\CC^{0,1}=\CC\setminus\{0,1\}$, $\eta$ is a branched double cover, branching at $0,1$ (whose preimages are $\pm1$).
It induces a homotopy equivalence
\begin{equation}\label{eq:V=C}
    V_{\wh{C_n}}/W_{\wh{C_n}}\simeq\on{conf}_n(\CC^{0,1}),
\end{equation}
where $\mathfrak{S}_n$ is permutation group of $n$ elements.

When passing from $V_{\wh{C_n}}$ to $V_{\wh{B_n}}$,
one adds hyperplanes
\begin{equation}\label{eq:hyper-B}
    H^1_{j,m}=\{\ x_j=\frac{1}{2}+m \}\cap V_0,\qquad m\in\ZZ,
\end{equation}
that corresponds to partial compactification
$\overline{\on{conf}_n(\CC^{0,1})}^{1}$ of $\on{conf}_n(\CC^{0,1})$
by allowing one of the points in the configuration goes to $1$.

When passing from $V_{\wh{B_n}}$ to $V_{\wh{D_n}}$,
one adds more hyperplanes
\begin{equation}\label{eq:hyper-C}
    H^0_{k,m}=\{\ x_k=m \}\cap V_0,\qquad m\in\ZZ,
\end{equation}
that corresponds to further partial compactification $\overline{\on{conf}_n(\CC^{0,1})}^{0,1}$
by allowing one of the points in the configuration goes to $0$ as well.
\end{remark}

%=========================================================
\subsection{Affine type $\wh{A_n}$}\label{sec:A}\
%=========================================================

Take a partition $n=p+q$ for some positive integers $p,q\in\ZZ$.
The diagram $\ELam=\wh{A_{p,q}}=\wh{A_{n}}$ is an $(n+1)$-cycle.
Let $\surfp$ is an annulus (with $\P=\emptyset$) with
$p$ and $q$ marked points in two boundaries of $\surf$ respectively.
For simplicity, assume that $p\ne q$. %(to get rid the extra $\ZZ_2$ symmetry in $\MCG(\surf)$).

The standard generators of $\MCG(\surf)$ are the clockwise rotations along one of the boundaries, moving a marked point to the next adjacent one, denoted by $r_0$ and $r_\infty$ respectively.
They satisfy $r_0^p=D=r_\infty^q$, where $D$ is the unique Dehn twist along either boundary of $\surf$, and fit into a short exact sequence (with an asymmetric choice between $p$ and $q$)
\[
    1\to\ZZ\<r_0\>\to\MCG(\surf)\to\ZZ_q\<r_\infty\>\to1.
\]
Moreover, $\Quad(\surf)$ consists of differentials of the form
\[
    \phi(z)=\frac{ \prod_{i=1}^{n+1} (z-z_i) }{z^{p+2}} \dd z^{2},\quad
    z_i\in\CC^*,\quad z_i\neq z_j
\]
and considered modulo the action of $\CC$ rescaling $z$.
%Let $\on{conf}_{n+1}(\CC^*)=\on{conf}_\Tri(\surf)$ be the (unordered) configuration space of $(n+1)$ points in $\CC^*$.
Then there is a commutative diagram
\begin{equation}\label{eq:3spaces}
\begin{tikzpicture}[xscale=1.6,yscale=0.8,baseline=(bb.base)]
\path (0,1) node (bb) {}; % baseline
\draw (0,2) node (s0) {$\FQuad(\surf)$}
 (0,0) node (s1) {$\on{conf}_{n+1}(\CC^*)$}
 (2.5,1) node (s2) {$\Quad(\surf)$.  };
\draw [-stealth, font=\scriptsize]
 (s0) edge node [left] {$\ZZ$} (s1)
 (s0) edge node [above] {$\MCG(\surf)$} (s2)
 (s1) edge node [below] {$\ZZ_q$} (s2);
\end{tikzpicture}
\end{equation}
%Noting that $\hreg_{\wh{A_{n}}}$ is the ordered configuration space $\widetilde{\on{conf}}_{n+1}(\CC^*)$, then One can deduce taht there is a homotopy equivalence $\FQuad(\surf)\simeq\hreg_{\wh{A_{n}}}/W_{\wh{A_{n}}}$.

By \cite[\S~8.3]{QQ},
the the contractible space
$\Stap(\D_3(\surfo))\cong\FQuad^{\T}(\surfo)$ is the universal cover of
\[
    \FQuad(\surf)=\FQuad(0;1^{p+q},-p-2,-q-2)\simeq\hreg_{\wh{A_{n}}}/W_{\wh{A_{n}}}
\]
with covering group
\begin{gather*}
    \pi_1\FQuad(\surf)=\BT(\surfo)=\Br(\wh{A_{p,q}}).
\end{gather*}
%Finally note that t
The short exact sequence \eqref{eq:SES-b} in \Cref{thm:QZ+} becomes
\begin{equation}\label{eq:BrA}
  \begin{tikzcd}[column sep=17,row sep=10]
    1 \ar[r] & \BT(\surfo) \ar[r] \ar[equal]{d} &
        \SBr(\surfo) \ar[equal]{d} \ar[r] & \Ho{1}(\surf) \ar[r]\ar[equal]{d} & 1\\
    1 \ar[r] & \pi_1\FQuad(\surf) \ar[r]  & \pi_1\on{conf}_{n+1}(\CC^*)
        \ar[r] & \ZZ \ar[r] & 1.
  \end{tikzcd}
\end{equation}

%=========================================================
\subsection{Affine type $\wh{C_n}$}\label{sec:C}\
%=========================================================

The diagram $\ELam=\wh{C_{n}}$ is shown in the right diagram in \Cref{fig:Ctilde} (for $n=5$).
Consider a DMSp $\sop$, which is an $(n-2)$-gon with two punctures $\sun=\{P_0,P_1\}$ and $n$ decorations.
\begin{figure}[hbt]\centering
\makebox[\textwidth][c]{
 \includegraphics[width=12cm]{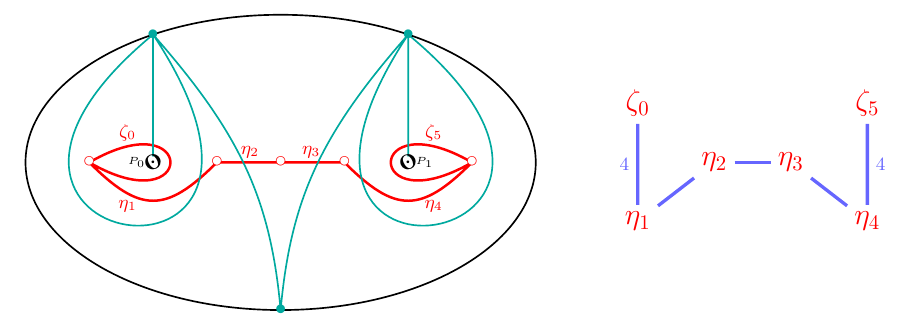}\;
}\caption{A (triangulated) DMSp $\sop$ for Artin braid group of type $\wh{C_{n=5}}$.}
\label{fig:Ctilde}
\end{figure}
Up to translation and scaling (that fixes two double poles to be 0 and 1),
$\FQuad(\surfp)$ consists of differentials of the form
\[
    \phi(z)={\lambda}\cdot\frac{ \prod_{i=1}^{n} (z-z_i)}{z^{2}(z-1)^{2}} \dd z^{2},
        \quad \lambda\in\CC^*, z_i\in\CC^{0,1},\quad z_i\neq z_j,
\]
with a $\ZZ$-framing (corresponding to the action of clockwise rotating the boundary of $\surfp$,
moving a marked point to the next adjacent one).
Hence, $\FQuad(\surfp)$ is a $\CC$-bundle over $\on{conf}_{n}(\CC^{0,1})$,
which is a homotopy equivalent to $\hreg_{\wh{C_n}}/W_{\wh{C_n}}$ by \eqref{eq:V=C}:
\begin{equation}\label{eq:homo}
    \FQuad(\surfp)\xrightarrow[/\CC]{\simeq}
        \on{conf}_{n}(\CC^{0,1})
        \simeq \hreg_{\wh{C_n}}/W_{\wh{C_n}}.
\end{equation}
The contractible space $\Stap(\D_3(\sop))\cong\FQuad^{\T}(\sop)$ is the universal cover of
$$
    \FQuad(\surfp)=\FQuad(0;1^{n},(-2)^2,-n),
$$
with covering group
\begin{gather*}
    \pi_1\FQuad(\surfp)=\MT(\sop)=\Br(\wh{C_{n}}).
\end{gather*}

%=========================================================
\subsection{Affine type $\wh{B_n}$}\label{sec:B}\
%=========================================================

The diagram $\ELam=\wh{B_{n}}$ is shown in the right diagram in \Cref{fig:Btilde} (for $n=5$).
Keep the DMSp $\sop$ as above in \Cref{sec:C}.
Consider the DMSx $\sox$ by taking the partition $\sun=\yue\cup\kui$ for $\yue=\{P_0\}$ and $\kui=\{P_1\}$.
\begin{figure}[hbt]\centering
\makebox[\textwidth][c]{
 \includegraphics[width=12cm]{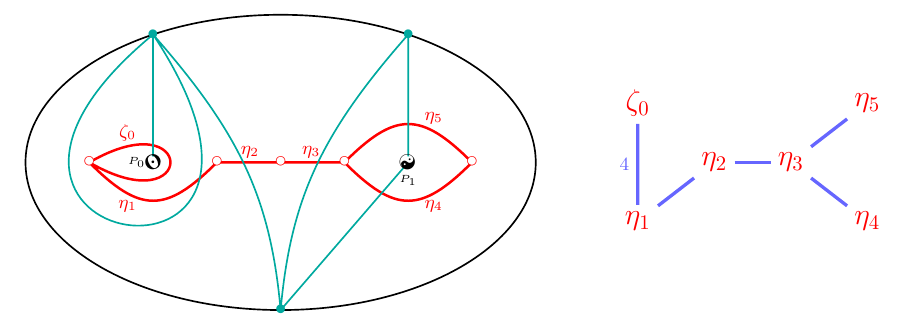}\;
}\caption{A (triangulated) DMSx $\sox$ for Artin braid group of type $\wh{B_{n=5}}$.}
\label{fig:Btilde}
\end{figure}
We have a moduli space
\begin{equation}\label{eq:Qaud-B}
\begin{array}{rl}
  \FQuad(\surfx)=&\FQuad(0;1^{n},(-2)^2,-n) \\
    \bigsqcup& \FQuad(0;1^{n-1},-1,-2,-n)^{\ZZ_2}
\end{array}
\end{equation}
with
\begin{gather}\label{eq:pi1=B+}
    \pi_1\FQuad(\surfx)=\MT(\sox)=\Br(\wh{C_{n}})/\LLk=\Br(\wh{B_{n}})\rtimes\ZZ_2.
\end{gather}
The notation $?^{\ZZ_2}$ means that there is a $\ZZ_2$-orbifolding for the extra strata.
Equivalently, \eqref{eq:pi1=B+} can be rewritten as
\begin{gather}\label{eq:pi1=B}
    \pi_1\FQuad^\pm(\surfx)=\MTsoyk=\Br(\wh{B_{n}}).
\end{gather}

By \Cref{rem:ELam}, $\FQuad(\surfx)$ is a $\CC$-bundle over
the partial compactification $\overline{\on{conf}_n(\CC^{0,1})}^{1}$,
that the extra strata corresponds to hyperplanes in \eqref{eq:hyper-B}.
These hyperplanes are fixed points under the corresponding reflections
and hence get $\ZZ_2$ orbifolding in the $\mathfrak{S}_n$-quotient,
i.e. we have
\begin{equation}\label{eq:homo-C}
    \FQuad(\surfx)\xrightarrow[/\CC]{\simeq}
        \overline{\on{conf}_n(\CC^{0,1})}^{1}
        \simeq \hreg_{\wh{B_n}}/W_{\wh{C_n}}.
\end{equation}

On the other hand, $\Br(\wh{C_{n}})=\MT(\sop)$ and
\[
    \Br(\wh{B_{n}})=\MTsoyk=\frac{\MTsoy}{ \LLk\cap\MTsoy }.
\]
Then the top right short exact sequence (in red) in \Cref{fig:3Octa} induces the following short exact sequence
\begin{equation}\label{eq:BinC}
  \begin{tikzcd}[column sep=17,row sep=20]
    & \MTsoyk \ar[twoheadrightarrow,"/?^2=1"]{d} &
        \MT(\sop) \ar[twoheadrightarrow,"/?^2=1"]{d} & \Ho{1}(\vortex) \ar[equal]{d} & \\
    1 \ar[r] & W_{\wh{B_n}} \ar[r]  & W_{\wh{C_n}}
        \ar[r] & \ZZ_2 \ar[r] & 1,
  \end{tikzcd}
\end{equation}
where $/?^2=1$ means taking quotient by (the normal subgroup generated by) squares of their standard generators.
Thus \eqref{eq:homo-C} is equivalent to
\[
    \FQuad(\surfx)\simeq
        \left(\hreg_{\wh{B_n}}/W_{\wh{B_n}}\right)\big/\ZZ_2
    \quad\Longleftrightarrow\quad
    \FQuad^\pm(\surfx)\simeq \hreg_{\wh{B_n}}/W_{\wh{B_n}}.
\]

%=========================================================
\subsection{Affine type $\wh{D_n}$}\label{sec:D}\
%=========================================================

The diagram $\ELam=\wh{D_{n}}$ is shown in the right diagram in \Cref{fig:Dtilde} (for $n=5$).
Keep the DMSp $\sop$ as above in \Cref{sec:C}.
Consider the DMSv $\sov$ by taking the partition $\sun=\yue\cup\kui$ for $\yue=\emptyset$ and $\kui=\sun$.
\begin{figure}[hbt]\centering
\makebox[\textwidth][c]{
 \includegraphics[width=12cm]{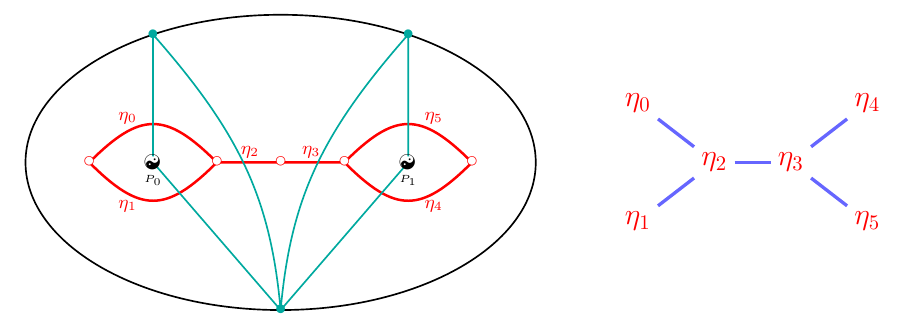}\;
}\caption{A (triangulated) DMSv $\sov$ for Artin braid group of type $\wh{D_{n=5}}$.}\label{fig:Dtilde}
\end{figure}
There is a moduli space
\begin{equation}\label{eq:Qaud-D}
\begin{array}{rl}
  \FQuad(\surfv)=&\FQuad(0;1^{n},(-2)^2,-n) \\
    \bigsqcup& \FQuad(0;1^{n-1},-1,-2,-n)^{\ZZ_2}\\
    \bigsqcup& \FQuad(0;1^{n-2},(-1)^2,-n)^{ \ZZ_2^{\oplus 2}}
\end{array}
\end{equation}
with
\begin{gather}\label{eq:pi1=D+}
    \pi_1\FQuad(\surfv)=\MT(\sov)=\Br(\wh{C_{n}})/\on{L}^2(\sun)
        =\Br(\wh{D_{n}})\rtimes{\ZZ_2 ^{\oplus 2}}.
\end{gather}
Equivalently, \eqref{eq:pi1=D+} can be rewritten as
\begin{gather}\label{eq:pi1=D}
    \pi_1\FQuad^\pm(\surfv)=\MT(\surf^{\emptyset,\sun})=\Br(\wh{D_{n}}).
\end{gather}
In this case, $\MT(\surf^{\emptyset,\sun})$ is also isomorphic to the cluster braid group of type $\wh{D_n}$,
in the sense of King-Qiu \cite{KQ2}.

Similar to type ${\wh{B_n}}$, we have
\begin{itemize}
  \item There is a homotopy equivalence
    \begin{equation}\label{eq:homo-D}
        \FQuad(\surfv)\xrightarrow[/\CC]{\simeq}
            \overline{\on{conf}_n(\CC^{0,1})}^{0,1}
            \simeq \hreg_{\wh{D_n}}/W_{\wh{C_n}}.
    \end{equation}
  \item The top right short exact sequence (in red) in \Cref{fig:3Octa} induces the following:
  \[\begin{tikzcd}[column sep=17,row sep=20]
    & 1 \ar[r] & W_{\wh{D_n}} \ar[r]  & W_{\wh{C_n}}
        \ar[r] & \ZZ_2^{\oplus2} \ar[r] & 1.
  \end{tikzcd}\]
  \item There are homotopy equivalences
\[
    \FQuad(\surfv)\simeq
        \left(\hreg_{\wh{D_n}}/W_{\wh{D_n}}\right)\big/\ZZ_2^{\oplus2}
    \quad\Longleftrightarrow\quad
    \FQuad^\pm(\surfv)\simeq \hreg_{\wh{D_n}}/W_{\wh{D_n}}.
\]
\end{itemize}

%\begin{remark}\label{rem:ABG}
%In summary, we see that the mixed twist groups of DMSp/DMSx/DMSv give a generalization of
%Euclidean Artin braid groups of type $\wh{C_n}$/$\wh{B_n}$/$\wh{D_n}$ (up to $\ZZv$).
%All of which can be realized as
%fundamental groups of moduli spaces of certain framed quadratic differentials.
%\end{remark}

%=========================================================
\subsection{Spherical Artin braid groups}\label{sec:sph}\
%=========================================================

We also list the surfaces for the four infinity families of spherical Artin braid groups $\RootS$.
Note that in this case the space of hyperplane arrangements is
\begin{equation}\label{eq:hreg0}
    \hreg_{\RootS}\colon= \fkh_\RootS\setminus
    \big(\bigcup_{\alpha\in\Phi_{\RootS}}H_{\alpha}\big)
\end{equation}
for $H_{\alpha}=\{\varsigma\in\fkh_{\RootS} \mid \<\varsigma,\alpha\>=0\}$.
\begin{itemize}
  \item For an $(n+3)$-gon $\surf$ with no punctures,
    one has that $\FQuad(\surf)\cong\hreg_{A_n}/W_{A_n}$ and
    $$\pi_1\FQuad(\surf)=\BT(\surfo)=\Br(A_n).$$
  \item For an $n$-gon $\surfp$ with one puncture,
    one has that $\FQuad(\surf)\cong\hreg_{B_n}/W_{B_n}$ and
    $$\pi_1\FQuad(\surfp)=\MT(\sop)=\Br(C_n).$$
  \item For an $n$-gon $\surfv$ with one vortex,
    one has that $\FQuad(\surf)\cong\hreg_{D_n}/W_{D_n}$ and
    $$\pi_1\FQuad^\pm(\surfv)=\MTsoyk=\Br(D_n).$$
\end{itemize}
All the statements in the affine cases also hold here.

%=========================================================
\subsection{Summary}\label{sec:sum}\
%=========================================================

Let $\DorE$ be a non-exceptional finite/affine Dynkin diagram
and $\sox$ be the corresponding DMSx mentioned above.
%(which maybe DMSp or DMSv when $\yue$ and/or $\kui$ is empty).
Combining all the discussion above,
we obtain the following geometric realization of hyperplane arrangements $\hreg_{\DorE}/W_{\DorE}$.
\begin{theorem}\label{thm:PS+}
There is a homotopy equivalence
\[
    \FQuad^\pm(\surfx)\simeq \hreg_{\DorE}/W_{\DorE}
\]
whose fundamental group fits into the following commutative diagram:
\begin{equation}\label{eq:finalSES}
  \begin{tikzcd}[column sep=17,row sep=10]
%    & \Br(\DorE) \ar[equal]{d}\\
    1 \ar[r] & \MTsoyk \ar[r] \ar[equal]{d} &
        \SBr(\surfx) \ar[equal]{d} \ar[r,"\AJ^\vortex"] & \Ho{1}(\surfv) \ar[r]\ar[equal]{d} & 1\\
    1 \ar[r] & \pi_1\FQuad^\pm(\surfx) \ar[r]  & \pi_1\on{conf}_{\Delta}(\surfx)
        \ar[r] & \Ho{1}(\surf)\oplus\ZZv \ar[r] & 1 ,
  \end{tikzcd}
\end{equation}
where $\AJ^\vortex$ is the vortex version of the AJ map.
In the `lattice' case (i.e. finite Dynkin, $\wh{A}_n$ or $\wh{C}_n$, cf. \cite{McC}),
a contractible universal cover of $\hreg_{\DorE}/W_{\DorE}$ is given by $\Stap(D_3(\sox))$.
\end{theorem}
\begin{proof}
As $$\Br(\DorE)=\pi_1(\hreg_{\DorE}/W_{\DorE})=\pi_1\FQuad(\surfx)=\MTsoyk,$$
the top short exact sequence is precisely \eqref{eq:ses H1-xx}.

Then we only need to clarify the last statement:
\begin{itemize}
  \item The finite Dynkin case follows from \cite[Thm.~B]{QW}.
  \item The $\wh{A}_n$ and $\wh{C}_n$ case follows from \Cref{thm:2} and
  the contractibility of \Cref{thm:PS} (or \cite{Oko} in these two cases).
\end{itemize}
\end{proof}

%=========================================================

%=========================================================
\end{document}